\documentclass[11pt,letterpaper,reqno]{amsart}

\usepackage{amsmath,amsthm,amsfonts,amssymb}
\usepackage{bm} % bold math symbols

\usepackage[T1]{fontenc}
\usepackage[utf8]{inputenc}
\usepackage{fourier} % font package

\usepackage[margin=1in]{geometry}
\numberwithin{equation}{section}

\usepackage{array, boldline, makecell, booktabs}

\usepackage{tikz}
\usetikzlibrary{backgrounds}
\usepackage{graphicx,import}

\usepackage{ytableau, young}
\usepackage{euscript, mathrsfs}

\usepackage{xypic}
\usepackage{amscd}
\usepackage{cases}
\usepackage{color}
\usepackage{enumerate}
\usepackage{mathtools}

\usepackage[colorlinks]{hyperref}
\usepackage{cleveref}
\usepackage[style=alphabetic,maxnames=99,maxalphanames=6]{biblatex}
\newtheorem{thm}{Theorem}[section]
\newtheorem{lem}[thm]{Lemma}
\newtheorem{proposition}[thm]{Proposition}
\newtheorem{corollary}[thm]{Corollary}
\theoremstyle{definition}
\newtheorem{example}[thm]{Example}
\newtheorem{definition}[thm]{Definition}
\newtheorem{conj}[thm]{Conjecture}
\newtheorem{rmk}[thm]{Remark}

\newtheorem{question}[thm]{Question}

\newcommand{\bmlambda}{{\bm{\lambda}}}
\newcommand{\bmmu}{{\bm{\mu}}}
\newcommand{\bmnu}{{\bm{\nu}}}
\newcommand{\bmT}{{\bm{T}}}
\newcommand{\bmR}{{\bm{R}}}

\newcommand{\QQ}{{\mathbb {Q}}}

\DeclareMathOperator{\Set}{Set}

\DeclareMathOperator{\dd}{dd}

\DeclareMathOperator{\QSym}{QSym}
\DeclareMathOperator{\Frob}{Frob}
\DeclareMathOperator{\Inv}{Inv}

\DeclareMathOperator{\SSS}{SS}

\DeclareMathOperator{\iDes}{iDes}
\DeclareMathOperator{\Des}{Des}
\DeclareMathOperator{\dst}{dst}

\DeclareMathOperator{\maj}{maj}

\DeclareMathOperator{\arm}{arm}
\DeclareMathOperator{\leg}{leg}

\DeclareMathOperator{\inv}{inv}
\DeclareMathOperator{\dinv}{dinv}
\DeclareMathOperator{\I}{I}
\DeclareMathOperator{\des}{Des}
\DeclareMathOperator{\ssyt}{SSYT}
\DeclareMathOperator{\syt}{SYT}
\DeclareMathOperator{\llt}{LLT}
\DeclareMathOperator{\st}{st}
\DeclareMathOperator{\std}{std}
\DeclareMathOperator{\rw}{rw}
\DeclareMathOperator{\cycling}{cyc}

\newlength\cellsize 

\savebox2{%
\begin{picture}(12,12)
\put(0,0){\line(1,0){12}}
\put(0,0){\line(0,1){12}}
\put(12,0){\line(0,1){12}}
\put(0,12){\line(1,0){12}}
\end{picture}}

\newcommand\cellify[1]{\def\thearg{#1}\def\nothing{}%
\ifx\thearg\nothing\vrule width0pt height\cellsize depth0pt%
  \else\hbox to 0pt{\usebox2\hss}\fi%
  \vbox to 12\unitlength{\vss\hbox to 12\unitlength{\hss$#1$\hss}\vss}}

\newcommand\tableau[1]{\vtop{\let\\=\cr
\setlength\baselineskip{-12000pt}
\setlength\lineskiplimit{12000pt}
\setlength\lineskip{0pt}
\halign{&\cellify{##}\cr#1\crcr}}}

\def\diag{ \begin{tikzpicture} \draw[dashed] (-.12,-.12) -- (.42, .42); \end{tikzpicture} }

\newcommand{\stat}{\operatorname{stat}}

\usetikzlibrary{arrows.meta} % 화살표 메타 라이브러리 로드

\tikzset{
  >={Stealth[length=2mm, width=1.5mm]}
}

\newcommand\Matching[2]{%
  \draw [->, line width=0.3mm] 
    (#1,0) to [out=90, in=90, looseness=1.5] (#2,0); % <--- 1.5 값을 조절하세요
}
\title{Toward Butler's conjecture}

\author{Donghyun Kim}
\address{Department of Mathematical Sciences \\ Seoul National University \\Seoul 08826 \\ Republic of Korea}
\email{hyun920310@snu.ac.kr}

\author{Seung Jin Lee}
\address{Department of Mathematical Sciences \\ Research institute of Mathematics \\ Seoul National University \\ Seoul 08826 \\ Republic of Korea}
\email{lsjin@snu.ac.kr}

\author{Jaeseong Oh}
\address{Department of Mathematics \\ Sungkyunkwan University \\ Suwon \\ Korea}
\email{jaeseongoh@skku.edu}

\begin{document}
\maketitle
\begin{abstract}
The celebrated Haglund–Haiman–Loehr (HHL) formula provides an explicit monomial expansion of the Macdonald polynomials. In 1994, Butler introduced a refinement $\operatorname{I}_{\lambda,\mu}[X;q,t] = (T_\lambda\,\widetilde{H}_\mu[X;q,t] - T_\mu\,\widetilde{H}_\lambda[X;q,t]) / (T_\lambda - T_\mu)$ of the Macdonald polynomial and conjectured its Schur positivity. According to the Science Fiction conjecture by Bergeron and Garsia, this refinement represents the `intersection' of Macdonald polynomials. 

In this work, we introduce a novel combinatorial tool, the \emph{column exchange rule}, which enables us to derive a positive monomial expansion for Butler's symmetric function \(\text{I}_{\lambda,\mu}\), thereby refining the HHL formula. Additionally, we prove Butler's conjecture on the Schur positivity of \(\text{I}_{\lambda,\mu}\) in specific cases.

\end{abstract}

\textit{2020 Mathematics Subject Classification [MSC] codes:} 05E05, 05E10.

\section{Introduction}
\subsection{Overview}
In his seminal paper \cite{Mac88}, Macdonald introduced the \emph{Macdonald} \emph{$P$ polynomials} $P_\mu[X;q,t]$ which are $(q,t)$-extensions of Schur functions indexed by partitions $\mu$.  Macdonald polynomials specialize to important families of symmetric functions such as Jack symmetric functions and Hall--Littlewood polynomials. They have been widely researched and have numerous applications in representation theory, algebraic geometry, and mathematical physics, among others. Later, Garsia and Haiman introduced a variant of Macdonald polynomials called the \emph{modified Macdonald polynomial}, which is defined by the plethystic substitution of the \emph{Macdonald $J$ polynomial} of \cite{Mac88}:
\[
    \widetilde{H}_\mu[X;q,t] = t^{n(\mu)}J_\mu\left[\dfrac{X}{1-t^{-1}};q,t^{-1}\right].
\]
Since their appearance, Schur positivity of the modified Macdonald polynomials has been an active area of research. One of the possible attempts to prove Schur positivity was to construct an $\mathfrak{S}_n$-module whose Frobenius characteristic is the modified Macdonald polynomial.

In 1993, Garsia and Haiman introduced the $\mathfrak{S}_n$-module $V_\mu$, now called the \emph{Garsia--Haiman module}, for a partition $\mu\vdash n$. It is defined as the subspace of
\[
\mathbb{C}[x_1,\dots,x_n,y_1,\dots,y_n]
\]
spanned by the partial derivatives of the polynomial $\Delta_\mu$, which is analogous to the Vandermonde determinant:
\[
    \Delta_\mu:=\det\left[x_i^{p_j}y_i^{q_j}\right]_{i,j=1,\dots,n},
\]
where $c_j=(p_j,q_j)$ runs over the cells in $\mu$, and the symmetric group $\mathfrak{S}_n$ acts diagonally by permuting the $x$ and $y$ variables \cite{GH93}. They conjectured that the dimension of $V_{\mu}$ is $n!$, a statement known as the \emph{$n!$ conjecture}. This conjecture implies that the Frobenius characteristic of $V_{\mu}$ equals $\widetilde{H}_\mu[X;q,t]$.

Since the conjecture was proposed, there have been many efforts to understand the structure of the Garsia--Haiman module \cite{Ste94, GH96, All97, BG99, BBGHT99}. Later, Haiman \cite{Hai01} ultimately proved the $n!$ conjecture by utilizing the geometry of the isospectral Hilbert scheme of $n$ points in the plane, thereby confirming the Macdonald positivity conjecture. However, his proof does not provide a combinatorial formula for the Schur coefficients, the \emph{$(q,t)$--Kostka polynomials} $\widetilde{K}_{\lambda,\mu}(q,t)$, which remains an outstanding open problem.

Before Haiman's proof of the Macdonald positivity, Bergeron and Garsia \cite{BG99} studied relationships among Garsia--Haiman modules called the \emph{Science Fiction conjecture}, stating the existence of bases of Garsia--Haiman modules and their intersections with elegant properties. This conjecture has diverse implications, not only on Garsia--Haiman modules but also on modified Macdonald polynomials and the Frobenius characteristic of intersections of the Garsia--Haiman modules. One of the implications of the Science Fiction conjecture generalizing the $n!$ theorem, is the following:
\begin{conj}
Let $\nu$ be a partition of $n$ and $\mu^{(1)},\dots,\mu^{(k)}\subseteq \nu$ be $k$ distinct partitions such that $|\nu/\mu^{(1)}|=\cdots=|\nu/\mu^{(k)}|=1.$
Then the bigraded $\mathfrak{S}_n$-module $\bigcap_{i=1}^{k} V_{\mu^{(i)}}$ is of dimension $\frac{n!}{k}$, and its Frobenius characteristic is 
\begin{equation}\label{eq: Frob = SF}
    \Frob\left(\bigcap_{i=1}^{k} V_{\mu^{(i)}};q,t\right) = \sum_{i=1}^{k} \prod_{j\neq i}\dfrac{T_{\mu^{(j)}}}{(T_{\mu^{(j)}}-T_{\mu^{(i)}})}\widetilde{H}_{\mu^{(i)}}[X;q,t],
\end{equation}
where $T_\mu:=\prod_{(i,j)\in\mu}t^{i-1}q^{j-1}$.
\end{conj}

The first assertion is called the \emph{$\frac{n!}{k}$ conjecture}, and Armon recently proved the $\frac{n!}{2}$ conjecture for hook shapes \cite{Arm22}. The second implication of the Science Fiction conjecture gives a formula for the Frobenius characteristic of the intersection of the Garsia--Haiman modules as a linear combination of the modified Macdonald polynomials. We define the \emph{Macdonald Intersection polynomial} to be the symmetric function on the right-hand side of \eqref{eq: Frob = SF}, and it is denoted by $\I_{\mu^{(1)},\dots,\mu^{(k)}}[X;q,t]$. Note that $\I_{\mu^{(1)},\dots,\mu^{(k)}}[X;q,t]$ is independent of the order of $\mu^{(1)},\dots,\,\mu^{(k)}$, so we usually assume $\mu^{(1)}\trianglerighteq\dots\trianglerighteq\mu^{(k)}$ in the dominance order. In the companion paper of the authors \cite{KLO22+}, we study a remarkable connection between $\I_{\mu^{(1)},\dots,\mu^{(k)}}[X;q,t]$ and $\nabla e_{k-1}$, which is the Frobenius characteristic of the diagonal harmonics \cite{Hai02}.

In this paper, we focus more on the case when $k=2$, which is related to Butler's conjecture. In 1994, Butler observed a surprising behavior of the modified Macdonald polynomials.
\begin{conj}(Butler's conjecture \cite{But94})
Let $\nu$ be a partition and $\lambda,\mu\subseteq \nu$ be two distinct partitions such that $|\nu/\lambda|=|\nu/\mu|=1$. Then the Macdonald intersection polynomial $\I_{\lambda,\mu}[X;q,t]=\frac{T_\lambda\widetilde{H}_\mu[X;q,t]-T_\mu\widetilde{H}_\lambda[X;q,t]}{T_\lambda-T_\mu}$ is Schur positive.
\end{conj}
Note that the Science Fiction conjecture implies Butler's conjecture, as mentioned in \cite{BG99}. More precisely, if the Science Fiction conjecture is true, the Macdonald intersection polynomial \(\I_{\lambda,\mu}[X;q,t]\) is given as the Frobenius characteristic of an \(\mathfrak{S}_n\)-module \(V_\lambda \cap V_\mu\), and is therefore Schur positive. This paper makes significant progress toward Butler's conjecture. We divide our main results into three parts.

\subsection{Fundamental expansion} The first main result is to give a combinatorial formula for the Macdonald intersection polynomial $\I_{\lambda,\mu}[X;q,t]$ as a sum of fundamental quasisymmetric functions over the \emph{Butler permutations} $\mathfrak{B}_{\lambda,\mu}$ (see Definition~\ref{def: butler permutation for two partitions}). Note that Theorem \ref{thm: first main, F-expansion} naturally gives a positive monomial expansion formula for $\I_{\lambda,\mu}[X;q,t]$ (Corollary \ref{Cor: m-expansion}).

\begin{thm}\label{thm: first main, F-expansion} 
Let $\nu$ be a partition, and $\lambda,\mu \subseteq \nu$ be two distinct partitions such that $|\nu/\lambda|=|\nu/\mu|=1$. Then there is a statistic $\stat_{\lambda,\mu}$, which is a monomial in $q,t$ (see Definition~\ref{def: butler permutation for two partitions}), defined over a set of certain permutations $\mathfrak{B}_{\lambda,\mu}$ such that the fundamental quasisymmetric expansion of $\I_{\lambda,\mu}[X;q,t]$ is given by
\begin{equation*}
    \I_{\lambda,\mu}[X;q,t] = \sum_{w \in \mathfrak{B}_{\lambda,\mu}} \stat_{\lambda,\mu}(w)F_{\iDes(w)}[X].
\end{equation*}
Here, $F_S$ denotes a fundamental quasisymmetric function. 
\end{thm}

Our result on $\I_{\lambda,\mu}[X;q,t]$ refines the following foundational result: the combinatorial formula for the modified Macdonald polynomials given by Haglund, Haiman, and Loehr. The undefined terms are defined in Section~\ref{Sec: preliminaries} and Section~\ref{Sec: generalization of Macdonald}. 
\begin{thm}\label{thm: HHL formula}\cite{HHL05}
For a partition $\mu$, there is a statistic $\stat_{(\mu,f^{\st}_\mu)}$ such that the fundamental quasisymmetric expansion of $\widetilde{H}_\mu[X;q,t]$ is given by
\begin{equation*}
    \widetilde{H}_\mu[X;q,t] = \sum_{w\in \mathfrak{S}_n} \stat_{(\mu,f^{\st}_{\mu})}(w) F_{\iDes(w)}[X].
\end{equation*}
\end{thm}

To prove Theorem~\ref{thm: first main, F-expansion}, we introduce intricate relations between two (generalized) modified Macdonald polynomials: the column exchange rule (Proposition~\ref{lem: column exchange}) and Proposition~\ref{lem: main LLT lemma}. As a byproduct, we obtain a (positive) monomial symmetric function expansion for $\I_{\lambda,\mu}[X;q,t]$ defined over Butler words (Corollary~\ref{Cor: m-expansion}). It also follows that $\I_{\lambda,\mu}[X;1,1]=h_{(2,1^{n-2})}[X]$, which depends only on the size $|\lambda|=|\mu|=n$ (Corollary~\ref{Cor: h21111 at q=t=1}). Here, $h_\lambda$ denotes the complete homogeneous symmetric function. This is reminiscent of the fact that for any partition $\mu\vdash n$, $\widetilde{H}_\mu[X;1,1]=h_{(1^n)}[X]$. This corollary is also consistent with the $\frac{n!}{2}$ conjecture.

\subsection{LLT polynomials and Schur positivity}
While studying quantum affine algebras and unipotent varieties, Lascoux, Leclerc, and Thibon introduced a $q$-analogue of a product of skew Schur functions, now called the \emph{LLT polynomials} \cite{LLT97}. Lascoux, Leclerc, and Thibon conjectured, and Grojnowski and Haiman proved, the Schur positivity of LLT polynomials \cite{GH07}. In addition, Theorem~\ref{thm: HHL formula} can be translated into a (positive) LLT-expansion (indexed by tuples of ribbons) of the modified Macdonald polynomials. Therefore, one can study the linear combinations of LLT polynomials to tackle the Schur positivity of $\I_{\lambda,\mu}[X;q,t]$.

Recently, relations between linear combinations of LLT polynomials (called the LLT equivalence, see Section~\ref{Sec: Proof of Schur positivities}) have been extensively studied in \cite{Lee21, HNY20, Mil19, AS22, Tom21}. After applying the column exchange rule and a series of LLT equivalences, we obtain the following theorem:

\begin{thm}\label{thm: second main, s-positivities} 
Let $\nu$ be a partition and $\lambda,\mu \subseteq \nu$ be two distinct partitions such that $|\nu/\lambda| = |\nu/\mu| = 1$. Suppose that the cell $\nu/\lambda$ is in the first or second row of the partition $\lambda$. Then $\I_{\lambda,\mu}[X;q,t]$ has a positive expansion in LLT polynomials. In particular, $\I_{\lambda,\mu}[X;q,t]$ is Schur positive.
\end{thm}

As a corollary, we obtain Schur positivity results for $\I_{\lambda,\mu}[X;q,t]$ at $t = 1$ or $q = 1$ (Corollary~\ref{Cor: Schur positivity at t=1}).

\subsection{Combinatorial formula for Schur coefficients}

Macdonald \cite{Mac88} conjectured, and Haiman proved \cite{Hai01}, that the Schur coefficients of the modified Macdonald polynomials are, in fact, polynomials in $q$ and $t$ with nonnegative integer coefficients. These Schur coefficients are called the modified $(q,t)$-Kostka polynomials. Haiman's proof relies heavily on the geometry of the isospectral Hilbert scheme of $n$ points in the plane and does not provide a combinatorial interpretation for the modified $(q,t)$-Kostka polynomials.

We hope that Theorem~\ref{thm: second main, s-positivities} sheds light on providing a combinatorial formula for the modified $(q,t)$-Kostka polynomials. In particular, we describe the Schur expansion of $\I_{\mu,\nu}[X;q,t]$ for two neighboring hook shapes $\mu$ and $\nu$ (Corollary~\ref{cor: hook}), building upon Assaf's formula \cite{Ass18}..

\subsection{Sketch of the proof of Theorem~\ref{thm: first main, F-expansion} and Theorem~\ref{thm: second main, s-positivities}}
We first define the (generalized) modified Macdonald polynomials for filled diagrams in Section~\ref{Sec: generalization of Macdonald}. Then, we introduce two identities between modified Macdonald polynomials. The first one is the \emph{cycling identity} (Lemma~\ref{lem: cycling}, Figure~\ref{fig: cycling}), which follows directly from the definition. The second one is the \emph{column exchange rule} (Proposition~\ref{lem: column exchange}, Figure~\ref{fig: generic filled diagrams column exchange lemma}), which is essential in the proofs of Theorem~\ref{thm: first main, F-expansion} and Theorem~\ref{thm: second main, s-positivities}. By applying these identities to the modified Macdonald polynomials $\widetilde{H}_\mu[X;q,t]$ and $\widetilde{H}_\lambda[X;q,t]$, we reduce Theorem~\ref{thm: first main, F-expansion} to Proposition~\ref{lem: main LLT lemma} and Theorem~\ref{thm: second main, s-positivities} to Lemma~\ref{lem: s positivity reduction}.

\subsection{Organization}
This paper is organized as follows. We begin with some preliminaries in Section~\ref{Sec: preliminaries}. In Section~\ref{Sec: generalization of Macdonald}, we define a generalization of the modified Macdonald polynomials indexed by general shapes and statistics, called a filled diagram. In Section~\ref{Sec: Butler permutations, main Lemmas}, we introduce and prove the column exchange rule, which plays a central role in our work. We also define Butler permutations and prove an $F$-expansion formula for the two-column diagram case (Proposition~\ref{lem: main LLT lemma}). Section~\ref{Sec: proof of main theorem} contains a proof of Theorem~\ref{thm: first main, F-expansion}. In Section~\ref{Sec: Proof of Schur positivities}, we provide background on LLT polynomials and their connection to the modified Macdonald polynomials. Then, using several LLT equivalences, we prove Theorem~\ref{thm: second main, s-positivities}. In Section~\ref{Sec: Schur expansions of Macdonald polynomials}, we discuss two combinatorial formulas, which are consistent with Butler's conjecture. In the final section, we pose some open questions for future research.

\section{Preliminaries}\label{Sec: preliminaries}
\subsection{Partitions and tableaux}
A \emph{partition} is a nonincreasing sequence 
\[
    \mu = (\mu_1, \mu_2, \dots, \mu_\ell)
\]
of positive integers (parts), and its \emph{size} is defined as $\mu_1 + \mu_2 + \cdots + \mu_\ell$. We write $\mu \vdash n$ to denote that $\mu$ is a partition of size $n$. For a partition $\mu = (\mu_1, \mu_2, \dots, \mu_\ell)$, we abuse notation and write
\[
    \mu = \{(i,j) \in \mathbb{Z}_+ \times \mathbb{Z}_+ : 1 \leq i \leq \ell, 1 \leq j \leq \mu_i\}
\]
to denote its \emph{Young diagram}, whose elements are called \emph{cells}. We draw the Young diagram in the first quadrant, using French notation; see the left side of Figure~\ref{Fig: Young diagram arm and leg} for an example. The \emph{conjugate} partition $\mu' = (\mu'_1, \mu'_2, \dots)$ is the partition obtained by reflecting the Young diagram of $\mu$ along the diagonal $y = x$.

For a cell $u$ in a partition $\mu$, we write $u = (i,j)$, where $i$ is the row index and $j$ is the column index. The \emph{arm} of a cell, denoted by $\arm_\mu(u)$, is the number of cells strictly to the right of $u$ in the same row. Its \emph{leg}, denoted by $\leg_\mu(u)$, is the number of cells strictly above $u$ in the same column. For example, the red cell $u$ on the right side of Figure~\ref{Fig: Young diagram arm and leg} is denoted by $(2,1)$. It has three cells (denoted by $a$) strictly to the right of it in the same row, which gives $\arm_{\mu}(u) = 3$. By similar reasoning, we have $\leg_{\mu}(u) = 2$.

 \begin{figure}[h]
    \centering
    \[
        \begin{ytableau}
      \\
    &  &  \\
     &  &  & \\
    &   &  &  &
    \end{ytableau}\qquad\qquad\qquad
    \begin{ytableau}
     \ell \\
     \ell&  &  \\
    *(red) u  & a& a & a\\
     &  &  &  & 
    \end{ytableau}
    \]
    \caption{The left figure shows the Young diagram in the French notation for a partition $(5,4,3,1)$. The right figure shows the computation of the arm and the leg for the red cell $u$.}
    \label{Fig: Young diagram arm and leg}
\end{figure}

There is a partial order $\trianglerighteq$ called the \emph{dominance order} of partitions of $n$ which is defined by 
\begin{equation*}
    \lambda\trianglerighteq\mu \text{ if } \lambda_1+\cdots+\lambda_k\geq\mu_1+\cdots+\mu_k \text{ for all } k.
\end{equation*}
For partitions $\lambda$ and $\mu$ with $\mu\subseteq\lambda$, a \emph{skew partition} is a subset of $\mathbb{Z}_+\times\mathbb{Z}_+$ of the form $\lambda/\mu$.

For a skew partition $\nu$, a \emph{semistandard tableau} of shape $\nu$ is a filling of $\nu$ with positive integers where each row is weakly increasing from left to right and each column is strictly increasing from bottom to top. For a tuple $\bmnu=(\nu^{(1)},\nu^{(2)},\dots)$ of skew partitions, a semistandard tableau $\bmT=(T^{(1)},T^{(2)},\dots)$ of shape $\bmnu$ is a tuple of semistandard tableaux where each $T^{(i)}$ is a semistandard tableau of shape $\nu^{(i)}$. The set of semistandard tableaux of shape $\bmnu$ is denoted by $\ssyt(\bmnu)$. A semistandard tableau of size $n$ is called \emph{standard} if its filling consists of $1,2,\dots,n$. We denote the set of standard tableaux of shape $\bmnu$ by $\syt(\bmnu)$. 

\subsection{Symmetric functions}
We denote by $\Lambda = \bigoplus_{d \geq 0} \Lambda_d$ the graded ring of symmetric functions in an infinite set of variables $X = x_1, x_2, \dots$ over the ground field $\mathbb{Q}(q,t)$. Here, $\Lambda_d$ denotes the subspace of $\Lambda$ consisting of homogeneous symmetric functions of degree $d$. For a partition $\lambda \vdash n$, the \emph{monomial symmetric function} $m_\lambda[X] \in \Lambda_n$ is defined by
\begin{equation*}
    m_\lambda[X] := \sum_{\alpha} x^\alpha,
\end{equation*}
where $x^\alpha = x_1^{\alpha_1} x_2^{\alpha_2} \cdots$, and the sum is over all $\alpha$ such that the parts of $\alpha$ are a rearrangement of the parts of $\lambda$, by appending zeros if necessary.

The {\em homogeneous} and {\em elementary} symmetric functions associated to a partition $\lambda$ are defined by
\begin{equation*}
h_\lambda[X] := h_{\lambda_1} h_{\lambda_2} \cdots, \quad \quad \text{and} \quad \quad e_\lambda[X] := e_{\lambda_1} e_{\lambda_2} \cdots,
\end{equation*}
where 
\begin{equation*}
h_d[X] := \sum_{i_1 \leq \cdots \leq i_d} x_{i_1} \cdots x_{i_d}, \quad \quad \text{and} \quad \quad
e_d[X] := \sum_{i_1 < \cdots < i_d} x_{i_1} \cdots x_{i_d}.
\end{equation*}

The most prevalent symmetric function is the \emph{Schur function}. For a partition $\lambda \vdash n$, we define the Schur function $s_\lambda[X]$ by
\begin{equation*}
    s_\lambda[X] := \sum_{T \in \ssyt(\lambda)} x^T
\end{equation*}
where $x^T = x_1^{T_1} x_2^{T_2} \cdots$. Here $T_i$ is the number of $i$'s in the semistandard tableau $T$. Let $\langle -,-\rangle$ be an inner product on symmetric functions such that
\begin{equation*}
    \langle s_\lambda, s_\mu \rangle = \delta_{\lambda,\mu},
\end{equation*}
thus implying that Schur functions form an orthonormal basis. We call this inner product the \emph{Hall inner product}.

The \emph{modified Macdonald polynomials} are defined by the unique family of symmetric functions satisfying the following \emph{triangulation} and \emph{normalization} axioms \cite[Proposition 2.6]{Hai99}:
\begin{enumerate}
    \item $\widetilde{H}_\mu[X(1-q);q,t] = \sum_{\lambda \trianglerighteq \mu} a_{\lambda,\mu}(q,t) s_\lambda[X]$,
    \item $\widetilde{H}_\mu[X(1-t);q,t] = \sum_{\lambda \trianglerighteq \mu'} b_{\lambda,\mu}(q,t) s_\lambda[X]$, and
    \item $\langle \widetilde{H}_\mu[X;q,t], s_{(n)}[X] \rangle = 1$,
\end{enumerate}
for suitable coefficients $a_{\lambda,\mu}, b_{\lambda,\mu} \in \QQ(q,t)$. Here, $[-]$ denotes the plethystic substitution, which can be simply understood in terms of the power sum symmetric function $p_k[X]$. We have $p_k[X(1-q)] = (1-q^k) p_k[X]$ and $p_k[X(1-t)] = (1-t^k) p_k[X]$. As $p_k[X]$'s are algebra generators of $\Lambda$, we can compute $f[X(1-q)]$ or $f[X(1-t)]$ for any symmetric function $f \in \Lambda$.

Each of these families of symmetric functions — monomial symmetric functions, complete homogeneous symmetric functions, elementary symmetric functions, Schur functions, and modified Macdonald polynomials — forms a basis for $\Lambda$.

\subsection{Quasisymmetric functions}
A function $f$ is \emph{quasisymmetric} if for every composition $\alpha$ and any $0<i_1<i_2<\cdots$, the coefficient of $x_1^{\alpha_1}x_2^{\alpha_2}\cdots$ is the same as the coefficient of $x_{i_1}^{\alpha_1}x_{i_2}^{\alpha_2}\cdots$. We denote the graded ring of quasisymmetric functions in an infinite variable set $X=x_1,x_2,\dots$ over the ground field $\mathbb{Q}(q,t)$ by $\QSym=\bigoplus_{d\ge0}\QSym_d$. Here, $\QSym_d$ is the subspace of $\QSym$ consisting of homogeneous quasisymmetric functions of degree $d$. 

Quasisymmetric functions of degree $n$ are usually indexed by subsets of $[n-1] := \{1, 2, \dots, n-1\}$. For a subset $S = \{s_1 < s_2 < \dots < s_\ell\} \subseteq [n-1]$, the \emph{monomial quasisymmetric function} $M_{n,S}[X]$ is defined by 
\begin{equation*}
    M_{n,S}[X] := \sum_{i_1 < i_2 < \dots < i_{\ell+1}} x_{i_1}^{s_1} x_{i_2}^{s_2 - s_1} \cdots x_{i_\ell}^{s_\ell - s_{\ell-1}} x_{i_{\ell+1}}^{n - s_{\ell}}.
\end{equation*}

The most important family of quasisymmetric function is Gessel's \emph{fundamental quasisymmetric functions} \cite{Ges84}. For $S \subseteq [n-1]$, the fundamental quasisymmetric function $F_{n,S}[X]$ is defined by
\begin{equation*}
    F_{n,S}[X] := \sum_{\substack{1 \leq b_1 \leq b_2 \leq \cdots \leq b_n \\ i \in S \Rightarrow b_i < b_{i+1}}} x_{b_1} x_{b_2} \cdots x_{b_n}.
\end{equation*}
In other words, in terms of monomial quasisymmetric functions, we have
\begin{equation*}
    F_{n,S}[X] = \sum_{S \subseteq T \subseteq [n-1]} M_{n,T}[X].
\end{equation*}
If $n$ is clear from the context, we denote the fundamental quasisymmetric function for $S \subseteq [n-1]$ simply by $F_S[X]$, and similarly for the monomial quasisymmetric function.

Given a permutation $w\in \mathfrak{S}_n$, we associate a fundamental quasisymmetric function in the following way. We first define 
\begin{equation*}
\des(w):=\{i:w_i>w_{i+1}\}, 
\end{equation*}
then, clearly $\des(w)\subseteq [n-1]$. Now we let $F_{\des(w)}=F_{n,\des(w)}$. We also let $\iDes(w):=\des(w^{-1})$ and similarly define $F_{\iDes(w)}$.

It is worth noting that subsets of $[n-1]$ are in one-to-one correspondence with compositions of $n$ in a natural way. We denote by $\Set$ the bijection from compositions to $2^{[n-1]}$ given by
\begin{align*}
    \Set:(\beta_1, \beta_2 - \beta_1, \dots, n - \beta_{\ell-1}) \mapsto \{\beta_1 < \dots < \beta_{\ell-1}\}.
\end{align*}
For a composition $\alpha$, we abuse our notation to refer $F_\alpha$ and $M_\alpha$ as
\[
F_\alpha[X]=F_{n,\Set(\alpha)}[X] \quad \text{ and } \quad M_\alpha[X] = M_{n,\Set(\alpha)}[X].
\]
\section{Generalization of modified Macdonald polynomials}\label{Sec: generalization of Macdonald}

A (general) \emph{diagram} $D$ is a collection of points (cells) in $\mathbb{Z}_+ \times \mathbb{Z}_+$, and a \emph{bottom cell} of $D$ is a cell located in the lowest position in each column of $D$. We denote by $D^+$ the collection of cells in $D$ that are not bottom cells. A \emph{filled diagram} $(D,f)$ consists of a diagram $D$ together with a filling 
$$ f: D^{+} \to \mathbb{F}, $$
which assigns a scalar in $\mathbb{F}$ to each cell of $D$ that is not a bottom cell. For the remainder of this paper, we set $\mathbb{F} = \mathbb{Q}(q,t)$. We visualize $(D,f)$ by writing the corresponding value of $f$ for each cell. For example, see the left side of Figure~\ref{Fig: filled diagram}.

\begin{figure}[ht] \centering
\begin{tikzpicture}[scale=0.9]

\draw[-] (0,0)--(1,0);
\draw[-] (0,1)--(3,1);
\draw[-] (0,2)--(3,2);
\draw[-] (0,3)--(2,3);

\draw[-] (0,0)--(0,3);
\draw[-] (1,0)--(1,3);
\draw[-] (2,1)--(2,3);
\draw[-] (3,1)--(3,2);

\filldraw[black] (0.5,2.3) circle (0.000001pt) node[anchor=south] {$q^2$};
\filldraw[black] (0.5,1.3) circle (0.000001pt) node[anchor=south] {$q^3t$};

\filldraw[black] (1.5,2.3) circle (0.000001pt) node[anchor=south] {$t$};

\begin{scope}[shift={(8,0)}]

\draw[-] (0,0)--(1,0);
\draw[-] (0,1)--(3,1);
\draw[-] (0,2)--(3,2);
\draw[-] (0,3)--(2,3);

\draw[-] (0,0)--(0,3);
\draw[-] (1,0)--(1,3);
\draw[-] (2,1)--(2,3);
\draw[-] (3,1)--(3,2);

\filldraw[black] (0.5,2.3) circle (0.000001pt) node[anchor=south] {$1$};
\filldraw[black] (1.5,2.3) circle (0.000001pt) node[anchor=south] {$2$};
\filldraw[black] (0.5,1.3) circle (0.000001pt) node[anchor=south] {$3$};
\filldraw[black] (1.5,1.3) circle (0.000001pt) node[anchor=south] {$4$};
\filldraw[black] (2.5,1.3) circle (0.000001pt) node[anchor=south] {$5$};
\filldraw[black] (0.5,0.3) circle (0.000001pt) node[anchor=south] {$6$};

\end{scope}

\end{tikzpicture}
\caption{The left figure shows an example of a filled diagram $(D,f)$, and the right figure shows the total order $N_{D}$ on each cell of $D$.}\label{Fig: filled diagram}
\end{figure}

We denote a diagram by \( D = [D^{(1)}, D^{(2)}, \dots] \), where \( D^{(j)} \) represents the set of row indices of the cells in the \( j \)-th column. Throughout this paper, we assume that for every diagram \( D \), each column \( D^{(j)} \) is an interval \( [a,b] := \{a, a+1, \dots, b\} \) for some \( a \leq b \). For example, the diagram in Figure~\ref{Fig: filled diagram} is represented as 
\[
    D = [[1,3], [2,3], [2,2]],
\]
and the Young diagram for a partition \( \mu = (\mu_1, \dots, \mu_\ell) \) can be written as \( \mu = [[\mu'_1], \dots, [\mu'_{\mu_1}]] \), where \( \mu' \) is the conjugate of \( \mu \) and \( [n] \) denotes the shorthand notation for \( [1, n] = \{1, 2, \dots, n\} \).

\begin{definition} \label{def: stat definition}
We define a total order on the cells in \( D \) row by row, from top to bottom and left to right within each row. We denote this total order by \( N_D: D \rightarrow [|D|] \); see the right of Figure~\ref{Fig: filled diagram} for an example.

For a filled diagram \((D,f)\), we define functions
\[
    \inv_{D}: \mathfrak{S}_{|D|} \rightarrow \mathbb{F}, \qquad \maj_{(D,f)}: \mathfrak{S}_{|D|} \rightarrow \mathbb{F}
\]
as follows. For a permutation \( w \in \mathfrak{S}_{|D|} \), we say that a pair \((u,v)\) of cells in \( D \) is an \emph{inversion} with respect to \( w \) if \( w_{N_D(u)} > w_{N_D(v)} \) and either
\begin{itemize}
    \item \( u = (i,j) \), \( v = (i,j') \) such that \( j < j' \), or
    \item \( u = (i,j) \), \( v = (i-1,j') \) such that \( j > j' \).
\end{itemize}
Then we define
\[
\inv_{D}(w) := \prod_{(u,v)} q,
\]
where the product is over all pairs \((u,v)\) of cells in \( D \) that are inversions with respect to \( w \).

Similarly, for \( w \in \mathfrak{S}_{|D|} \), we say that a cell \( u = (i,j) \) in \( D \) is a \emph{descent} with respect to \( w \) if \( w_{N_D(u)} > w_{N_D(v)} \), where \( v = (i-1,j) \) is the cell just below \( u \). Then we define
\[
\maj_{(D,f)}(w) := \prod_{u} f(u),
\]
where the product is over all cells \( u \) that are descents with respect to \( w \). Since a bottom cell cannot be a descent, we do not require the filling \( f \) to assign a value to a bottom cell in the definition of a filled diagram.

Finally, we define a function \(\stat_{(D,f)}: \mathfrak{S}_{|D|} \rightarrow \mathbb{F}\) by
\[
\stat_{(D,f)}(w) := \inv_{D}(w) \maj_{(D,f)}(w).
\]

Although we have defined $\stat_{(D,f)}(w)$ for a permutation $w \in \mathfrak{S}_{|D|}$, $\stat_{(D,f)}(w)$ can also be defined for any word $w$ of length $|D|$ consisting of positive integers, by setting $\stat_{(D,f)}(w) := \stat_{(D,f)}(\std(w))$.

\end{definition}

\begin{example}
For the filled diagram $(D,f)$ on the left of Figure~\ref{Fig: filled diagram} and the permutation $w=254316$, the inversions with respect to $w$ are
\[
    ((3,2),(2,1)),\quad ((2,1),(2,2)),\quad ((2,1),(2,3)),\quad \text{and} \quad ((2,2),(2,3)).
\]
The only descent with respect to $w$ is the cell $(3,2)$. Therefore, we have $\inv_D(w) = q^4$ and $\maj_{(D,f)}(w) = t$, which gives $\stat_{(D,f)}(w) = q^4 t$.
\end{example}

Following the spirit of Haglund--Haiman--Loehr formula (Theorem~\ref{thm: HHL formula}), we define a generalization of the modified Macdonald polynomial for a filled diagram.
\begin{definition}\label{Def: generalized modified Macdonald}The \emph{(generalized) modified Macdonald polynomial} $\widetilde{H}_{(D,f)}[X;q,t]$ for a filled diagram $(D,f)$ is 
\[
    \widetilde{H}_{(D,f)}[X;q,t] = \sum_{w\in \mathfrak{S}_n} \stat_{(D,f)}(w) F_{\iDes(w)}.
\]
\end{definition}

Given a partition $\mu$, we define the \emph{standard filling} of $\mu$, \[f^{\st}_{\mu}: \mu^{+} \rightarrow \mathbb{F}\] by
\(f^{\st}_{\mu}(u)=q^{-\arm_\mu(u)}t^{\leg_\mu(u)+1}.
\) Then the modified Macdonald polynomial for the filled diagram $(\mu,f^{\st}_\mu)$ is the usual modified Macdonald polynomial
\[
\widetilde{H}_{(\mu,f^{\st}_\mu)}[X;q,t] = \widetilde{H}_\mu[X;q,t],
\]
by Theorem~\ref{thm: HHL formula}.
\begin{rmk}
Several generalizations of modified Macdonald polynomials for general diagrams have been studied in \cite{BBGHT99, Ban07, CM18}. Our definition (Definition~\ref{Def: generalized modified Macdonald}) differs from the lattice diagram polynomial in \cite{BBGHT99} and generalizes $C_L[X;q,t]$ in \cite{Ban07}. With a suitable modification, the modified Macdonald polynomials $\widetilde{H}_{(D,f)}[X;q,t]$ for a filled diagram $(D,f)$ actually coincide with the \emph{weighted characteristic polynomials} for (corner) weighted Dyck paths studied in \cite{CM18}.
\end{rmk}

\begin{example}\label{ex: filled diagram and hdf}
Let $(D,f)$ be a filled diagram depicted as below:
\[
    \begin{ytableau}
    \textcolor{white}{0} & \alpha & \none\\
    \none &  & \none
    \end{ytableau}
\]
where $\alpha\in\mathbb{F}$. The table below shows the statistics needed to compute $\widetilde{H}_{(D,f)}[X;q,t]$.
\begin{center}
\begin{tabular}{|c|c|c|c|c|c|c|}
    \hline
    $w \in \mathfrak{S}_3$ & 123 & 132 & 213 & 231 & 312 & 321 \\\hline
    $\iDes(w)$ & $\emptyset$ & $\{2\}$ & $\{1\}$ & $\{1\}$ & $\{2\}$ & $\{1,2\}$ \\
    \hline
    
    $\inv_D(w)$ & 1 & 1 & $q$ & 1 & $q$ & $q$\\
    \hline
    $\maj_{(D,f)}(w)$ & 1 & $\alpha$ & 1 & $\alpha$ & 1 & $\alpha$\\
    \hline
\end{tabular}
\end{center}

Thus, the modified Macdonald polynomial for $(D,f)$ is given by
\[
    \widetilde{H}_{(D,f)}[X;q,t] = F_{\emptyset} + (q+\alpha)F_{\{1\}} + (q+\alpha)F_{\{2\}} + q\alpha F_{\{1,2\}}.
\]
\end{example}

Two distinct filled diagrams $(D,f)$ and $(D',f')$ can give the same modified Macdonald polynomials. For example, given a filled diagram $(D,f)$ where 
\[
    D=[D^{(1)},D^{(2)},\dots,D^{(\ell)}],
\]
consider the diagram
\[
    D'=[D^{(2)},\dots,D^{(\ell)},D^{(1)}+1],
\]
where $I+1=\{a+1:a\in I\}$ for an interval $I$. In other words, $D'$ is the diagram obtained by moving the leftmost column of $D$ to the rightmost end and shifting it up by one row. The cells in $D$ and $D'$ are naturally in bijection, and we define the filling $f'$ on $D'$ by inheriting the filling $f$ on $D$. We define $\cycling (D,f):=(D',f')$; see Figure~\ref{fig: cycling} for an example. Then the following is a routine exercise.

\begin{lem}\label{lem: cycling}(Cycling rule)
Let $(D,f)$ be a filled diagram and $(D',f')=\cycling(D,f)$. Then we have
\[
    \widetilde{H}_{(D,f)}[X;q,t] = \widetilde{H}_{(D',f')}[X;q,t].
\]
\end{lem}
\begin{proof}
For any $w\in \mathfrak{S}_{|D|}$, we simply have
\begin{align*}
    \inv_D(w)=\inv_{D'}(w), \qquad  \maj_{(D,f)}(w)=\maj_{(D',f')}(w),
\end{align*}
which gives $\stat_{(D,f)}(w)=\stat_{(D',f')}(w)$.
\end{proof}

\begin{figure}[ht] \centering
\begin{tikzpicture}[scale=0.9]

\draw[-] (0,0)--(3,0);
\draw[-] (0,1)--(3,1);
\draw[-] (0,2)--(3,2);
\draw[-] (0,3)--(1,3);

\draw[-] (0,0)--(0,3);
\draw[-] (1,0)--(1,3);
\draw[-] (2,0)--(2,2);
\draw[-] (3,0)--(3,2);

\filldraw[black] (0.5,2.3) circle (0.000001pt) node[anchor=south] {$a$};
\filldraw[black] (0.5,1.3) circle (0.000001pt) node[anchor=south] {$b$};

\filldraw[black] (1.5,1.3) circle (0.000001pt) node[anchor=south] {$c$};
\filldraw[black] (2.5,1.3) circle (0.000001pt) node[anchor=south] {$d$};

\draw[->] (5,2)--(7,2);
\filldraw[black] (6,2) circle (0.000001pt) node[anchor=south] {$\cycling$};

\begin{scope}[shift={(8,0)}]
\draw[-] (0,0)--(2,0);
\draw[-] (0,1)--(3,1);
\draw[-] (0,2)--(3,2);
\draw[-] (2,3)--(3,3);
\draw[-] (2,4)--(3,4);

\draw[-] (0,0)--(0,2);
\draw[-] (1,0)--(1,2);
\draw[-] (2,0)--(2,4);
\draw[-] (3,1)--(3,4);

\filldraw[black] (0.5,1.3) circle (0.000001pt) node[anchor=south] {$c$};
\filldraw[black] (1.5,1.3) circle (0.000001pt) node[anchor=south] {$d$};

\filldraw[black] (2.5,3.3) circle (0.000001pt) node[anchor=south] {$a$};
\filldraw[black] (2.5,2.3) circle (0.000001pt) node[anchor=south] {$b$};

\end{scope}

\end{tikzpicture}
\caption{Applying the operator $\cycling$.}\label{fig: cycling}
\end{figure}

In Section~\ref{Sec: Butler permutations, main Lemmas} and Section~\ref{Sec: Proof of Schur positivities}, we provide various relations between modified Macdonald polynomials (or LLT polynomials). We conclude this section with the definitions that will be used throughout the rest of the paper.

\begin{definition}
    Let ${w}=a_1\cdots a_n$ be a word consisting of positive integers. A \emph{standardization} $\std({w})$ of ${w}$ is defined to be the unique permutation $\sigma\in\mathfrak{S}_n$ such that 
    \[
    w_i< w_j \text{ or } (w_i=w_j \text{ and } i<j) \text{ if and only if } \std(w)_i<\std(w)_j.
    \] 
    For example, we have $\std (3155)=2134$. Given a map $\phi:\mathfrak{S}_n\rightarrow\mathfrak{S}_n$, we abuse the notation and refer to $\phi({w})$ as the word $w_{\sigma(1)}\cdots w_{\sigma(n)}$, where $\sigma = \std(w)^{-1}\circ \phi(\std(w))$. For example, if $\phi$ sends $2134$ to $4312$, then $\phi(3155)=5513$.

   For a permutation $w\in\mathfrak{S}_n$ and a subset $C=\{c_1<\dots<c_r\}$ of $[n]$, we define 
   \begin{equation*}
       w\vert_{C}:=w_{c_1}\cdots w_{c_r},
   \end{equation*}
   which we frequently abbreviate as $w_C$.
   Let $D$ be a diagram and $E\subseteq D$ be its subdiagram. Given a permutation $w\in \mathfrak{S}_{|D|}$, we define
    \begin{equation*}
        w{\downarrow}^{D}_{E}:=w\vert_{\{N_D(u):u \in E\}}.
    \end{equation*}
     Now given a map $\phi: \mathfrak{S}_{|E|}\rightarrow\mathfrak{S}_{|E|}$, we define $ \phi{\uparrow}^{D}_{E}:\mathfrak{S}_{|D|}\rightarrow\mathfrak{S}_{|D|}$ to be the (unique) map satisfying
    \begin{align*}
        \left(\phi{\uparrow}^{D}_{E}(w)\right){\downarrow}_{E}^{D}=\phi(w{\downarrow}^{D}_{E}),\qquad \left(\phi{\uparrow}^{D}_{E}(w)\right){\downarrow}_{D\setminus E}^{D}=w{\downarrow}^{D}_{D\setminus E}.
    \end{align*}
In other words, the map $\phi{\uparrow}^{D}_{E}$ preserves the restriction $w{\downarrow}^{D}_{D\setminus E}$ of the word $w$ on $D\setminus E$ and changes the restriction $w{\downarrow}^{D}_{E}$ of $w$ on $E$ via the map $\phi$. 
\end{definition}

\begin{example}
Figure~\ref{fig: diagram and sub-diagram} shows a diagram $D$ with a total order $N_D$ written on each cell. A subdiagram $E$ is outlined with a blue boundary. We have $\{N_D(u) : u \in E\} = \{3,5,7,8\}$, which means that for a permutation $w = w_1 \cdots w_9 =629417538 \in \mathfrak{S}_9$,
\[
w\downarrow_{E}^{D} = w_3 w_5 w_7 w_8 = 9153.
\]
Therefore, denoting $w' = \phi\uparrow_{E}^{D}(w)$, we have $w'_i = w_i$ if $i \notin \{3,5,7,8\}$, and the word $w'_3w'_5w'_7 w'_8$ is given by $\phi(w_3w_5w_7w_8)=\phi(9153)$.  
For example, if a map $\phi : \mathfrak{S}_4 \to \mathfrak{S}_4$ sends $4132$ to $1423$, then $\phi(9153)=1935$, and 
\[
\phi\uparrow_{E}^{D}(62\underline{9}4\underline{1}7\underline{5}\underline{3}8) = 62\underline{1}4\underline{9}7\underline{3}\underline{5}8.
\]

\begin{figure}[ht] \centering
\begin{tikzpicture}[scale=0.9]

\draw[-] (0,0)--(4,0);
\draw[-] (0,1)--(4,1);
\draw[-] (0,2)--(2,2);
\draw[-] (0,3)--(2,3);
\draw[-] (0,4)--(1,4);

\draw[-] (0,0)--(0,4);
\draw[-] (1,0)--(1,4);
\draw[-] (2,0)--(2,3);
\draw[-] (3,0)--(3,1);
\draw[-] (4,0)--(4,1);

\draw[blue,very thick] (1,0)--(1,3);
\draw[blue,very thick] (1,3)--(2,3);
\draw[blue,very thick] (2,3)--(2,1);
\draw[blue,very thick] (2,1)--(3,1);
\draw[blue,very thick] (3,1)--(3,0);
\draw[blue,very thick] (3,0)--(1,0);

\filldraw[black] (0.5,3.3) circle (0.000001pt) node[anchor=south] {$1$};
\filldraw[black] (0.5,2.3) circle (0.000001pt) node[anchor=south] {$2$};
\filldraw[black] (0.5,1.3) circle (0.000001pt) node[anchor=south] {$4$};
\filldraw[black] (0.5,0.3) circle (0.000001pt) node[anchor=south] {$6$};

\filldraw[black] (1.5,2.3) circle (0.000001pt) node[anchor=south] {$3$};
\filldraw[black] (1.5,1.3) circle (0.000001pt) node[anchor=south] {$5$};
\filldraw[black] (1.5,0.3) circle (0.000001pt) node[anchor=south] {$7$};

\filldraw[black] (2.5,0.3) circle (0.000001pt) node[anchor=south] {$8$};
\filldraw[black] (3.5,0.3) circle (0.000001pt) node[anchor=south] {$9$};

\end{tikzpicture}
\caption{A diagram and its subdiagram.}
\label{fig: diagram and sub-diagram}
\end{figure}
\end{example}

\section{Column exchange rule and Butler permutations}\label{Sec: Butler permutations, main Lemmas} In this section, we prove Proposition~\ref{lem: column exchange} and Proposition~\ref{lem: main LLT lemma}, which will be central tools for the proof of Theorem~\ref{thm: first main, F-expansion} and Theorem~\ref{thm: second main, s-positivities}. We also describe the permutations, which we call Butler permutations, that appear in the $F$-expansion for $\I_{\lambda,\mu}[X;q,t]$. 
 
\subsection{Column exchange rule}
In this subsection, we introduce the \emph{column exchange rule} (Proposition~\ref{lem: column exchange}). Given a filled diagram $(D,f)$ satisfying a certain (local) condition \eqref{eq: column exchange condition}, Proposition~\ref{lem: column exchange} enables us to find another filled diagram $(D',f')=S_j(D,f)$ (see Definition~\ref{def: the operator S}) such that
\begin{equation*}
    \widetilde{H}_{(D,f)}[X;q,t]=\widetilde{H}_{(D',f')}[X;q,t].
\end{equation*}

\begin{definition}\label{def: the operator S}
For positive integers $n>m$, we define $\mathcal{V}(n,m)$ to be the set of all filled diagrams $(\mu,f_{\mu})$ such that $\mu=[[n],[m]]$ and a filling $f_\mu$ on it satisfying the following condition:
\begin{equation}\label{eq: column exchange condition}
   f_\mu(i,1) = q^{-1}f_\mu(m+1,1) f_\mu(i,2), \text{ for } 1<i\leq m.  
\end{equation}
For $(\mu,f_{\mu})\in \mathcal{V}(n,m)$, we let $S(\mu,f_{\mu})$ be the filled diagram $(\lambda,f_{\lambda})$ such that $\lambda=[[m],[n]]$ and:
\begin{align}
        &f_\lambda(i,2)=f_\mu(i,1), \text{ for } i>m+1\label{eq: column exchange the first condition},\\
        &f_\lambda(m+1,2)=q^{-1}f_\mu(m+1,1)\label{eq: column exchange the second condition},\\
        &f_\lambda(i,1) = f_{\mu}(i,2) \text{ and } f_\lambda(i,2)=f_\mu(i,1), \text{ for } 1<i\leq m.  \nonumber
\end{align}
See Figure~\ref{fig: generic filled diagrams column exchange lemma} for the generic $(\mu,f_{\mu})\in\mathcal{V}(n,m)$ and the corresponding $S(\mu,f_{\mu})$.

\begin{figure}[h]
\[
    \ytableausetup{boxsize=3.3em}
    (\mu,f_{\mu}) = \begin{ytableau}
     b_{1} \\
     \vdots  \\
     b_{n-m-1} \\
    q\alpha  \\
    \alpha a_1 & a_1 \\
    \vdots & \vdots \\
    \alpha  a_{m-1} & a_{m-1} \\
     & 
    \end{ytableau}
    \qquad \qquad S(\mu,f_{\mu} ) = \begin{ytableau}
     \none &b_{1} \\
     \none &\vdots  \\
     \none &b_{n-m-1} \\
    \none & \alpha  \\
    a_1 & \alpha a_1 \\
    \vdots & \vdots \\
     a_{m-1} & \alpha a_{m-1} \\
     & 
    \end{ytableau}
\]
\caption{Any $(\mu,f_{\mu})\in\mathcal{V}(n,m)$ is of the form 
 in the left figure for suitable $a_i$'s, $b_i$'s and $\alpha$ in $\mathbb{F}$. The right figure shows the corresponding $S(\mu,f_{\mu})$.}\label{fig: generic filled diagrams column exchange lemma}
\end{figure}

Let $(D,f)$ be a filled diagram such that the restriction to the $j$ and $j+1$-th columns is in $\mathcal{V}(n,m)$ for some $n$ and $m$. We define $S_j(D,f)$ to be the filled diagram obtained from $(D,f)$ by applying the map $S$ to the two columns ($j$ and $j+1$-th columns).

\end{definition}

\begin{example}
Consider the partition $\lambda=(3,3,2,1)$. The left of Figure~\ref{fig: application of column exchange picture} shows the filled diagram $(\lambda,f^{\st}_{\lambda})$. The restriction to the first two columns of $(\lambda,f_{\lambda}^{\st})$ is in $\mathcal{V}(4,3)$ so we can apply the operator $S_1$ and get the filled diagram in the middle. Now the restriction to the second and third columns of $S_1(\lambda,f_{\lambda}^{\st})$ is in $\mathcal{V}(4,2)$, applying $S_2$ gives the filled diagram on the right.
    \begin{figure}[h] \centering
\begin{tikzpicture}[scale=0.9]

\draw[-] (0,0)--(3,0);
\draw[-] (0,1)--(3,1);
\draw[-] (0,2)--(3,2);
\draw[-] (0,3)--(2,3);
\draw[-] (0,4)--(1,4);

\draw[-] (0,0)--(0,4);
\draw[-] (1,0)--(1,4);
\draw[-] (2,0)--(2,3);
\draw[-] (3,0)--(3,2);

\filldraw[black] (0.5,3.3) circle (0.000001pt) node[anchor=south] {$t$};
\filldraw[black] (0.5,2.3) circle (0.000001pt) node[anchor=south] {$q^{-1}t^2$};
\filldraw[black] (0.5,1.3) circle (0.000001pt) node[anchor=south] {$q^{-2}t^3$};

\filldraw[black] (1.5,2.3) circle (0.000001pt) node[anchor=south] {$t$};
\filldraw[black] (1.5,1.3) circle (0.000001pt) node[anchor=south] {$q^{-1}t^2$};

\filldraw[black] (2.5,1.3) circle (0.000001pt) node[anchor=south] {$t$};

\filldraw[black] (1.5,-1) circle (0.000001pt) node[anchor=south] {$(\lambda,f_{\lambda}^{\st})$};

\begin{scope}[shift={(5,0)}]
\draw[-] (0,0)--(3,0);
\draw[-] (0,1)--(3,1);
\draw[-] (0,2)--(3,2);
\draw[-] (0,3)--(2,3);
\draw[-] (1,4)--(2,4);

\draw[-] (0,0)--(0,3);
\draw[-] (1,0)--(1,4);
\draw[-] (2,0)--(2,4);
\draw[-] (3,0)--(3,2);

\filldraw[black] (1.5,3.3) circle (0.000001pt) node[anchor=south] {$q^{-1}t$};
\filldraw[black] (1.5,2.3) circle (0.000001pt) node[anchor=south] {$q^{-1}t^2$};
\filldraw[black] (1.5,1.3) circle (0.000001pt) node[anchor=south] {$q^{-2}t^3$};

\filldraw[black] (0.5,2.3) circle (0.000001pt) node[anchor=south] {$t$};
\filldraw[black] (0.5,1.3) circle (0.000001pt) node[anchor=south] {$q^{-1}t^2$};

\filldraw[black] (2.5,1.3) circle (0.000001pt) node[anchor=south] {$t$};

\filldraw[black] (1.5,-1) circle (0.000001pt) node[anchor=south] {$S_1(\lambda,f_{\lambda}^{\st})$};
\end{scope}

\begin{scope}[shift={(10,0)}]
\draw[-] (0,0)--(3,0);
\draw[-] (0,1)--(3,1);
\draw[-] (0,2)--(3,2);
\draw[-] (0,3)--(1,3);\draw[-] (2,3)--(3,3);
\draw[-] (2,4)--(3,4);

\draw[-] (0,0)--(0,3);
\draw[-] (1,0)--(1,3);
\draw[-] (2,0)--(2,4);
\draw[-] (3,0)--(3,4);

\filldraw[black] (2.5,3.3) circle (0.000001pt) node[anchor=south] {$q^{-1}t$};
\filldraw[black] (2.5,2.3) circle (0.000001pt) node[anchor=south] {$q^{-2}t^2$};
\filldraw[black] (2.5,1.3) circle (0.000001pt) node[anchor=south] {$q^{-2}t^3$};

\filldraw[black] (0.5,2.3) circle (0.000001pt) node[anchor=south] {$t$};
\filldraw[black] (0.5,1.3) circle (0.000001pt) node[anchor=south] {$q^{-1}t^2$};

\filldraw[black] (1.5,1.3) circle (0.000001pt) node[anchor=south] {$t$};

\filldraw[black] (1.5,-1) circle (0.000001pt) node[anchor=south] {$S_2 S_1(\lambda,f_{\lambda}^{\st})$};
\end{scope}

\end{tikzpicture}
\caption{Applying the operators $S_j$.}\label{fig: application of column exchange picture}
\end{figure}
\end{example}

\begin{proposition}\label{lem: column exchange}
For positive integers $n>m$, let $(\mu,f_{\mu})\in\mathcal{V}(n,m)$ and $(\lambda,f_{\lambda})=S(\mu,f_{\mu})$.

Then there is a bijection $\phi_{n,m}:\mathfrak{S}_{n+m}\rightarrow\mathfrak{S}_{n+m}$ satisfying the following three conditions:
\begin{align*}
    (\phi1)\quad& \stat_{(\mu,f_{\mu})}(w) = \stat_{(\lambda,f_{\lambda})}(\phi_{n,m}(w)), \\
    (\phi2)\quad& \iDes(w) = \iDes(\phi_{n,m}(w)), \text{ and}\\
    (\phi3)\quad& \text{denoting $w'=\phi_{n,m}(w)$, we have } \\ &\begin{cases}
     \{w_{n-m+2i-1}, w_{n-m+2i}\} = \{w'_{n-m+2i-1}, w'_{n-m+2i}\} & \text{  for } 1\le i \le m, \text{ and }
    \\
    w_i=w'_i  & \text{  for } 1\le i \le n-m. 
    \end{cases}
\end{align*}
\end{proposition}

The proof of Proposition~\ref{lem: column exchange} will be deferred to the latter part of this section; for now, we first present an immediate consequence of it.

\begin{corollary}\label{cor: column exchange}
Let $(D,f)$ be a filled diagram such that the restriction to the $j$ and $j+1$-th columns is in $\mathcal{V}(n,m)$. We denote $C_i$ to be the set of $N_D(u)$ where $u$ ranges over cells in the $i$-th row of $D$. Then there is a bijection $\phi: \mathfrak{S}_{|D|}\rightarrow \mathfrak{S}_{|D|}$ satisfying the following three conditions:
\begin{align*}
    (\phi1)\quad& \stat_{(D,f)}(w) = \stat_{S_j(D,f)}(\phi(w)), \\
    (\phi2)\quad& \iDes(w) = \iDes(\phi(w)), \text{ and}\\
    (\phi3)\quad& \text{denoting $w'=\phi(w)$, we have } \{w_{a}: a \in C_i\} = \{w'_{a}: a \in C_i\}  \text{  for all } i. 
\end{align*}
In particular, we have
\begin{equation}\label{eq: column exchange final} \widetilde{H}_{(D,f)}[X;q,t]= \widetilde{H}_{S_j(D,f)}[X;q,t].\end{equation} 
\end{corollary}
\begin{proof}
Let $E$ be the subdiagram of $D$ obtained by restricting to the $j$ and $(j+1)$-th columns. Assume that we have constructed a map $\phi_{n,m}$ in Proposition~\ref{lem: column exchange}. Then we claim that $\phi = (\phi_{n,m})\uparrow_{E}^{D}$ is the desired bijection.

Let $w \in \mathfrak{S}_{|D|}$ and denote $w' = \phi(w)$. For $1 \leq i \leq m$, let $b_i = N_D(u)$ for a cell $u = (i,j)$ in $D$. Then, obviously, we have $b_i+1 = N_D(u')$ for $u' = (i,j+1)$. Since $\phi_{n,m}$ satisfies its property $(\phi3)$ in Proposition~\ref{lem: column exchange}, we conclude that
\begin{equation}\label{123}
w_a = w'_a \quad \text{if } a \neq b_i, b_i+1, \hspace{1mm}\text{and} \hspace{2mm} \{w_{b_i}, w_{b_i+1}\} = \{w'_{b_i}, w'_{b_i+1}\} \quad \text{for } 1 \leq i \leq m.
\end{equation}
Therefore, $\phi$ satisfies $(\phi3)$.

Pick some integer $1 \le r \le |D|-1$ and take $u,v \in D$ such that $w_{N_D(u)} = r$ and $w_{N_D(v)} = r+1$. If $u,v \notin E$, then it is obvious that $r \in \iDes(w)$ if and only if $r \in \iDes(w')$. Now assume $u \in E$ and $v \notin E$, and further assume that $r$ is of the form $(i,j)$. Then we have $w'_{N(v)} = r+1$ and either $w'_{N(u)} = r$ or $w'_{N(u)+1} = r$. Therefore, we deduce that $r \in \iDes(w)$ if and only if $r \in \iDes(w')$. Applying similar reasoning, we obtain that if only one of $r$ and $r+1$ lies in $E$, then $r \in \iDes(w)$ if and only if $r \in \iDes(w')$.
We conclue $\iDes(w) = \iDes(w')$ unless there exists $r$ such that $\{w_{b_i}, w_{b_i+1}\} = \{w'_{b_i}, w'_{b_i+1}\} = \{r, r+1\}$ for some $i$, with $w_{b_i} = w'_{b_i+1}$ and $w_{b_i+1} = w'_{b_i}$. Assume that such an $r$ exists, and let $y = \std(w\downarrow^{D}_{E})$ and $y' = \std(w'\downarrow^{D}_{E})$. We further denote by $t$ the position where $r$ is the $t$-th smallest letter among the letters in $w\downarrow^{D}_{E}$; that is, $r$ is mapped to $t$ under standardization. Then $t$ belongs to exactly one of $\iDes(y)$ and $\iDes(y')$. Since $\phi_{n,m}$ satisfies its property $(\phi2)$, we have
\[
\iDes(y) = \iDes(\phi_{n,m}(y)) = \iDes(y'),
\]
which is a contradiction. Thus, $\phi$ satisfies $(\phi2)$.

It remains to show that $\phi$ satisfies $(\phi1)$. First, denote $(D', f') = S_j(D, f)$. Let $L_1$ be the set of inversions $(u,v)$ with respect to $w$ such that $u,v \notin E$; $L_2$ the set of inversions $(u,v)$ with respect to $w$ such that exactly one of $u,v$ belongs to $E$; and $L_3$ the set of inversions $(u,v)$ with respect to $w$ such that $u,v \in E$.  
We further define $P_1$ to be the set of descents $u$ with respect to $w$ such that $u \notin E$, and $P_2$ the set of descents $u$ with respect to $w$ such that $u \in E$.  
Similarly, we define $L'_1, L'_2, L'_3, P'_1$, and $P'_2$ for $w'$.  

Then, obviously, we have $L_1 = L'_1$ and $P_1 = P'_1$, once we naturally identify the cells in $D$ and $D'$ other than those in the $j$ or $(j+1)$-th columns. Moreover, we have $|L_2| = |L'_2|$ due to \eqref{123}. Therefore, we conclude that
\begin{align*}
    \stat_{(D,f)}(w) &= \inv_{D}(w) \maj_{(D,f)}(w) \\
    &= \left(q^{|L_1|+|L_2|} \prod_{u \in P_1} f(u)\right) \stat_{(E,g)}(w\downarrow^{D}_{E}) \\
    &= \left(q^{|L'_1|+|L'_2|} \prod_{u \in P'_1} f'(u)\right) \stat_{(E',g')}(w'\downarrow^{D'}_{E'}) \\
    &= \stat_{(D',f')}(w'),
\end{align*}
where $(E, g)$ (respectively, $(E', g')$) is the filled diagram obtained from $(D, f)$ (respectively, $(D', f')$) by restricting to the $j$ and $(j+1)$-th columns. Note that $$\stat_{(E,g)}(w\downarrow^{D}_{E}) = \stat_{(E',g')}(w'\downarrow^{D'}_{E'})$$ follows from the fact that $\phi_{n,m}$ satisfies its property $(\phi1)$.

\end{proof}

\begin{rmk} Let $(D,f)=(\mu,f^{\st}_\mu)$ for a partition $\mu$ and standard filling $f^{\st}_\mu$. It is easy to check that $S_j(D,f)$ is well-defined as long as $j<\mu_1$ and $\mu'_j>\mu'_{j+1}$. Then \eqref{eq: column exchange final} can be also deduced from \cite[Theorem 5.1.1]{HHL08}. However, their proof was not a bijective proof, while ours gives a bijective proof. In addition, we would like to highlight advantages of Corollary \ref{cor: column exchange}.

First of all, \eqref{eq: column exchange final} implies that there exists a bijection $\phi: \mathfrak{S}_{|D|}\rightarrow \mathfrak{S}_{|D|}$ satisfying conditions $(\phi1)$ and $(\phi2)$. Our bijection also satisfies $(\phi3)$ answering a question of Haglund \cite[Conjecture 22]{AS17}. Moreover, we rely on a map constructed in Proposition \ref{lem: column exchange} in the proof of Proposition~\ref{lem: main LLT lemma} which is the main ingredient for the proof of our main theorem (Theorem \ref{thm: first main, F-expansion}). Finally, we may apply Corollary \ref{cor: column exchange}  for $(D,f)$ other than $(\mu,f^{\st}_\mu)$.
\end{rmk}

The following definition and lemma will be used in the proof of Proposition~\ref{lem: column exchange}. 
\begin{definition}
    For a permutation $w\in \mathfrak{S}_n$, we define
    \begin{align*}
        \overline{\iDes}(w)=\begin{cases*}
            \iDes(w)\setminus\{w_n\} \qquad \text{ if } w_{n}=w_{n-1}-1\\
            \iDes(w) \qquad \text{ otherwise}
        \end{cases*}.
    \end{align*}
    For example, we have $\iDes(14532)=\{2,3\}$ while $\overline{\iDes}(14532)=\{3\}$.
\end{definition}

\begin{lem}\label{lem: ides preserving}
    For permutations $w\in \mathfrak{S}_n$ and $\sigma\in \mathfrak{S}_{n-2}$ such that $\{n-3,n-2\}=\{\sigma_{n-3},\sigma_{n-2}\}$, consider 
    \begin{equation*}
        w^{(1)}=w_{\sigma_1}\cdots w_{\sigma_{n-2}} w_{n-1} w_n, \qquad  w^{(2)}=w_{\sigma_1}\cdots w_{\sigma_{n-2}}w_{n}w_{n-1}. 
    \end{equation*} 
    Then:
        \begin{enumerate}
            \item If \hspace{1.5mm}$\iDes(\std(w_1 \cdots w_{n-2}))=\iDes(\std(w_{\sigma_1} \cdots w_{\sigma_{n-2}}))$, we have
            \begin{enumerate}
                \item $\iDes(w)=\iDes(w^{(1)})$,
                \item $\overline{\iDes}(w)=\overline{\iDes}(w^{(1)})=\overline{\iDes}(w^{(2)})$, and
                \item $\iDes(w)=\iDes(w^{(2)})$ if $|w_{n}-w_{n-1}|>1$.
            \end{enumerate}
            \item If \hspace{1.5mm}$\overline{\iDes}(\std(w_1 \cdots w_{n-2}))=\overline{\iDes}(\std(w_{\sigma_1}\cdots w_{\sigma_{n-2}}))$ and $|w_{n-2}-w_{n-3}|>1$, we have
            \begin{enumerate}
                \item $\iDes(w)=\iDes(w^{(1)})$,
                \item $\overline{\iDes}(w)=\overline{\iDes}(w^{(1)})=\overline{\iDes}(w^{(2)})$, and
                \item $\iDes(w)=\iDes(w^{(2)})$ if $|w_{n}-w_{n-1}|>1$.
            \end{enumerate}
        \end{enumerate}
\end{lem}
\begin{proof}
    We provide a proof for (2a), as the remaining cases are easier or follow similarly.

Let $i \in \iDes(w)$. Then there exist indices $r < s$ such that $w_r = i+1$ and $w_s = i$.  
If $s > n-2$, then $i \in \iDes(w^{(1)})$ since $w_{n-1} = w^{(1)}_{n-1}$ and $w_n = w^{(1)}_n$.  
On the other hand, if $s \leq n-2$ we cannot have $s=n-2$ and $r=n-3$. Let $t$ be an integer such that $i$ is the $t$-th smallest letter among letters in $w_{1}\cdots w_{n-2}$. Then $t\in\overline{\iDes}(\std(w_1 \cdots w_{n-2}))$. Note that $i+1$ appears before $i$ in $w_{\sigma_1} \cdots w_{\sigma_{n-2}}$ otherwise $t$ does not belong to $\overline{\iDes}(\std(w_{\sigma_1}\cdots w_{\sigma_{n-2}}))$. We conclude that $i \in \iDes(w^{(1)})$.  

Similarly, if $i \notin \iDes(w)$, then $i \notin \iDes(w^{(1)})$.

\end{proof}

\begin{proof}[Proof of Proposition~\ref{lem: column exchange}]
 Once we have constructed $\phi_{m+1,m}$ with the desired properties ($\phi1$), ($\phi2$) and ($\phi3$), we simply let \begin{equation*}
\phi_{n,m}=\phi_{m+1,m}\uparrow_{[[m+1],[m]]}^{[[n],[m]]},
\end{equation*}
then $\phi_{n,m}$ satisfies ($\phi1$), ($\phi2$) and ($\phi3$) due to \eqref{eq: column exchange the first condition}. From now on we fix $m=n-1$. It is enough to show the claim for $(\mu^{(n)},f_{\mu^{(n)}})$ and $(\lambda^{(n)}, f_{\lambda^{(n)}})$ as in Figure~\ref{fig: filled diagrams column exchange lemma}, where $\alpha\in\mathbb{F}$ and $a_i\in\mathbb{F}$. In other words, we argue that there is a bijection \(\phi_n:=\phi_{n,n-1}: \mathfrak{S}_{2n-1}\rightarrow \mathfrak{S}_{2n-1}\) satisfying the following three conditions:
\begin{align*}
    (\phi1)\quad& \stat_{(\mu^{(n)},f_{\mu^{(n)}})}(w) = \stat_{(\lambda^{(n)},f_{\lambda^{(n)}})}(\phi_{n}(w)), \\
    (\phi2)\quad& \iDes(w) = \iDes(\phi_{n}(w)), \text{ and}\\
    (\phi3)\quad& \text{denoting $w'=\phi_{n}(w)$, we have } \\ &\begin{cases}
    \{w_{2i}, w_{2i+1}\} = \{w'_{2i}, w'_{2i+1}\} & \text{ for } 1\le i \le n-1, \text{ and }\\
    w_1=w'_1. &
    \end{cases}  
\end{align*}
\begin{figure}[h]
\[
    \ytableausetup{boxsize=3em}
    (\mu^{(n)},f_{\mu^{(n)}}) = \begin{ytableau}
    q\alpha  \\
    \alpha a_1 & a_1 \\
    \alpha a_2 & a_2 \\
    \vdots & \vdots \\
    \alpha  a_{n-2} & a_{n-2} \\
     & 
    \end{ytableau}
    \qquad \qquad (\lambda^{(n)}, f_{\lambda^{(n)}}) = \begin{ytableau}
    \none & \alpha  \\
    a_1 & \alpha a_1 \\
    a_2 & \alpha a_2 \\
    \vdots & \vdots \\
     a_{n-2} & \alpha a_{n-2} \\
     & 
    \end{ytableau}
\]
\caption{Filled diagrams $(\mu^{(n)},f_{\mu^{(n)}})$ and $(\lambda^{(n)},f_{\lambda^{(n)}})$. }\label{fig: filled diagrams column exchange lemma}
\end{figure}

Given a filled diagram $(D,f)$, we define maps $\stat^{(1)}$ and $\stat^{(2)}$ from $\mathfrak{S}_{|D|}$ to $\mathbb{F}$ as follows:
\begin{align}\label{eq: stat1 stat2 definition}
    \stat^{(1)}_{(D,f)}(w) := \begin{cases}
    q\stat_{(D,f)}(w) \quad&\text{ if } w_{|D|-1} < w_{|D|}\\
    \alpha\stat_{(D,f)}(w) \quad&\text{ if } w_{|D|-1} > w_{|D|}
    \end{cases},\\
   \stat^{(2)}_{(D,f)}(w) := \begin{cases}
    q\alpha\stat_{(D,f)}(w) \quad&\text{ if } w_{|D|-1} < w_{|D|}\\
    \stat_{(D,f)}(w) \quad&\text{ if } w_{|D|-1} > w_{|D|}\nonumber
    \end{cases}.
\end{align}

To recursively construct $\phi_n$, we need an auxiliary bijection $\psi_n:\mathfrak{S}_{2n-1}\rightarrow\mathfrak{S}_{2n-1}$ satisfying the following three conditions:
\begin{align*}
    (\psi1)\quad& \stat^{(1)}_{(\mu^{(n)},f_{\mu^{(n)}})}(w) = \stat^{(2)}_{(\lambda^{(n)},f_{\lambda^{(n)}})}(\psi_{n}(w)), \\
    (\psi2)\quad& \overline{\iDes}(w) = \overline{\iDes}(\psi_{n}(w)), \text{ and}\\
    (\psi3)\quad& \text{denoting $w'=\psi_{n}(w)$, we have } 
    \\ &\begin{cases}
    \{w_{2i}, w_{2i+1}\} = \{w'_{2i}, w'_{2i+1}\} &\text{ for } 1\le i \le n-1, \text{ and }
    \\ w_1=w'_1.  &
    \end{cases}
\end{align*}
We proceed by induction on $n$. We will construct $\phi_{n+1}$ and $\psi_{n+1}$ from $\phi_{n}$ and $\psi_{n}$. For the initial case $n=2$, we define $\phi_2$ and $\psi_2$ by
\begin{align*}
    &\phi_2(123) = 123, \quad \phi_2(132) = 132, \quad\phi_2(213) = 231, \\ & \phi_2(231) = 213, \quad\phi_2(312)=312,  \quad\phi_2(321) = 321,\\
    &\psi_2(123) = 132,\quad \psi_2(132) = 123, \quad\psi_2(213) = 213, \\ & \psi_2(231) = 231, \quad\psi_2(312) = 321, \quad \psi_2(321) = 312.
\end{align*}
The table below shows that $\phi_2$ and $\psi_2$ indeed satisfy all desired properties $(\phi1),$ $ (\phi2)$ and $(\phi3)$ and $(\psi1), (\psi2)$, and $(\psi3)$.
\begin{center}
\begin{tabular}{|c|c|c|c|c|c|c|}
    \hline
    $w$ &  $\stat_{(\mu^{(2)},f_{\mu^{(2)}})}$  & $\stat_{(\lambda^{(2)},f_{\lambda^{(2)}})}$ & $\stat^{(1)}_{(\mu^{(2)},f_{\mu^{(2)}})}$  & $\stat^{(2)}_{(\lambda^{(2)},f_{\lambda^{(2)}})}$ & $\iDes$ & $\overline{\iDes}$ \\\hline
    123 & 1 & 1 & $q$ & $q\alpha$ &$\emptyset$ & $\emptyset$ \\\hline
    132 & $q$ & $q$ & $q\alpha$ & $q$ & $\{2\}$ & $\emptyset$\\ \hline
    213 & $q\alpha$ & $q$ & $q^2\alpha$ & $q^2\alpha$ & $\{1\}$& $\{1\}$ \\ \hline
    231 & $q$ & $q\alpha$ & $q\alpha$ & $q\alpha$ & $\{1\}$ & $\{1\}$\\ \hline
    312 & $q\alpha$ & $q\alpha$ & $q^2\alpha$ & $q^2\alpha^2$ & $\{2\}$ &$\{2\}$\\ \hline
    321 & $q^2\alpha$ & $q^2\alpha$ &$q^2\alpha^2$ & $q^2\alpha$ & $\{1,2\}$ & $\{2\}$ \\ \hline
\end{tabular}
\end{center}

Now assume that we have constructed bijections 
\[
\phi_n:\mathfrak{S}_{2n-1}\rightarrow \mathfrak{S}_{2n-1} \qquad \text{ and } \qquad \psi_n:\mathfrak{S}_{2n-1}\rightarrow \mathfrak{S}_{2n-1}
\]
satisfying $(\phi1), (\phi2)$ and $(\phi3)$ and $(\psi1), (\psi2)$ and $(\psi3)$, respectively.

\begin{table}[ht]
    \centering
    \footnotesize{\begin{tabular}{|c|c|c|c|c|}
    \hline
    &$\std(w_{[2n-2,2n+1]})$ &   $\frac{\stat_{(\mu^{(n+1)},f_{\mu^{(n+1)}})}(w)}{\stat_{(\mu^{(n)},f_{\mu^{(n)}})}(w_{[2n-1]})}$  &  $\std(w_{[2n-2,2n+1]})$ & $\frac{\stat_{(\lambda^{(n+1)},f_{\lambda^{(n+1)}})}(w)}{
   \stat_{(\lambda^{(n)},f_{\lambda^{(n)}})}(w_{[2n-1]})}$ \\\hline
     (Case1)&1234 & 1 & 1234 & 1 \\
    &2134 & 1 & 2134 & 1 \\\hline
    (Case2)&1243 & $q$ & 1243 & $q$  \\
    &2143 & $q$ & 2143 & $q$ \\\hline
    (Case3)&1324 & $q$ & 1324 & $q$ \\
    &3142 & $q$ & 3142 & $q$ \\\hline
    (Case4)&{1342} & {$qa_{n-1}$} & {1342} & {$q\alpha a_{n-1}$} \\
    &{3124} & {$\alpha a_{n-1}$} & {3124} & {$a_{n-1}$}   \\\hline
    (Case5)&{1423} & {$q a_{n-1}$} & {1423} & {$q \alpha a_{n-1}$} \\
    &{4123} & {$\alpha a_{n-1}$} & {4123} & {$a_{n-1}$} \\\hline
    (Case6)&{1432} & {$q^2 a_{n-1}$}  & {1432} & {$q^2 \alpha a_{n-1}$} \\
    &{4132} & {$q \alpha a_{n-1}$} & {4132} & {$q a_{n-1}$} \\\hline
    (Case7)&2314 & $q \alpha a_{n-1}$ & 2341 & $q  \alpha a_{n-1}$ \\
    &3214 & $q \alpha a_{n-1}$ & 3241 & $q \alpha a_{n-1}$ \\\hline
    
    (Case8)&2341 & $q a_{n-1}$ & 2314 & $q  a_{n-1}$ \\
    &3241 & $q a_{n-1}$ & 3214 & $q  a_{n-1}$ \\\hline
    (Case9)&2413 & $q \alpha a_{n-1}^2$ & 2413 & $q \alpha a_{n-1}^2$ \\
    &4231 & $q \alpha a_{n-1}^2$ & 4231 & $q \alpha a_{n-1}^2$ \\\hline

    (Case10)&{2431} & {$q^2 a_{n-1}$} & {2431} & {$q^2 \alpha a_{n-1}$} \\
    &{4213} & {$q \alpha a_{n-1}$} & {4213} &  {$q a_{n-1}$} \\\hline
    (Case11)&3412 & $q \alpha a_{n-1}^2$  & 3412 & $q \alpha a_{n-1}^2$ \\
    &4312 & $q \alpha a_{n-1}^2$ & 4312 & $q \alpha a_{n-1}^2$ \\\hline
    (Case12)&3421 & $q^2 \alpha a_{n-1}^2$ & 3421 & $q^2 \alpha a_{n-1}^2$ \\
    &4321 & $q^2 \alpha a_{n-1}^2$ & 4321 & $q^2 \alpha a_{n-1}^2$ \\\hline
\end{tabular}}
    \caption{Ratios of $\stat$.}
    \label{tab: ratios of statistic}
\end{table}

\textbf{(Construction of $\phi_{n+1}$ from $\phi_{n}$ and $\psi_{n}$)} For the filled diagrams $(\mu^{(n+1)},f_{\mu^{(n+1)}})$ and $(\lambda^{(n+1)},f_{\lambda^{(n+1)}})$, from the definition of $\stat$, \[\dfrac{\stat_{(\mu^{(n+1)},f_{\mu^{(n+1)}})}(w)}{\stat_{(\mu^{(n)},f_{\mu^{(n)}})}(w_{[2n-1]})}=q^{\chi(w_{2n-1}>w_{2n})+\chi(w_{2n}>w_{2n+1})}(\alpha a_{n-1})^{\chi(w_{2n-2}>w_{2n})}(a_{n-1})^{\chi(w_{2n-1}>w_{2n+1})}, \text{ and }\]
\[\dfrac{\stat_{(\lambda^{(n+1)},f_{\lambda^{(n+1)}})}(w)}{\stat_{(\lambda^{(n)},f_{\lambda^{(n)}})}(w_{[2n-1]})}=q^{\chi(w_{2n-1}>w_{2n})+\chi(w_{2n}>w_{2n+1})}( a_{n-1})^{\chi(w_{2n-2}>w_{2n})}(\alpha a_{n-1})^{\chi(w_{2n-1}>w_{2n+1})}\]
where
\begin{align*}
    \chi(P)=\begin{cases*}
        1 \qquad \text{if $P$ is true},\\
        0 \qquad \text{if $P$ is false}.
    \end{cases*}
\end{align*}
In particular, the ratios $\dfrac{\stat_{(\mu^{(n+1)},f_{\mu^{(n+1)}})}(w)}{\stat_{(\mu^{(n)},f_{\mu^{(n)}})}(w_{[2n-1]})}$ and $\dfrac{\stat_{(\lambda^{(n+1)},f_{\lambda^{(n+1)}})}(w)}{\stat_{(\lambda^{(n)},f_{\lambda^{(n)}})}(w_{[2n-1]})}$ only depend on $\std(w_{[2n-2,2n+1]})$; see Table~\ref{tab: ratios of statistic} for the exhaustive list.

Construction of $\phi_{n+1}$ is given in Table~\ref{tab: construction of phi}. Each case of Table~\ref{tab: ratios of statistic} has two possible values for $\std(w_{[2n-2,2n+1]})$, which we denote by $w^{(1)}$ and $w^{(2)}$ in the second column, and $w^{(3)}$ and $w^{(4)}$ in the fourth column. We claim that for each case, $\phi_{n+1}$ sends the elements in 
\[
\{w\in \mathfrak{S}_{2n+1} : \std(w_{[2n-2,2n+1]})=w^{(1)} \text{ or } \std(w_{[2n-2,2n+1]})=w^{(2)}\}
\]
bijectively to the elements in
\[
\{w\in \mathfrak{S}_{2n+1} : \std(w_{[2n-2,2n+1]})=w^{(3)} \text{ or } \std(w_{[2n-2,2n+1]})=w^{(4)}\}.
\]
We will prove the claim for (Case3), as the remaining cases are easier or follow in a similar way.

\begin{table}[h]
    \centering
    \footnotesize{\begin{tabular}{|c|c|}
    \hline
    & $\phi_{n+1}$ \\\hline
    (Case1) & $\phi_{n+1}(w)=(\phi_n(w_{[2n-1]}),w_{2n},w_{2n+1})$ .\\\hline
    (Case2) & $\phi_{n+1}(w)=(\phi_n(w_{[2n-1]}),w_{2n},w_{2n+1})$ .\\\hline
   (Case3) & $
    \phi_{n+1}(w)=\begin{cases}
    (\phi_n(w_{[2n-1]}),r,s)&\text{if $\phi_n(w_{[2n-1]})_{2n-2}<\phi_n(w_{[2n-1]})_{2n-1}$},\\
    (\phi_n(w_{[2n-1]}),s,r)&\text{if $\phi_n(w_{[2n-1]})_{2n-2}>\phi_n(w_{[2n-1]})_{2n-1}$},
    \end{cases} $ \\
    & where $r=\min(w_{2n},w_{2n+1})$ and $s=\max(w_{2n},w_{2n+1})$.\\ \hline
    (Case4) & $
    \phi_{n+1}(w)=\begin{cases}
    (\psi_n(w_{[2n-1]}),s,r)&\text{if $\psi_n(w_{[2n-1]})_{2n-2}<\psi_n(w_{[2n-1]})_{2n-1}$}\\
    (\psi_n(w_{[2n-1]}),r,s),& \text{if $\psi_n(w_{[2n-1]})_{2n-2}>\psi_n(w_{[2n-1]})_{2n-1}$},
    \end{cases} $ \\
    & where $r=\min(w_{2n},w_{2n+1})$ and $s=\max(w_{2n},w_{2n+1})$.\\ \hline
     (Case5) & $\phi_{n+1}(w)=(\psi_n(w_{[2n-1]}),w_{2n},w_{2n+1})$ .\\\hline
     (Case6) & $\phi_{n+1}(w)=(\psi_n(w_{[2n-1]}),w_{2n},w_{2n+1})$ .\\\hline
     (Case7) & $\phi_{n+1}(w)=(\phi_n(w_{[2n-1]}),w_{2n+1},w_{2n})$ .\\\hline
     (Case8) & $\phi_{n+1}(w)=(\phi_n(w_{[2n-1]}),w_{2n+1},w_{2n})$ .\\\hline
     (Case9) & $
    \phi_{n+1}(w)=\begin{cases}
    (\phi_n(w_{[2n-1]}),r,s)&\text{if $\phi_n(w_{[2n-1]})_{2n-2}<\phi_n(w_{[2n-1]})_{2n-1}$},\\
    (\phi_n(w_{[2n-1]}),s,r),& \text{if $\phi_n(w_{[2n-1]})_{2n-2}>\phi_n(w_{[2n-1]})_{2n-1}$},
    \end{cases} $ \\
    & where $r=\min(w_{2n},w_{2n+1})$ and $s=\max(w_{2n},w_{2n+1})$.\\ \hline
    (Case10) & $
    \phi_{n+1}(w)=\begin{cases}
    (\psi_n(w_{[2n-1]}),s,r)&\text{if $\psi_n(w_{[2n-1]})_{2n-2}<\psi_n(w_{[2n-1]})_{2n-1}$},\\
    (\psi_n(w_{[2n-1]}),r,s)&\text{if $\psi_n(w_{[2n-1]})_{2n-2}>\psi_n(w_{[2n-1]})_{2n-1}$},
    \end{cases}$ \\
    & where $r=\min(w_{2n},w_{2n+1})$ and $s=\max(w_{2n},w_{2n+1})$.\\ \hline
     (Case11) & $\phi_{n+1}(w)=(\phi_n(w_{[2n-1]}),w_{2n},w_{2n+1})$. \\\hline
     (Case12) & $\phi_{n+1}(w)=(\phi_n(w_{[2n-1]}),w_{2n},w_{2n+1})$. \\\hline
    \end{tabular}}
    \caption{Construction of $\phi_{n+1}$.}
    \label{tab: construction of phi}
\end{table}

Pick any $w' \in \mathfrak{S}_{2n+1}$ with $\std(w'_{[2n-2,2n+1]}) = 1324$ or $3142$. There exists a word $v$ such that $\phi_n(v) = w'_{[2n-1]}$, as $\phi_n$ is a bijection. Denote $r' = \min(w'_{2n}, w'_{2n+1})$ and $s' = \max(w'_{2n}, w'_{2n+1})$. If $v_{2n-2} < v_{2n-1}$, consider a word $w = (v, r', s')$ which satisfies $\std(w_{[2n-2,2n+1]}) = 1324$. Then we have $w'=\phi_{n+1}(w)$ according to the description of  $\phi_{n+1}$ given in (Case3). If $v_{2n-2} > v_{2n-1}$, a word $w = (v, s', r')$ satisfies $\std(w_{[2n-2,2n+1]}) = 3142$ and maps to $w'$ via $\phi_{n+1}$. Since the cardinalities of the two sets are the same, the claim is proved.

It remains to show that $\phi_{n+1}$ satisfies $(\phi1)$, $(\phi2)$ and $(\phi3)$. Note that the condition $(\phi3)$ is obvious from the construction and the induction hypothesis for $\phi_{n}$ and $\psi_{n}$. We will show that $\phi_{n+1}$ satisfies $(\phi1)$ and $(\phi2)$ for (Case4) and (Case7). The remaining cases follow with a similar argument.

(Case4):
From Table~\ref{tab: ratios of statistic}, we have
\begin{align*}
\stat_{(\mu^{(n+1)},f_{\mu^{(n+1)}})}(w)&=\begin{cases*} a_{n-1} q\stat_{(\mu^{(n)},f_{\mu^{(n)}})}(w_{[2n-1]})\qquad &\text{if $w_{2n-2}<w_{2n-1}$},\\ a_{n-1}\alpha\stat_{(\mu^{(n)},f_{\mu^{(n)}})}(w_{[2n-1]})\qquad &\text{if $w_{2n-2}>w_{2n-1}$},\end{cases*}\\
\stat_{(\lambda^{(n+1)},f_{\lambda^{(n+1)}})}(w)&=\begin{cases*} a_{n-1} q\alpha\stat_{(\lambda^{(n)},f_{\lambda^{(n)}})}(w_{[2n-1]})\qquad &\text{if} $w_{2n-2}<w_{2n-1}$,\\ a_{n-1}\stat_{(\lambda^{(n)},f_{\lambda^{(n)}})}(w_{[2n-1]})\qquad &\text{if} $w_{2n-2}>w_{2n-1}$.\end{cases*}
\end{align*}

In other words, we can write
\begin{align*} \stat_{(\mu^{(n+1)},f_{\mu^{(n+1)}})}(w)&=a_{n-1} \stat^{(1)}_{(\mu^{(n)},f_{\mu^{(n)}})}(w_{[2n-1]}), \text{ and}\\ \stat_{(\lambda^{(n+1)},f_{\lambda^{(n+1)}})}(w)&=a_{n-1} \stat^{(2)}_{(\lambda^{(n)},f_{\lambda^{(n)}})}(w_{[2n-1]}).
\end{align*}
\begin{table}[b]
    \centering
    \footnotesize{\begin{tabular}{|c|c|c|c|c|}
    \hline
    &$\std(w_{[2n-2,2n+1]})$ &   $\frac{\stat^{(1)}_{(\mu^{(n+1)},f_{\mu^{(n+1)}})}(w)}{\stat_{(\mu^{(n)},f_{\mu^{(n)}})}(w_{[2n-1]})}$  &  $\std(w_{[2n-2,2n+1]})$ & $\frac{\stat^{(2)}_{(\lambda^{(n+1)},f_{\lambda^{(n+1)}})}(w)}{\stat_{(\lambda^{(n)},f_{\lambda^{(n)}})}(w_{[2n-1]})}$ \\\hline
    (Case1)&1234 & $q$ & 1243 & $q$ \\
    &2134 & $q$ & 2143 & $q$ \\\hline
    (Case2)&1243 & $q\alpha$ & 1234 & $q\alpha$  \\
    &2143 & $q\alpha$ & 2134 & $q\alpha$ \\\hline
    
    (Case3)&{1324} & {$q^2$} & {1324} & {$q^2\alpha$} \\
    &{3142} & {$q\alpha$} & {3142} & {$q$} \\\hline 
    (Case4)&1342 & $q \alpha a_{n-1}$ & 1342 & {$q \alpha a_{n-1}$} \\
    &{3124} & {$q \alpha a_{n-1}$} & {3124} & {$q \alpha a_{n-1}$} \\\hline
    
    (Case5)&{1423} & {$q^2 a_{n-1}$} & {1432} & {$q^2 \alpha a_{n-1}$} \\
    &{4123} & {$q \alpha a_{n-1}$} & {4132} & {$q a_{n-1}$} \\\hline
    (Case6)&{1432} & {$q^2 \alpha a_{n-1}$}  & {1423} & {$q^2 \alpha^2 a_{n-1}$} \\
    &{4132} & {$q \alpha^2 a_{n-1}$} & {4123} & {$q \alpha a_{n-1}$} \\\hline
    
    (Case7)&2314 & $q^2 \alpha a_{n-1}$ & 2314 & $q^2 \alpha a_{n-1}$ \\
    &3214 & $q^2 \alpha a_{n-1}$ & 3214 & $q^2 \alpha a_{n-1}$ \\\hline
    (Case8)&2341 & $q \alpha a_{n-1}$ & 2341 & $q \alpha  a_{n-1}$ \\
    &3241 & $q \alpha a_{n-1}$ & 3241 & $q \alpha a_{n-1}$ \\\hline
    
    (Case9)&{2413} & {$q^2 \alpha a_{n-1}^2$} & {2413} & {$q^2 \alpha^2 a_{n-1}^2$} \\
    &{4231} & {$q \alpha^2 a_{n-1}^2$} & {4231} & {$q \alpha a_{n-1}^2$} \\\hline
    (Case10)&2431 & $q^2 \alpha a_{n-1}$ & 2431 & $q^2 \alpha a_{n-1}$ \\
    &4213 & $q^2 \alpha a_{n-1}$ & 4213 &  $q^2 \alpha a_{n-1}$ \\\hline
    
    (Case11)&3412 & $q^2 \alpha a_{n-1}^2$  & 3421 & $q^2 \alpha a_{n-1}^2$ \\
    &4312 & $q^2 \alpha a_{n-1}^2$ & 4321 & $q^2 \alpha a_{n-1}^2$ \\\hline
    (Case12)&3421 & $q^2 \alpha^2 a_{n-1}^2$ & 3412 & $q^2 \alpha^2 a_{n-1}^2$ \\
    &4321 & $q^2 \alpha^2 a_{n-1}^2$ & 4312 & $q^2 \alpha^2 a_{n-1}^2$ \\\hline
\end{tabular}}
    \caption{The ratios $\stat^{(i)}/\stat$ for $i=1,2$.}
    \label{tab: ratios of prime statistics}
\end{table}

Thus, we obtain 
\begin{align*}
\stat_{(\mu^{(n+1)},f_{\mu^{(n+1)}})}(w)&=a_{n-1} \stat^{(1)}_{(\mu^{(n)},f_{\mu^{(n)}})}(w_{[2n-1]})=a_{n-1} \stat^{(2)}_{(\lambda^{(n)},f_{\lambda^{(n)}})}(\psi_n(w_{[2n-1]}))\\&=\stat_{(\lambda^{(n+1)},f_{\lambda^{(n+1)}})}(\phi_{n+1}(w)),
\end{align*}
which is the condition $(\phi1)$. Since $\std(w|_{[2n-2,2n+1]})$ is $1342$ or $3124$, we have $|w_{2n-2}-w_{2n-1}|>1$ and $|w_{2n}-w_{2n+1}|>1$. Now the condition $(\phi2)$ follows from Lemma~\ref{lem: ides preserving} (2a) and (2c).

(Case7):
We have
\begin{align*} \stat_{(\mu^{(n+1)},f_{\mu^{(n+1)}})}(w)&=q\alpha a_{n-1} \stat_{(\mu^{(n)},f_{\mu^{(n)}})}(w_{[2n-1]}), \text{ and}\\ \stat_{(\lambda^{(n+1)},f_{\lambda^{(n+1)}})}(w)&=q\alpha a_{n-1} \stat_{(\lambda^{(n)},f_{\lambda^{(n)}})}(w_{[2n-1]}).
\end{align*}
Therefore, we obtain
\begin{align*}
\stat_{(\mu^{(n+1)},f_{\mu^{(n+1)}})}(w)&=q\alpha a_{n-1} \stat_{(\mu^{(n)},f_{\mu^{(n)}})}(w_{[2n-1]})=q\alpha a_{n-1} \stat_{(\lambda^{(n)},f_{\lambda^{(n)}})}(\phi_n(w_{[2n-1]}))\\&=\stat_{(\lambda^{(n+1)},f_{\lambda^{(n+1)}})}(\phi_{n+1}(w))
\end{align*}
that verifies condition $(\phi1)$. The condition $(\phi2)$ follows from Lemma~\ref{lem: ides preserving} (1c).

\textbf{(Construction of $\psi_{n+1}$ from $\phi_{n}$ and $\psi_{n}$)}  Table~\ref{tab: ratios of prime statistics} shows the ratios 
\[
    \dfrac{\stat^{(1)}_{(\mu^{(n+1)},f_{\mu^{(n+1)}})}(w)}{\stat_{(\mu^{(n)},f_{\mu^{(n)}})}(w_{[2n-1]})}\quad \text{ and }\quad
    \dfrac{\stat^{(2)}_{(\lambda^{(n+1)},f_{\lambda^{(n+1)}})}(w)}{\stat_{(\lambda^{(n)},f_{\lambda^{(n)}})}(w_{[2n-1]})}
\]
according to $\std(w_{[2n-2,2n+1]})$.

We will construct $\psi_{n+1}$ for each case in a similar way as we did in the construction for $\phi_{n+1}$. See Table~\ref{tab: Construction psi} for the construction. The proof that $\psi_{n+1}$ satisfies ($\psi1$), ($\psi2$) and ($\psi3$) can be done in a similar way.
\begin{table}[ht]
    \centering
    \footnotesize{\begin{tabular}{|c|c|}
    \hline
    & $\psi_{n+1}$ \\\hline
    (Case1) & $\psi_{n+1}(w)=(\phi_n(w_{[2n-1]}),w_{2n+1},w_{2n})$ \\\hline
    (Case2) & $\psi_{n+1}(w)=(\phi_n(w_{[2n-1]}),w_{2n+1},w_{2n})$ \\\hline
   (Case3) & $
    \psi_{n+1}(w)=\begin{cases}
    (\psi_n(w_{[2n-1]}),r,s)&\text{if $\psi_n(w_{[2n-1]})_{2n-2}<\psi_n(w_{[2n-1]})_{2n-1}$},\\
    (\psi_n(w_{[2n-1]}),s,r)&\text{if $\psi_n(w_{[2n-1]})_{2n-2}>\psi_n(w_{[2n-1]})_{2n-1}$,}
    \end{cases} $ \\
    & where $r=\min(w_{2n},w_{2n+1})$ and $s=\max(w_{2n},w_{2n+1})$.\\ \hline
    (Case4) & $
    \psi_{n+1}(w)=\begin{cases}
    (\phi_n(w_{[2n-1]}),s,r)&\text{if $\phi_n(w_{[2n-1]})_{2n-2}<\phi_n(w_{[2n-1]})_{2n-1}$,}\\
    (\phi_n(w_{[2n-1]}),r,s)&\text{if $\phi_n(w_{[2n-1]})_{2n-2}>\phi_n(w_{[2n-1]})_{2n-1}$,}
    \end{cases} $ \\
    & where $r=\min(w_{2n},w_{2n+1})$ and $s=\max(w_{2n},w_{2n+1})$.\\ \hline
     (Case5) & $\psi_{n+1}(w)=(\psi_n(w_{[2n-1]}),w_{2n+1},w_{2n})$ \\\hline
     (Case6) & $\psi_{n+1}(w)=(\psi_n(w_{[2n-1]}),w_{2n+1},w_{2n})$ \\\hline
     (Case7) & $\psi_{n+1}(w)=(\phi_n(w_{[2n-1]}),w_{2n},w_{2n+1})$ \\\hline
     (Case8) & $\psi_{n+1}(w)=(\phi_n(w_{[2n-1]}),w_{2n},w_{2n+1})$ \\\hline
     (Case9) & $
    \psi_{n+1}(w)=\begin{cases}
    (\psi_n(w_{[2n-1]}),r,s)&\text{if $\psi_n(w_{[2n-1]})_{2n-2}<\phi_n(w_{[2n-1]})_{2n-1}$,}\\
    (\psi_n(w_{[2n-1]}),s,r)&\text{if $\psi_n(w_{[2n-1]})_{2n-2}>\psi_n(w_{[2n-1]})_{2n-1}$,}
    \end{cases} $ \\
    & where $r=\min(w_{2n},w_{2n+1})$ and $s=\max(w_{2n},w_{2n+1})$.\\ \hline
    (Case10) & $
    \psi_{n+1}(w)=\begin{cases}
    (\phi_n(w_{[2n-1]}),s,r)&\text{if $\phi_n(w_{[2n-1]})_{2n-2}<\phi_n(w_{[2n-1]})_{2n-1}$,}\\
    (\phi_n(w_{[2n-1]}),r,s)&\text{if $\phi_n(w_{[2n-1]})_{2n-2}>\phi_n(w_{[2n-1]})_{2n-1}$,}
    \end{cases}$ \\
    & where $r=\min(w_{2n},w_{2n+1})$ and $s=\max(w_{2n},w_{2n+1})$.\\ \hline
     (Case11) & $\psi_{n+1}(w)=(\phi_n(w_{[2n-1]}),w_{2n+1},w_{2n})$ \\\hline
     (Case12) & $\psi_{n+1}(w)=(\phi_n(w_{[2n-1]}),w_{2n+1},w_{2n})$ \\\hline
    \end{tabular}}
    \caption{Construction of $\psi_{n+1}$.}
    \label{tab: Construction psi}
\end{table}

\end{proof}

\subsection{Butler permutations}
In this subsection, we define Butler permutations, which constitute exactly half of all permutations. Butler permutations play a central role in the $F$-expansion of $\I_{\lambda,\mu}[X;q,t]$ (see Theorem~\ref{thm: first main, F-expansion}).

Before we define Butler permutations, we set some terminology for directed perfect matchings. For a directed perfect matching $M=\{(a_1,a_2), \dots, (a_{2n-1},a_{2n})\}$, we say an arc $(a,b)$ is in the \emph{forward direction} if $a<b$ and the \emph{reverse direction} otherwise. An arc $(a,b)$ is \emph{nested} in the other arc $(c,d)$ if $\min\{c,d\}<a,b$ and $\max\{c,d\}>a,b$. We say that a pair of arcs $(a,b)$ and $(c,d) $ is \emph{crossing} if $\min\{a,b\}<\min\{c,d\}<\max\{a,b\}<\max\{c,d\}$ or $\min\{c,d\}<\min\{a,b\}<\max\{c,d\}<\max\{a,b\}$. If not, we say the pair is \emph{noncrossing}.
\begin{definition}
For a permutation $w = w_1 \dots w_n \in \mathfrak{S}_n$ with $n \geq 2$, consider a directed perfect matching $M(w)$ on $\{0, 1, \dots, n+1\}$ if $n$ is even (or on $\{0, 1, \dots, n\}$ if $n$ is odd), with directed arcs $\alpha_1(w), \alpha_2(w), \dots$ as follows:
\begin{enumerate}
\item if $n$ is even,
\begin{itemize}
    \item $\alpha_1(w)$ is a directed arc from $w_{n}$ to $n+1$,
    \item $\alpha_i(w)$ is a directed arc from $w_{n-2(i-1)}$ to $w_{n-2(i-1)+1}$ for $2\leq i\leq \frac{n}{2}$, and
    \item $\alpha_{\frac{n}{2}+1}(w)$ is a directed arc from 0 to $w_1$.
\end{itemize}
\item if $n$ is odd,
\begin{itemize}
    \item $\alpha_i(w)$ is a directed arc from $w_{n-2(i-1)-1}$ to $w_{n-2(i-1)}$ for $1\leq i\leq \frac{n-1}{2}$, and
    \item $\alpha_{\frac{n+1}{2}}(w)$ is a directed arc from 0 to $w_1$.
\end{itemize}
\end{enumerate}
Let $k(w)$ be the smallest integer $k$ such that the pair of arcs $\alpha_k(w)$ and $\alpha_{k+1}(w)$ is noncrossing, if such $k$ exists. We say that $w$ is a \emph{Butler permutation} if either
\begin{itemize}
    \item $\alpha_{k+1}(w)$ is nested in $\alpha_{k}(w)$ and $\alpha_k(w)$ is in the reverse direction, or
    \item $\alpha_{k+1}(w)$ is not nested in $\alpha_{k}(w)$ and $\alpha_k(w)$ is in the forward direction.
\end{itemize}
If there is no such $k$, we let $k(w)$ be $\frac{n}{2}$ when $n$ is even, and $\frac{n-1}{2}$ when $n$ is odd. We say that $w$ is a \emph{Butler permutation} if $\alpha_{k(w)}$ is in the reverse direction. We denote the set of Butler permutations of length $n$ by $\mathfrak{B}_{n}$.
\end{definition}

\begin{example}
We construct $M(w)$ for $w=2413, 3214$, and $1423$. For each $M(w)$, we draw each arc $\alpha_1(w), \alpha_2(w)$, and $\alpha_3(w)$ in red, blue, and green, respectively.
\[  M(2413) = \begin{tikzpicture}[scale=0.7]
      \textcolor{red}{\Matching{3}{5}}\textcolor{blue}{\Matching{4}{1}}\textcolor{green}{\Matching{0}{2}}
      \foreach \i in {0,...,5}{
        \draw [fill] (\i,0) circle [radius=0.075] ;
      }
    \end{tikzpicture},~
    M(3214) = \begin{tikzpicture}[scale=0.7]
      \textcolor{red}{\Matching{4}{5}}\textcolor{blue}{\Matching{2}{1}}\textcolor{green}{\Matching{0}{3}}
      \foreach \i in {0,...,5}{
        \draw [fill] (\i,0) circle [radius=0.075] ;
      }
    \end{tikzpicture},~
    M(1423) = \begin{tikzpicture}[scale=0.7]
      \textcolor{red}{\Matching{3}{5}}\textcolor{blue}{\Matching{4}{2}}\textcolor{green}{\Matching{0}{1}}
      \foreach \i in {0,...,5}{
        \draw [fill] (\i,0) circle [radius=0.075] ;
      }
    \end{tikzpicture}.
\]
For each matching $M(w)$ we find the smallest integer $k$ such that the pair $\alpha_{k}(w)$ and $\alpha_{k+1}(w)$ is noncrossing. For \( M(2413) \), since there is no \( k \) such that \( \alpha_k(2413) \) and \( \alpha_{k+1}(2413) \) are noncrossing, we have \( k(2413) = 2 \). In addition, since \( \alpha_2(2413) \) is in the reverse direction, \( 2413 \) is a Butler permutation.
For \( M(3214) \), we have \( k(3214) = 1 \). Since \( \alpha_2(3214) \) is not nested in \( \alpha_1(3214) \) and \( \alpha_1(3214) \) is in the forward direction, \( 3214 \) is also a Butler permutation.
For \( M(1423) \), we have \( k(1423) = 2 \). Since \( \alpha_3(1423) \) is not nested in \( \alpha_2(1423) \) and \( \alpha_2(1423) \) is in the reverse direction, \( 1423 \) is not a Butler permutation.
\end{example}

The following lemma is direct from the definition.

\begin{lem}\label{lem: Butler recursion}
    For $n\ge 2$, the following statements hold:
     \begin{enumerate}
         \item For a permutation $w\in \mathfrak{S}_{2n+1}$, if the pair of directed arcs $\alpha_1(w)=(w_{2n},$ $w_{2n+1})$ and $\alpha_2(w)=(w_{2n-2},w_{2n-1})$ is crossing, then $w\in\mathfrak{B}_{2n+1}$ if and only if $\std(w_{[2n-1]})\in\mathfrak{B}_{2n-1}$.
         \item  For a permutation $w\in \mathfrak{S}_{2n}$, if the pair of directed arcs $\alpha_1(w)=(w_{2n},2n+1)$ and $\alpha_2(w)=(w_{2n-2},w_{2n-1})$ is crossing, then $w\in\mathfrak{B}_{2n}$ if and only if $\std(w_{[2n-1]})\in\mathfrak{B}_{2n-1}$.
     \end{enumerate}
\end{lem}

It turns out that the number of elements in \( \mathfrak{B}_n \) equals \( \frac{n!}{2} \) which follows from Proposition~\ref{Prop: h21111}. Before proving this, we establish a lemma.

\begin{lem}\label{lem: pieri rule for F}
     For any positive integers $k,m$, consider a subset $A \subseteq [k+m]$ such that $|A|=m$. Then for a permutation $v\in \mathfrak{S}_m$, we have
    \begin{equation*}
        \sum_{\substack{w\in \mathfrak{S}_{k+m}\\\std(w\vert_{A})=v}}F_{\iDes(w)}=(h_1)^{k}F_{\iDes(v)}.
    \end{equation*}
\end{lem}
\begin{proof}
When $A=[m]$ the proof follows from the product formula for fundamental quasisymmetric functions, see \cite[Proposition 5.2.15]{GR20} for example. For any two $A,A'\subseteq [k+m]$ with $|A|=|A'|=m$, we have 
\begin{equation*}
        \sum_{\substack{w\in \mathfrak{S}_{k+m}\\\std(w\vert_{A})=v}}F_{\iDes(w)}= \sum_{\substack{w\in \mathfrak{S}_{k+m}\\\std(w\vert_{A'})=v}}F_{\iDes(w)}.
\end{equation*}
    by \cite[Exercise 3.161]{StanleyEC1-2ed}.
    \end{proof}

\begin{proposition}\label{Prop: h21111}
Let
\begin{equation*}
    \I_n[X]:=\sum_{w\in \mathfrak{B}_{n}} F_{\iDes(w)}[X].
\end{equation*}
Then we have $\I_n[X]=h_{(2,1^{n-2})}[X]$.
\end{proposition}
\begin{proof}
We proceed by induction on $n$. We treat $n=2,3,4$ and $5$ as the base cases. If $n=2$, we have $\mathfrak{B}_2=\{12\}$, which gives $F_{\iDes(12)}=F_{\emptyset}=h_2$. If $n=3$, we have $\mathfrak{B}_3=\{123,231,312\}$, which gives
\begin{equation*}
\sum_{w\in \mathfrak{B}_{3}} F_{\iDes(w)}[X]=F_{\iDes(123)}+F_{\iDes(231)}+F_{\iDes(312)}=F_{\emptyset}+F_{\{1\}}+F_{\{2\}}=h_{(2,1)}.
\end{equation*}
For $n=4$, we list elements of $\mathfrak{B}_4$ sorted by $\iDes(w)$: $\iDes(w)=\emptyset$: 1234, $\iDes(w)=\{1\}$: 2134,2314, $\iDes(w)=\{2\}$: 1324, 3124, 3412, $\iDes(w)=\{3\}$: 1243, 4123, $\iDes(w)=\{1,2\}: 3214$, $\iDes(w)=\{1,3\}: 2413,4213$, $\iDes(w)=\{2,3\}$: 4132. We conclude 
\begin{equation*}
    \sum_{w\in \mathfrak{B}_{4}} F_{\iDes(w)}[X]=F_{\emptyset}+2F_{\{1\}}+3F_{\{2\}}+2F_{\{3\}}+F_{\{1,2\}}+2F_{\{1,3\}}+F_{\{2,3\}}=h_{(2,1,1)}.
\end{equation*}
For $n=5$, we list elements of $\mathfrak{B}_5$ sorted by $\iDes(w)$: $\iDes(w)=\emptyset$: 12345, $\iDes(w)=\{1\}$: 21345,23145,23451 $\iDes(w)=\{2\}$: 13245, 13452, 31245, 34125, 34152, 34512, $\iDes(w)=\{3\}$: 12435, 12453, 14523, 41235, 41523, 45123, $\iDes(w)=\{4\}$: 12534, 15234, 51234, $\iDes(w)=\{1,2\}$: 32145, 32451, 34251, $\iDes(w)=\{1,3\}$: 24135, 24153, 24351, 24513, 42135, 42351, 45213, 45231, $\iDes(w)=\{1,4\}$: 21534, 23514, 23541, 25134, 52134, 52341, $\iDes(w)=\{2,3\}$: 14352, 41325, 41352, 43512, 45312 $\iDes(w)=\{2,4\}$: 13524, 13542, 31524, 35124, 35412, 51324, 51342, 53412, $\iDes(w)=\{3,4\}$: 15423, 51423, 54123, $\iDes(w)=\{1,2,3\}$: 43251, $\iDes(w)=\{1,2,4\}$: 35214, 35241, 53241, $\iDes(w)=\{1,3,4\}$: 25413, 52413, 52431, $\iDes(w)=\{2,3,4\}$: 54312. We conclude 
\begin{align*}
    \sum_{w\in \mathfrak{B}_{5}} F_{\iDes(w)}[X]=F_{\emptyset}+3F_{\{1\}}+6F_{\{2\}}+6F_{\{3\}}+3F_{\{4\}}+3F_{\{1,2\}}+8F_{\{1,3\}}+6 F_{\{1, 4\}}\\ +5F_{\{2,3\}}+8F_{\{2,4\}}+3F_{\{3,4\}}+F_{\{1,2,3\}}+3F_{\{1,2,4\}}+3F_{\{1,3,4\}}+F_{\{2,3,4\}}=h_{(2,1,1,1)}.
\end{align*}

Now let $n\geq 6$. We define
\begin{align*}
    \mathfrak{S}'_n =& \{w\in \mathfrak{S}_n: \text{the pair }\alpha_1(w) \text{ and } \alpha_2(w) \text{ is noncrossing}\}\\
     \mathfrak{S}''_n = &\{w\in \mathfrak{S}_n: \text{the pair }\alpha_1(w) \text{ and } \alpha_2(w) \text{ is crossing}\}
\end{align*}
and let $\mathfrak{B}'_n=\mathfrak{B}_n \cap \mathfrak{S}'_n$ and $\mathfrak{B}''_n=\mathfrak{B}_n \cap \mathfrak{S}''_n$.

As $n-2\geq 4$, $\alpha_1(\std(w_{[3,n]}))$ and $\alpha_2(\std(w_{[3,n]}))$ exist and relative positions and directions of these arcs are identical to those of  $\alpha_1(w)$ and $\alpha_2(w)$. It is direct to see that $w\in\mathfrak{B}'_{n}$ if and only if $\std(w_{[3,n]})\in\mathfrak{B}'_{n-2}$. Thus we have
\begin{equation}\label{eq: recursion for I^(1)}
    h_{(1^2)} \sum_{w\in \mathfrak{B}'_{n-2}} F_{\iDes(w)}[X] = \sum_{w\in \mathfrak{B}'_{n}} F_{\iDes(w)}[X]
\end{equation}
by Lemma~\ref{lem: pieri rule for F}. Similarly, we get
\begin{equation}\label{eq: recursion for S''}
    h_{(1^2)} \sum_{w\in \mathfrak{S}''_{n-2}} F_{\iDes(w)}[X] = \sum_{w\in \mathfrak{S}''_{n}} F_{\iDes(w)}[X].
\end{equation}

On the other hand, there is a map from $\mathfrak{B}''_{n}$ to $\mathfrak{S}''_n\setminus\mathfrak{B}''_n$ sending $w$ to $w'$ obtained from $w$ by reversing the arc $\alpha_{k(w)}(w)$. Clearly, this map is a bijection. Moreover $\alpha_{k(w)-1}(w)$ must form a crossing with $\alpha_{k(w)}(w)$ so $\alpha_{k(w)}(w)$ does not consist of adjacent integers, which implies $\iDes(w)=\iDes(w')$. In particular, we deduce
\[
    2 \sum_{w\in \mathfrak{B}''_{n}} F_{\iDes(w)}[X]  = \sum_{w \in \mathfrak{S}''_n} F_{\iDes(w)}[X].
\]
Together with \eqref{eq: recursion for S''}, this proves a recursive relation, namely
\begin{equation}\label{eq: recursion for I^(2)}
    h_{(1^2)} \sum_{w\in \mathfrak{B}''_{n-2}} F_{\iDes(w)}[X]  = \sum_{w\in \mathfrak{B}''_{n}} F_{\iDes(w)}[X].
\end{equation}

By combining \eqref{eq: recursion for I^(1)} and \eqref{eq: recursion for I^(2)}, we have
\[
    h_{(1^2)} \I_{n-2} [X]= \I_{n}[X],
\]
which completes the proof.
\end{proof}

For the remainder of this subsection, our goal is to prove Proposition~\ref{lem: main LLT lemma}. To state Proposition~\ref{lem: main LLT lemma} we need the following definition. 

\begin{definition}
    For positive integers $n>m$,  We define $\bar{\mathcal{V}}(n,m;\alpha)$ to be the set of all filled diagrams $(\bar{\mu},f_{\bar{\mu}})$ such that   $\bar{\mu}=[[n],[m]\setminus\{1\}]$ and a filling $f_{\bar{\mu}}$ on it satisfying the following conditions:
    \begin{equation}\label{eq: column exchange half condition}
   q^{-1}f_{\bar{\mu}}(m+1,1)=\alpha, \qquad f_{\bar{\mu}}(i,1) = \alpha f_{\bar{\mu}}(i,2), \text{ for  $2<i\leq m$}.  
    \end{equation}
     For $(\bar{\mu},f_{\bar{\mu}})\in \bar{\mathcal{V}}(n,m;\alpha)$, we let $\bar{S}(\bar{\mu},f_{\bar{\mu}})$ be the filled diagram $(\bar{\lambda},f_{\bar{\lambda}})$ such that $\bar{\lambda}=[[m],[n]\setminus\{1\}]$ and:
    \begin{align*}
        &f_{\bar{\lambda}}(i,2)=f_{\bar{\mu}}(i,1), \text{ for  $i>m+1$}\\
         &f_{\bar{\lambda}}(m+1,2)=\alpha\\
         &f_{\bar{\lambda}}(i,1) =f_{\bar{\mu}}(i,2) \text{\hspace{2mm}and\hspace{2mm}}f_{\bar{\lambda}}(i,2)=f_{\bar{\mu}}(i,1) , \text{ for  $2<i\leq m$}.\\
         &\alpha f_{\bar{\lambda}}(2,1)=f_{\bar{\mu}}(2,1).
    \end{align*}
See Figure~\ref{fig: generic filled diagrams half} for a generic $(\bar{\mu},f_{\bar{\mu}})\in\bar{\mathcal{V}}(n,m;\alpha)$ and corresponding $\bar{S}(\bar{\mu},f_{\bar{\mu}})$.

Let $(D,f)$ be a filled diagram such that the restriction to the $j$ and $j+1$-th columns is in $\bar{\mathcal{V}}(n,m;\alpha)$ for some $n, m$ and $\alpha$. We define $\bar{S}_j(D,f)$ to be the filled diagram obtained from $(D,f)$ by applying the map $\bar{S}$ to the $j$ and $j+1$-th columns.
\begin{figure}
\[
    \ytableausetup{boxsize=3.3em}
    (\bar{\mu},f_{\bar{\mu}}) = \begin{ytableau}
     b_{1} \\
     \vdots  \\
     b_{n-m-1} \\
    q\alpha  \\
    \alpha a_1 & a_1 \\
    \vdots & \vdots \\
     \alpha a_{m-2} & a_{m-2}\\
    \alpha  a_{m-1} &  \\
     & \none
    \end{ytableau}
    \qquad \qquad \bar{S}(\bar{\mu},f_{\bar{\mu}})  = \begin{ytableau}
     \none &b_{1} \\
     \none &\vdots  \\
     \none &b_{n-m-1} \\
    \none & \alpha  \\
    a_1 & \alpha a_1 \\
    \vdots & \vdots \\
    a_{m-2} &\alpha a_{m-2}\\
     a_{m-1} &  \\
     & \none
    \end{ytableau}
\]
\caption{Any $(\bar{\mu},f_{\bar{\mu}})\in\bar{\mathcal{V}}(n,m;\alpha)$ is of the form 
 in the left figure for suitable $a_i$'s, $b_i$'s in $\mathbb{F}$. The right figure shows the corresponding $\bar{S}(\bar{\mu},f_{\bar{\mu}})$.}\label{fig: generic filled diagrams half}
\end{figure}
\end{definition}

\begin{proposition}\label{lem: main LLT lemma} For positive integers $n>m$, let $(\bar{\mu},f_{\bar{\mu}})\in \bar{\mathcal{V}}(n,m;\alpha)$ and $(\bar{\lambda},f_{\bar{\lambda}})=\bar{S}(\bar{\mu},f_{\bar{\mu}})$. Then there is a bijection 
\[
    \zeta_{n,m}:\mathfrak{S}_{n+m-1}\rightarrow\mathfrak{S}_{n+m-1}
\]
satisfying the following three conditions, where we denote $w'=\zeta_{n,m}(w)$:
\begin{align*}
    (\zeta1)\quad & \begin{cases}
     \stat_{(\bar{\mu},f_{\bar{\mu}})}(w) = \stat_{(\bar{\lambda},f_{\bar{\lambda}})}(w') \\ \qquad \qquad \text{ for } w\in\mathfrak{S}_{m+n-1} \text{ such that } \std(w|_{[n-m,n+m-1]})\in \mathfrak{B}_{2m}, \text{ and }\\
    \stat_{(\bar{\mu},f_{\bar{\mu}})}(w) = \alpha \stat_{(\bar{\lambda},f_{\bar{\lambda}})}(w')\\ \qquad \qquad \text{ for } w\in\mathfrak{S}_{m+n-1} \text{ such that } \std(w|_{[n-m,n+m-1]})\notin \mathfrak{B}_{2m}
    \end{cases}\\
    (\zeta2)\quad& \iDes(w) = \iDes(w'), \text{ and}\\
    (\zeta3)\quad& 
    \begin{cases}
    \{w_{n-m+2i-1}, w_{n-m+2i}\} = \{w'_{n-m+2i-1}, w'_{n-m+2i}\} \text{ for } 1\le i \le m-1,
     \text{ and } & \\ w_i=w'_i \quad \text{ for } 1\leq i\leq n-m \text{ or } i=n+m-1.&
    \end{cases} 
\end{align*}
In particular, we have
\begin{equation}\label{eq: ff}     \dfrac{\widetilde{H}_{(\bar{\mu},f_{\bar{\mu}})}[X;q,t] - \alpha \widetilde{H}_{(\bar{\lambda},f_{\bar{\lambda}})}[X;q,t]}{1-\alpha} = \sum_{\substack{w\in{\mathfrak{S}_{n+m-1}} \\ \std(w|_{[n-m,n+m-1]}) \in \mathfrak{B}_{2m}}} \stat_{(\bar{\mu},f_{\bar{\mu}})}(w) F_{\iDes(w)}.
\end{equation}
\end{proposition}

We briefly explain how one can deduce \eqref{eq: ff} from the map $\zeta_{n,m}$. 
Let $A$ be the set of $w\in\mathfrak{S}_{m+n-1}$ such that 
$\std(w|_{[n-m,n+m-1]})\in \mathfrak{B}_{2m}$, and let $B$ be the set of 
$w\in\mathfrak{S}_{m+n-1}$ such that 
$\std(w|_{[n-m,n+m-1]})\notin \mathfrak{B}_{2m}$. 
Then, by properties $(\zeta1)$ and $(\zeta2)$, we have
\[
    \sum_{w\in A} \stat_{(\bar{\mu},f_{\bar{\mu}})}(w) F_{\iDes(w)}
    = \sum_{w\in A'} \stat_{(\bar{\lambda},f_{\bar{\lambda}})}(w) F_{\iDes(w)},
\qquad
    \sum_{w\in B} \stat_{(\bar{\mu},f_{\bar{\mu}})}(w) F_{\iDes(w)}
    = \alpha \sum_{w\in B'} \stat_{(\bar{\lambda},f_{\bar{\lambda}})}(w) F_{\iDes(w)},
\]
where $A'$ and $B'$ are the images of $A$ and $B$ under the map $\zeta_{n,m}$. 
Now computing the left-hand side of \eqref{eq: ff} leaves only 
$\sum_{w\in A} \stat_{(\bar{\mu},f_{\bar{\mu}})}(w) F_{\iDes(w)}$.

Let $(D,f)$ be a filled diagram such that the restriction $(E,g)$ to the $j$ and $j+1$-th columns is in $\bar{\mathcal{V}}(n,m;\alpha)$. Then exploiting the map $\zeta=\zeta_{n,m}\uparrow^{D}_E$, the same argument as in the proof of Corollary \ref{cor: column exchange} shows the following: Denoting $w'=\zeta(w),$ we have
\begin{align*}
    (\zeta1)\quad & \begin{cases}
     \stat_{(D,f)}(w) = \stat_{(\bar{S}_j(D,f))}(w') \\ \qquad \qquad \text{ for } w\in\mathfrak{S}_{|D|} \text{ such that } \std\left((w\downarrow_{E}^{D}\bigg|_{[n-m,n+m-1]})\right)\in \mathfrak{B}_{2m}, \text{ and }\\
    \stat_{(D,f)}(w) = \alpha\stat_{(\bar{S}_j(D,f))}(w')\\ \qquad \qquad \text{ for } w\in\mathfrak{S}_{|D|} \text{ such that } \std\left((w\downarrow_{E}^{D}\bigg|_{[n-m,n+m-1]})\right)\notin \mathfrak{B}_{2m}
    \end{cases}\\
    (\zeta2)\quad& \iDes(w) = \iDes(w'), \text{ and}\\
    (\zeta3)\quad& 
    \{w_{a}: a \in C_i\} = \{w'_{a}: a \in C_i\}  \text{  for all  $i$, where $C_i=\{N_{D}(u): \text{$u$ is in the $i$-th row of $D$}\}$}. &
\end{align*}

We conclude
\begin{equation}\label{eq: concatenation2}
    \dfrac{\widetilde{H}_{(D,f)}[X;q,t] - \alpha \widetilde{H}_{\bar{S}_j(D,f)}[X;q,t]}{1-\alpha} = \sum_{w} \stat_{(D,f)}(w)F_{\iDes(w)},
\end{equation}
where the sum is over all $w\in\mathfrak{S}_{|D|}$ such that $\std\left((w\downarrow_{E}^{D}\bigg|_{[n-m,n+m-1]})\right)\in \mathfrak{B}_{2m}$.

To prove Proposition~\ref{lem: main LLT lemma}, we first prove the following lemma as we need two auxiliary maps $\eta_n^{(1)}$ and $\eta_n^{(2)}$ for the construction of $\zeta_{n,m}$.

\begin{lem}\label{lem: auxiliary maps}
Recall the filled diagrams $(\mu^{(n)},f_{\mu^{(n)}})$ and $(\lambda^{(n)},f_{\lambda^{(n)}})$ in Figure~\ref{fig: filled diagrams column exchange lemma}. Then there are bijections
\begin{align*}
    \eta^{(1)}_n: \{w\in\mathfrak{S}_{2n-1}: w_{2n-2}<w_{2n-1}\} \rightarrow \{w\in\mathfrak{S}_{2n-1}: w_{2n-2}<w_{2n-1}\} \\
    \eta^{(2)}_n: \{w\in\mathfrak{S}_{2n-1}: w_{2n-2}>w_{2n-1}\} \rightarrow \{w\in\mathfrak{S}_{2n-1}: w_{2n-2}>w_{2n-1}\} 
\end{align*}
satisfying the following three conditions:
\begin{align*}
    (\eta1) \quad& \begin{cases*}
    \stat_{(\mu^{(n)},f_{\mu^{(n)}})}(w)=\stat_{(\lambda^{(n)},f_{\lambda^{(n)}})}(\eta^{(1)}_n(w)), \text{ if $w\in \mathfrak{B}_{2n-1}$ and $w_{2n-2}<w_{2n-1}$}\\
    \stat_{(\mu^{(n)},f_{\mu^{(n)}})}(w)=\alpha \stat_{(\lambda^{(n)},f_{\lambda^{(n)}})}(\eta^{(1)}_n(w)), \text{ if $w\notin \mathfrak{B}_{2n-1}$ and $w_{2n-2}<w_{2n-1}$}
    \end{cases*}\\
    & \begin{cases*}
    \alpha\stat_{(\mu^{(n)},f_{\mu^{(n)}})}(w)=\stat_{(\lambda^{(n)},f_{\lambda^{(n)}})}(\eta^{(2)}_n(w)), \text{ if $w\in \mathfrak{B}_{2n-1}$ and $w_{2n-2}>w_{2n-1}$}\\
    \stat_{(\mu^{(n)},f_{\mu^{(n)}})}(w)=\stat_{(\lambda^{(n)},f_{\lambda^{(n)}})}(\eta^{(2)}_n(w)),  \text{ if $w\notin \mathfrak{B}_{2n-1}$ and $w_{2n-2}>w_{2n-1}$}
    \end{cases*}\\
    (\eta2)\quad&\iDes(w)=\iDes(\eta^{(1)}_n(w)) \text{ and } \iDes(w)=\iDes(\eta^{(2)}_n(w))\\
    (\eta3)\quad&\text{denoting }w'=\eta^{(j)}(w), \text{ we have }
    \\
    &\begin{cases}
    \{w_{2i}, w_{2i+1}\} = \{w'_{2i}, w'_{2i+1}\} \quad \text{ for } 1\le i \le n-1,\text{ and}\\
    w_1=w'
    _1 \quad 
    \end{cases}\\&\text{ when  } j=1,2.
\end{align*}
\end{lem}
\begin{proof}
    We proceed by induction on $n$. For the base case $n=2$, the maps $\eta_2^{(1)}$ and $\eta_2^{(2)}$ are given as the identity maps. 
The tables below show that they satisfy condition $(\eta1)$. 
Note that conditions $(\eta2)$ and $(\eta3)$ are trivial to check. 

Assume that we have constructed the desired $\eta_n^{(1)}$ and $\eta_n^{(2)}$. 
Our goal is to construct $\eta_{n+1}^{(1)}$ and $\eta_{n+1}^{(2)}$, building upon $\eta_n^{(1)}$ and $\eta_n^{(2)}$ together with the maps $\phi_n$ and $\psi_n$ in the proof of Proposition~\ref{lem: column exchange}.

\begin{table}[h]
\centering
\begin{tabular}{|c|c|c|c|}
\hline
$w$ & $\stat_{(\mu^{(2)},f_{\mu^{(2)}})}(w)$ & $\stat_{(\lambda^{(2)},f_{\lambda^{(2)}})}(\eta^{(1)}_2(w))$ & $w\in \mathfrak{B}_3$\\
\hline
123 & 1 & 1 & yes \\ \hline
213 & $\alpha q$ & $q$ & no \\ \hline
312 & $\alpha q$ & $\alpha q$ & yes \\ \hline
\end{tabular}
\end{table}

\begin{table}[h]
\centering
\begin{tabular}{|c|c|c|c|}
\hline
$w$ & $\stat_{(\mu^{(2)},f_{\mu^{(2)}})}(w)$ & $\stat_{(\lambda^{(2)},f_{\lambda^{(2)}})}(\eta^{(2)}_2(w))$ & $w\in \mathfrak{B}_3$\\
\hline
132 & $q$ & $q$ & no \\ \hline
231 & $q$ & $\alpha q$ & yes \\ \hline
321 & $\alpha q^2$ & $\alpha q^2$ & no \\ \hline
\end{tabular}
\end{table}

    \textbf{(Construction of $\eta^{(1)}_{n+1}$ from $\eta^{(1)}_{n}$ ,$\eta^{(2)}_{n}$, $\phi_n$ and $\psi_n$)}
   Table~\ref{tab: ratios of statistic eta1} shows the ratios
\[
\dfrac{\stat_{(\mu^{(n+1)},f_{\mu^{(n+1)}})}(w)}{\stat_{(\mu^{(n)},f_{\mu^{(n)}})}(w_{[2n-1]})} \qquad \text{and} \qquad \dfrac{\stat_{(\lambda^{(n+1)},f_{\lambda^{(n+1)}})}(w)}{\stat_{(\lambda^{(n)},f_{\lambda^{(n)}})}(w_{[2n-1]})}
\]
according to $\std(w_{[2n-2,2n+1]})$, under the assumption that $w_{2n} < w_{2n+1}$. Table~\ref{tab: eta1 construction} describes the construction of $\eta^{(1)}_{n+1}$ for each case categorized in Table~\ref{tab: ratios of statistic eta1}. 

Condition $(\eta2)$ follows from Lemma~\ref{lem: ides preserving} (1a) and (2a). Condition $(\eta3)$ holds directly from the induction hypothesis and the construction. It remains to prove condition $(\eta1)$, which we will do for (Case1), (Case2), (Case3), and (Case4). The remaining cases can be treated in a similar manner.

    \begin{table}[h]
    \centering
    \footnotesize{\begin{tabular}{|c|c|c|c|c|c|}
    \hline
    &$\std(w_{[2n-2,2n+1]})$ &   $\frac{\stat_{(\mu^{(n+1)},f_{\mu^{(n+1)}})}(w)}{\stat_{(\mu^{(n)},f_{\mu^{(n)}})}(w_{[2n-1]})}$  &  $\std(w_{[2n-2,2n+1]})$ & $\frac{\stat_{(\lambda^{(n+1)},f_{\lambda^{(n+1)}})}(w)}{\stat_{(\lambda^{(n)},f_{\lambda^{(n)}})}(w_{[2n-1]})}$ 
    \\\hline
    (Case1) &1234 & 1 & 1234 & 1 \\
    &2134 & 1 & 2134 & 1 \\\hline
    (Case2) &1324 & $q$ & 1324 & $q$  \\\hline
    (Case3) &3124 & $\alpha a_{n-1}$ & 3124 & $a_{n-1}$  \\\hline
    (Case4) &1423 & $q a_{n-1}$ & 1423 & $q\alpha a_{n-1}$ \\
    &4123 & $\alpha a_{n-1}$ & 4123 & $a_{n-1}$ \\\hline
    (Case5) &2314 & $q \alpha a_{n-1}$ & 2314 & $q  a_{n-1}$   \\
    &3214 & $q \alpha a_{n-1}$ & 3214 & $q a_{n-1}$ \\\hline
    (Case6) &2413 & $q \alpha a^{2}_{n-1}$ & 2413 & $q \alpha a^{2}_{n-1}$  \\\hline
    (Case7) &4213 & $q \alpha a_{n-1}$ & 4213 & $q a_{n-1}$  \\\hline
    (Case8) &3412 & $q \alpha a^{2}_{n-1}$ & 3412 & $q \alpha a^{2}_{n-1}$ \\
    &4312 & $q \alpha a^{2}_{n-1}$ & 4312 & $q \alpha a^{2}_{n-1}$ \\\hline
\end{tabular}}
    \caption{Ratios of $\stat$ when $w_{2n}<w_{2n+1}$}
    \label{tab: ratios of statistic eta1}
\end{table}

\begin{table}[h]
    \centering
    \begin{tabular}{|c|c|}
    \hline
    & $\eta^{(1)}_{n+1}$ \\\hline
    (Case1) & $\eta^{(1)}_{n+1}(w)=(\phi_n(w_{[2n-1]}),w_{2n},w_{2n+1})$ \\\hline
    (Case2) & $\eta^{(1)}_{n+1}(w)=(\eta^{(1)}_n(w_{[2n-1]}),w_{2n},w_{2n+1})$ \\\hline
    (Case3) & $\eta^{(1)}_{n+1}(w)=(\eta^{(2)}_n(w_{[2n-1]}),w_{2n},w_{2n+1})$ \\\hline
    (Case4) & $\eta^{(1)}_{n+1}(w)=(\psi_n(w_{[2n-1]}),w_{2n},w_{2n+1})$ \\\hline
    (Case5) & $\eta^{(1)}_{n+1}(w)=(\phi_n(w_{[2n-1]}),w_{2n},w_{2n+1})$ \\\hline
    (Case6) & $\eta^{(1)}_{n+1}(w)=(\eta^{(1)}_n(w_{[2n-1]}),w_{2n},w_{2n+1})$ \\\hline
    (Case7) & $\eta^{(1)}_{n+1}(w)=(\eta^{(2)}_n(w_{[2n-1]}),w_{2n},w_{2n+1})$ \\\hline
    (Case8) & $\eta^{(1)}_{n+1}(w)=(\phi_n(w_{[2n-1]}),w_{2n},w_{2n+1})$ \\\hline
\end{tabular}
    \caption{Construction of $\eta^{(1)}$.}
    \label{tab: eta1 construction}
\end{table}

    (Case1): We have
    \begin{align*}
    \stat_{(\mu^{(n+1)},f_{\mu^{(n+1)}})}(w)&=\stat_{(\mu^{(n)},f_{\mu^{(n)}})}(w_{[2n-1]})=\stat_{(\lambda^{(n)},f_{\lambda^{(n)}})}(\phi_n(w_{[2n-1]}))\\&=\stat_{(\lambda^{(n+1)},f_{\lambda^{(n+1)}})}(\eta^{(1)}_{n+1}(w)),
\end{align*}
and $\std(w_{[2n-2,2n+1]})=1234 \text{ or } 2134$ implies $w\in \mathfrak{B}_{2n+1}$.

    (Case2): We have \begin{equation*}\frac{ \stat_{(\mu^{(n+1)},f_{\mu^{(n+1)}})}(w)}{\stat_{(\lambda^{(n+1)},f_{\lambda^{(n+1)}})}(\eta^{(1)}_{n+1}(w))}=\frac{q \stat_{(\mu^{(n)},f_{\mu^{(n)}})}(w_{[2n-1]})}{q\stat_{(\lambda^{(n)},f_{\lambda^{(n)}})}(\eta^{(1)}_{n}(w_{[2n-1]}))}=\frac{ \stat_{(\mu^{(n)},f_{\mu^{(n)}})}(w_{[2n-1]})}{\stat_{(\lambda^{(n)},f_{\lambda^{(n)}})}(\eta^{(1)}_{n}(w_{[2n-1]}))}\end{equation*}
    and Lemma~\ref{lem: Butler recursion} (1) implies $w\in \mathfrak{B}_{2n+1}$ if and only if $\std(w_{[2n-1]})\in \mathfrak{B}_{2n-1}$. Therefore the condition $(\eta1)$ follows from the induction hypothesis $(\eta1)$ for $\eta^{(1)}_n$.

    (Case3): We have $\frac{ \stat_{(\mu^{(n+1)},f_{\mu^{(n+1)}})}(w)}{\stat_{(\lambda^{(n+1)},f_{\lambda^{(n+1)}})}(\eta^{(1)}_{n+1}(w))}=\frac{ \alpha\stat_{(\mu^{(n)},f_{\mu^{(n)}})}(w_{[2n-1]})}{\stat_{(\lambda^{(n)},f_{\lambda^{(n)}})}(\eta^{(2)}_{n}(w_{[2n-1]}))}$,
    and Lemma~\ref{lem: Butler recursion} (1) implies $w\in \mathfrak{B}_{2n+1}$ if and only if $\std(w_{[2n-1]})\in \mathfrak{B}_{2n-1}$. Therefore the condition $(\eta1)$ follows from the induction hypothesis $(\eta1)$ for $\eta^{(2)}_n$.

     \begin{table}[h]
    \centering
    \footnotesize{\begin{tabular}{|c|c|c|c|c|c|}
    \hline
    &$\std(w_{[2n-2,2n+1])}$ &   $\frac{\stat_{(\mu^{(n+1)},f_{\mu^{(n+1)}})}(w)}{\stat_{(\mu^{(n)},f_{\mu^{(n)}})}(w_{[2n-1]})}$  &  $\std(w_{[2n-2,2n+1]})$ & $\frac{\stat_{(\lambda^{(n+1)},f_{\lambda^{(n+1)}})}(w)}{\stat_{(\lambda^{(n)},f_{\lambda^{(n)}})}(w_{[2n-1]})}$ \\\hline
    (Case1) &1243 & $q$ & 1243 & $q$ \\
    &2143 & $q$ & 2143 & $q$ \\\hline
    (Case2) &1342 & $q a_{n-1}$ & 1342 & $q\alpha a_{n-1}$  \\\hline
    (Case3) &3142 & $q$ & 3142 & $q$  \\\hline
    (Case4) &1432 & $q^2 a_{n-1}$ & 1432 & $q^2 \alpha a_{n-1}$ \\
    &4132 & $q \alpha a_{n-1}$ & 4132 & $q a_{n-1}$ \\\hline
    (Case5) &2341 & $q  a_{n-1}$ & 2341 & $q \alpha a_{n-1}$   \\
    &3241 & $q  a_{n-1}$ & 3241 & $q \alpha a_{n-1}$ \\\hline
    (Case6) &2431 & $q^2  a_{n-1}$ & 2431 & $q^2 \alpha a_{n-1}$  \\\hline
    (Case7) &4231 & $q \alpha a^{2}_{n-1}$ & 4231 & $q \alpha a^{2}_{n-1}$  \\\hline
    (Case8) &3421 & $q^2 \alpha a^{2}_{n-1}$ & 3421 & $q^2 \alpha a^{2}_{n-1}$ \\
    &4321 & $q^2 \alpha a^{2}_{n-1}$ & 4321 & $q^2 \alpha a^{2}_{n-1}$ \\\hline
\end{tabular}}
    \caption{Ratios of $\stat$ when $w_{2n}>w_{2n+1}$.}
    \label{tab: ratios of statistic eta2}
\end{table}

     (Case4): Recall the definition of $\stat^{(1)}$ and $\stat^{(2)}$ in \eqref{eq: stat1 stat2 definition}. 
     We have
    \begin{align*}
        \frac{ \stat_{(\mu^{(n+1)},f_{\mu^{(n+1)}})}(w)}{\stat_{(\lambda^{(n+1)},f_{\lambda^{(n+1)}})}(\eta^{(1)}_{n+1}(w))}=\frac{a_{n-1} \stat^{(1)}_{(\mu^{(n)},f_{\mu^{(n)}})}(w_{[2n-1]})}{a_{n-1}\stat^{(2)}_{(\lambda^{(n)},f_{\lambda^{(n)}})}(\psi_{n}(w_{[2n-1]}))}=1,
    \end{align*}
   and $\std(w_{[2n-2,2n+1]})=1423 \text{ or } 4123$ implies $w\in \mathfrak{B}_{2n+1}$.

     \textbf{(Construction of $\eta^{(2)}_{n+1}$ from $\eta^{(1)}_{n}$ ,$\eta^{(2)}_{n}$, $\phi_n$ and $\psi_n$)}
Table~\ref{tab: ratios of statistic eta2} shows \\$\dfrac{\stat_{(\mu^{(n+1)},f_{\mu^{(n+1)}})}(w)}{\stat_{(\mu^{(n)},f_{\mu^{(n)}})}(w_{[2n-1]})}$ and $\dfrac{\stat_{(\lambda^{(n+1)},f_{\lambda^{(n+1)}})}(w)}{\stat_{(\lambda^{(n)},f_{\lambda^{(n)}})}(w_{[2n-1]})}$
    according to $\std(w_{[2n-2,2n+1]})$ such that $w_{2n}>w_{2n+1}$. Table~\ref{tab: eta2 construction} shows the construction of $\eta^{(2)}_{n+1}$ for each case specified in Table~\ref{tab: ratios of statistic eta2}. One can similarly prove that $\eta^{(2)}_{n+1}$ satisfies $(\eta1)$, $(\eta2)$ and $(\eta3)$.

\begin{table}[h]
    \centering
    \begin{tabular}{|c|c|}
    \hline
    & $\eta^{(2)}_{n+1}$ \\\hline
    (Case1) & $\eta^{(2)}_{n+1}(w)=(\phi_n(w_{[2n-1]}),w_{2n},w_{2n+1})$ \\\hline
    (Case2) & $\eta^{(2)}_{n+1}(w)=(\eta^{(1)}_n(w_{[2n-1]}),w_{2n},w_{2n+1})$ \\\hline
    (Case3) & $\eta^{(2)}_{n+1}(w)=(\eta^{(2)}_n(w_{[2n-1]}),w_{2n},w_{2n+1})$ \\\hline
    (Case4) & $\eta^{(2)}_{n+1}(w)=(\psi_n(w_{[2n-1]}),w_{2n},w_{2n+1})$ \\\hline
    (Case5) & $\eta^{(2)}_{n+1}(w)=(\phi_n(w_{[2n-1]}),w_{2n},w_{2n+1})$ \\\hline
    (Case6) & $\eta^{(2)}_{n+1}(w)=(\eta^{(1)}_n(w_{[2n-1]}),w_{2n},w_{2n+1})$ \\\hline
    (Case7) & $\eta^{(2)}_{n+1}(w)=(\eta^{(2)}_n(w_{[2n-1]}),w_{2n},w_{2n+1})$ \\\hline
    (Case8) & $\eta^{(2)}_{n+1}(w)=(\phi_n(w_{[2n-1]}),w_{2n},w_{2n+1})$ \\\hline
\end{tabular}
    \caption{Construction of $\eta^{(2)}_{n+1}$.}
    \label{tab: eta2 construction}
\end{table}    
     
\end{proof}

\begin{proof}[Proof of Proposition~\ref{lem: main LLT lemma}]
It suffices to show the claim when $m$ and $n$ differ by one. Let $(\bar{\mu}^{(n)},f_{\bar{\mu}^{(n)}})$ and $(\bar{\lambda}^{(n)}, f_{\bar{\lambda}^{(n)}})$ be the filled diagrams as depicted below, where $a_i\in\mathbb{F}$ for $i=1,2,\dots,n-1$.
\[
    \ytableausetup{boxsize=3em}
    (\bar{\mu}^{(n)},f_{\bar{\mu}^{(n)}}) = \begin{ytableau}
    q\alpha  \\
    \alpha a_1 & a_1 \\
    \alpha a_2 & a_2 \\
    \vdots & \vdots \\
    \alpha a_{n-2} & a_{n-2}\\
    \alpha  a_{n-1} &  \\
     & \none
    \end{ytableau}
    \qquad \qquad (\bar{\lambda}^{(n)}, f_{\bar{\lambda}^{(n)}}) = \begin{ytableau}
    \none & \alpha  \\
    a_1 & \alpha a_1 \\
    a_2 & \alpha a_2 \\
    \vdots & \vdots \\
    a_{n-2}& \alpha a_{n-2}\\
     a_{n-1} & \\
     & \none
    \end{ytableau}
\]

It is enough to construct a bijection $\zeta_n:=\zeta_{n+1,n}:\mathfrak{S}_{2n}\rightarrow \mathfrak{S}_{2n}$ satisfying the following three conditions, where we denote  $w'=\zeta_{n}(w)$:
\begin{align*}
    (\zeta1)\quad & \begin{cases}
     \stat_{(\bar{\mu}^{(n)},f_{\bar{\mu}^{(n)}})}(w) = \stat_{(\bar{\lambda}^{(n)},f_{\bar{\lambda}^{(n)}})}(w'), \text{ for } w\in\mathfrak{S}_{2n} \text{ such that } w\in \mathfrak{B}_{2n}\\
    \stat_{(\bar{\mu}^{(n)},f_{\bar{\mu}}^{(n)})}(w) = \alpha \stat_{(\bar{\lambda}^{(n)},f_{\bar{\lambda}}^{(n)})}(w') \text{ for } w\in\mathfrak{S}_{2n} \text{ such that } w\notin \mathfrak{B}_{2n},
    \end{cases}\\
    (\zeta2)\quad& \iDes(w) = \iDes(w'), \text{ and}\\
    (\zeta3)\quad& 
    \begin{cases}
    \{w_{2i}, w_{2i+1}\} = \{w'_{2i}, w'_{2i+1}\} \text{ for } 1\le i \le n-1 \text{ and }\\ w_1=w'_1 \text{ and }w_{2n}=w'_{2n}.
    \end{cases}
\end{align*}

For the case $n=1$, we take $\zeta_1 : \mathfrak{S}_2 \rightarrow \mathfrak{S}_2$ to be the identity map, and it is straightforward to check that it satisfies the above three conditions. Now assume $n>1$ and recall the filled diagrams $(\mu^{(n)},f_{\mu^{(n)}})$ and $(\lambda^{(n)},f_{\lambda^{(n)}})$
in Figure~\ref{fig: filled diagrams column exchange lemma}. We have
\begin{align*}
    \dfrac{\stat_{(\bar{\mu}^{(n)},f_{\bar{\mu}^{(n)}})}(w)}{\stat_{(\mu^{(n)},f_{\mu^{(n)}})}(w_{[2n-1]})}&=q^{\chi(w_{2n-1}>w_{2n})}(\alpha a_{n-1})^{\chi(w_{2n-2}>w_{2n})}\\
    \dfrac{\stat_{(\bar{\lambda}^{(n)},f_{\bar{\lambda}^{(n)}})}(w)}{\stat_{(\lambda^{(n)},f_{\lambda^{(n)}})}(w_{[2n-1]})}&=q^{\chi(w_{2n-1}>w_{2n})}(a_{n-1})^{\chi(w_{2n-2}>w_{2n})}.
\end{align*}
Table~\ref{tab: ratios of stat zeta} shows the ratios $  \dfrac{\stat_{(\bar{\mu}^{(n)},f_{\bar{\mu}^{(n)}})}(w)}{\stat_{(\mu^{(n)},f_{\mu^{(n)}})}(w_{[2n-1]})}$ and $ \dfrac{\stat_{(\bar{\lambda}^{(n)},f_{\bar{\lambda}^{(n)}})}(w)}{\stat_{(\lambda^{(n)},f_{\lambda^{(n)}})}(w_{[2n-1]})}$ according to $\std(w_{[2n-2,2n]})$.
\begin{table}[h]
    \centering
    \footnotesize{\begin{tabular}{|c|c|c|c|c|c|}
    \hline
    &$\std(w_{[2n-2,2n]})$ &   $\frac{\stat_{(\bar{\mu}^{(n)},f_{\bar{\mu}^{(n)}})}(w)}{\stat_{(\mu^{(n)},f_{\mu^{(n)}})}(w_{[2n-1]})}$  &  $\std(w_{[2n-2,2n]})$ & $\frac{\stat_{(\bar{\lambda}^{(n)},f_{\bar{\lambda}^{(n)}})}(w)}{\stat_{(\lambda^{(n)},f_{\lambda^{(n)}})}(w_{[2n-1]})}$ \\\hline
    (Case1) &123 & $1$ & 123 & $1$ \\
    &213 & $1$ & 213 & $1$ \\\hline
    (Case2) &132 & $q $ & 132 & $q$  \\\hline
    (Case3) &312 & $\alpha a_{n-1}$ & 312 & $a_{n-1}$  \\\hline
    (Case4) &231 & $q \alpha a_{n-1}$ & 231 & $q a_{n-1}$  \\
    &321 & $q \alpha a_{n-1}$  & 321 & $q a_{n-1}$  \\\hline
\end{tabular}}
    \caption{Ratios of $\stat$.}
    \label{tab: ratios of stat zeta}
\end{table}
Recall the maps $\phi_{n}$ in Proposition~\ref{lem: column exchange} and $\eta_n^{(1)}$ and $\eta_n^{(2)}$ in Lemma~\ref{lem: auxiliary maps}. Table~\ref{tab: construction zeta} shows the construction of $\zeta_n$ for each case specified in Table~\ref{tab: ratios of stat zeta}.
\begin{table}[h]
    \centering
    \begin{tabular}{|c|c|}
    \hline
    & $\zeta_{n}$ \\\hline
    (Case1) & $\zeta_{n}(w)=(\phi_n(w_{[2n-1]}),w_{2n})$ \\\hline
    (Case2) & $\zeta_{n}(w)=(\eta^{(1)}_n(w_{[2n-1]}),w_{2n})$ \\\hline
    (Case3) & $\zeta_{n}(w)=(\eta^{(2)}_n(w_{[2n-1]}),w_{2n})$ \\\hline
    (Case4) & $\zeta_{n}(w)=(\phi_n(w_{[2n-1]}),w_{2n})$ \\\hline
\end{tabular}
    \caption{Construction of $\zeta_n$.}
    \label{tab: construction zeta}
\end{table}
The condition $(\zeta3)$ follows from $(\phi3)$ for $\phi_n$ and $(\eta3)$ for $\eta^{(1)}_n$ and $\eta^{(2)}_n$. The condition $(\zeta2)$ follows easily from the fact that the maps $\phi_n$, $\eta^{(1)}_n$ and $\eta^{(2)}_n$ preserve $\iDes$. Now we show that $\zeta_n$ satisfies $(\zeta1)$.

  (Case1): We have
    \begin{align*}
    \stat_{(\bar{\mu}^{(n)},f_{\bar{\mu}^{(n)}})}(w)&=\stat_{(\mu^{(n)},f_{\mu^{(n)}})}(w_{[2n-1]})=\stat_{(\lambda^{(n)},f_{\lambda^{(n)}})}(\phi_n(w_{[2n-1]}))\\&=\stat_{(\bar{\lambda}^{(n)},f_{\bar{\lambda}^{(n)}})}(\zeta_{n}(w)),
    \end{align*}
and $\std(w_{[2n-2,2n]})=123 \text{ or } 213$ implies $w\in \mathfrak{B}_{2n}$.

    (Case2): We have
    \begin{align*}
        \frac{ \stat_{(\bar{\mu}^{(n)},f_{\bar{\mu}^{(n)}})}(w)}{\stat_{(\bar{\lambda}^{(n)},f_{\bar{\lambda}^{(n)}})}(\zeta_{n}(w))}=\frac{q \stat_{(\mu^{(n)},f_{\mu^{(n)}})}(w_{[2n-1]})}{q\stat_{(\lambda^{(n)},f_{\lambda^{(n)}})}(\eta^{(1)}_{n}(w_{[2n-1]}))}=\frac{ \stat_{(\mu^{(n)},f_{\mu^{(n)}})}(w_{[2n-1]})}{\stat_{(\lambda^{(n)},f_{\lambda^{(n)}})}(\eta^{(1)}_{n}(w_{[2n-1]}))},
    \end{align*}
    and Lemma~\ref{lem: Butler recursion} (2) implies $w\in \mathfrak{B}_{2n}$ if and only if $\std(w_{[2n-1]})\in \mathfrak{B}_{2n-1}$. Therefore the condition $(\zeta1)$ follows from  $(\eta1)$ for $\eta^{(1)}_n$.

    (Case3): We have
    \begin{align*}
        \frac{ \stat_{(\bar{\mu}^{(n)},f_{\bar{\mu}^{(n)}})}(w)}{\stat_{(\bar{\lambda}^{(n)},f_{\bar{\lambda}^{(n)}})}(\zeta_{n}(w))}=\frac{ \alpha\stat_{(\mu^{(n)},f_{\mu^{(n)}})}(w_{[2n-1]})}{\stat_{(\lambda^{(n)},f_{\lambda^{(n)}})}(\eta^{(2)}_{n}(w_{[2n-1]}))},
    \end{align*}
    and Lemma~\ref{lem: Butler recursion} (2) implies $w\in \mathfrak{B}_{2n}$ if and only if $\std(w_{[2n-1]})\in \mathfrak{B}_{2n-1}$. Therefore the condition $(\zeta1)$ follows from $(\eta1)$ for $\eta^{(2)}_n$.
     
     (Case4):  We have
    \begin{align*}
        \frac{ \stat_{(\bar{\mu}^{(n)},f_{\bar{\mu}^{(n)}})}(w)}{\stat_{(\bar{\lambda}^{(n)},f_{\bar{\lambda}^{(n)}})}(\zeta_{n}(w))}=\frac{\alpha \stat_{(\mu^{(n)},f_{\mu^{(n)}})}(w_{[2n-1]})}{\stat_{(\lambda^{(n)},f_{\lambda^{(n)}})}(\phi_{n}(w_{[2n-1]}))}=\alpha,
    \end{align*}
and $\std(w_{[2n-2,2n]})=231 \text{ or } 321$ implies $w\notin \mathfrak{B}_{2n}$.

\end{proof}

\section{Combinatorial formula for $\operatorname{I}_{\lambda,\mu}[X;q,t]$}\label{Sec: proof of main theorem}

\subsection{Proof of Theorem~\ref{thm: first main, F-expansion}}
Let $\nu$ be a partition and $\mu, \lambda \subseteq \nu$ be two distinct partitions such that $|\nu/\lambda| = |\nu/\mu| = 1$. Without loss of generality, we may assume that $\lambda \trianglerighteq \mu$, i.e., $\mu$ is the partition obtained from $\lambda$ by moving a single cell to an upper row. Then there exist indices $i < j$ such that $\lambda$ and $\mu$ are of the form
\begin{align*}
    \mu &= [[a_1],[a_2],\dots,[a_{j-1}], [a_j-1], [a_{j+1}],\dots,[a_\ell]], \\
    \lambda &= [[a_1],[a_2],\dots,[a_{i-1}], [a_i-1], [a_{i+1}],\dots,[a_\ell]].
\end{align*}
We first deform $(\mu, f^{\st}_{\mu})$ and $(\lambda, f^{\st}_{\lambda})$ by applying the operators $S_k$ and $\cycling$ to obtain $\mathfrak{D}_{\lambda,\mu}(\mu, f^{\st}_{\mu})$ and $\mathfrak{D}_{\lambda,\mu}(\lambda, f^{\st}_{\lambda})$ (Definition~\ref{def: deformation}). Then we show that $\mathfrak{D}_{\lambda,\mu}(\mu, f^{\st}_{\mu})$ and $\mathfrak{D}_{\lambda,\mu}(\lambda, f^{\st}_{\lambda})$ are related by the operator $\bar{S}_{\ell-1}$ (Lemma~\ref{lem: deformed shape is nice}). Finally, we prove Theorem~\ref{thm: first main, F-expansion}.

\begin{definition}\label{def: deformation}
   Let $\lambda$ and $\mu$ be partitions of the form
\begin{align*}
    \mu &= [[a_1],[a_2],\dots,[a_{j-1}],[a_j-1],[a_{j+1}],\dots,[a_\ell]], \\
    \lambda &= [[a_1],[a_2],\dots,[a_{i-1}],[a_i-1],[a_{i+1}],\dots,[a_\ell]],
\end{align*}
for some $i < j$. We define $\mathfrak{D}_{\lambda,\mu}(\mu,f^{\st}_{\mu})$ to be the filled diagram obtained by performing the following steps to $(\mu,f^{\st}_{\mu})$:
\begin{itemize}
    \item (Step 1) Apply the sequence of operators $(S_{\ell-1} \dots S_{i+1})(S_1 \dots S_{j-1})$ to $(\mu,f^{\st}_{\mu})$. Here, operators are applied right-to-left. In other words, move the $j$-th column of $(\mu,f^{\st}_{\mu})$ to the far left, and move the $(i+1)$-th column (which was originally the $i$-th column of $(\mu,f^{\st}_{\mu})$) to the far right.
    \item (Step 2) Apply the operator $\cycling$.
\end{itemize}

Similarly, we define $\mathfrak{D}_{\lambda,\mu}(\lambda,f^{\st}_{\lambda})$ to be the filled diagram obtained by performing the following steps to $(\lambda,f^{\st}_{\lambda})$:
\begin{itemize}
    \item (Step 1) Apply the sequence of operators $(S_1 \dots S_{i-1})(S_{\ell-1} \dots S_j)$ to $(\lambda,f^{\st}_{\lambda})$. In other words, move the $j$-th column of $(\lambda,f^{\st}_{\lambda})$ to the far right, and move the $i$-th column of $(\lambda,f^{\st}_{\lambda})$ to the far left.
    \item (Step 2) Apply the operator $\cycling$.
\end{itemize}

See Figure~\ref{fig: visualization} for a pictorial description and also see Example \ref{ex: 5.4}.

\end{definition}

To apply an operator $S_j$ to a filled diagram $(D,f)$, we require that the restriction to $j$ and $j+1$-th columns belongs to $\mathcal{V}(n,m)$ for some $n$ and $m$. The following lemma shows the validity of (Step 1) in the above definition.

\begin{figure} 
\begin{tikzpicture}[scale=0.6]
\begin{scope}[shift={(-10,0)}]
\filldraw[black] (0.35,6.5) circle (0.000001pt) node[anchor=south] {$(\mu,f^{\st}_{\mu})$};

\draw[blue] (0,0)--(0.7,0);
\draw[blue] (0,0)--(0,6);
\draw[blue] (0,6)--(0.7,6);
\draw[blue] (0.7,0)--(0.7,6);

\filldraw[black] (1.2,3) circle (0.000001pt) node[anchor=south] {\textcolor{blue}{$\dots$}};

\begin{scope}[shift={(1.7,0)}]
\draw[blue] (0,0)--(0.7,0);
\draw[blue] (0,0)--(0,5.5);
\draw[blue] (0,5.5)--(0.7,5.5);
\draw[blue] (0.7,0)--(0.7,5.5);
\end{scope}

\begin{scope}[shift={(2.6,0)}]
\draw[black] (0,0)--(0.7,0);
\draw[black] (0,0)--(0,5.1);
\draw[black] (0,5.1)--(0.7,5.1);
\draw[black] (0.7,0)--(0.7,5.1);
\draw[black] (0,4.4)--(0.7,4.4);
\filldraw[black] (0.35,-1) circle (0.000001pt) node[anchor=south] {\textcolor{black}{$i$}};
\draw[->] (0.35, -1) .. controls (4.5,-1.7) .. (6.7,-0.3);
\end{scope}

\begin{scope}[shift={(3.5,0)}]
\draw[red] (0,0)--(0.7,0);
\draw[red] (0,0)--(0,4);
\draw[red] (0,4)--(0.7,4);
\draw[red] (0.7,0)--(0.7,4);
\end{scope}

\filldraw[black] (4.6,2) circle (0.000001pt) node[anchor=south] {\textcolor{red}{$\dots$}};

\begin{scope}[shift={(4.9,0)}]
\draw[red] (0,0)--(0.7,0);
\draw[red] (0,0)--(0,3.7);
\draw[red] (0,3.7)--(0.7,3.7);
\draw[red] (0.7,0)--(0.7,3.7);
\end{scope}

\begin{scope}[shift={(5.8,0)}]
\draw[black] (0,0)--(0.7,0);
\draw[black] (0,0)--(0,2.8);
\draw[black] (0,2.8)--(0.7,2.8);
\draw[black] (0.7,0)--(0.7,2.8);
\filldraw[black] (0.35,-1) circle (0.000001pt) node[anchor=south] {\textcolor{black}{$j$}};
\draw[->] (0.35, -1) .. controls (-3.5,-1.7) .. (-6.2,-0.3);

\end{scope}

\begin{scope}[shift={(6.7,0)}]
\draw[green] (0,0)--(0.7,0);
\draw[green] (0,0)--(0,2.6);
\draw[green] (0,2.6)--(0.7,2.6);
\draw[green] (0.7,0)--(0.7,2.6);
\end{scope}

\filldraw[black] (7.9,1.7) circle (0.000001pt) node[anchor=south] {\textcolor{green}{$\dots$}};

\begin{scope}[shift={(8.2,0)}]
\draw[green] (0,0)--(0.7,0);
\draw[green] (0,0)--(0,2.3);
\draw[green] (0,2.3)--(0.7,2.3);
\draw[green] (0.7,0)--(0.7,2.3);
\end{scope}

\draw[->,thick] (4.7,-2.5)--(4.7,-5.3);
\filldraw[black] (4.7,-3.9) circle (0.000001pt) node[anchor=east] {column exchange rule};

\begin{scope}[shift={(11.5,0)}]

\filldraw[black] (0.35,6.5) circle (0.000001pt) node[anchor=south] {{$(\lambda,f^{\st}_{\lambda})$}};

\draw[blue] (0,0)--(0.7,0);
\draw[blue] (0,0)--(0,6);
\draw[blue] (0,6)--(0.7,6);
\draw[blue] (0.7,0)--(0.7,6);

\filldraw[black] (1.2,3) circle (0.000001pt) node[anchor=south] {\textcolor{blue}{$\dots$}};

\begin{scope}[shift={(1.7,0)}]
    \draw[blue] (0,0)--(0.7,0);
\draw[blue] (0,0)--(0,5.5);
\draw[blue] (0,5.5)--(0.7,5.5);
\draw[blue] (0.7,0)--(0.7,5.5);
\end{scope}

\begin{scope}[shift={(2.6,0)}]
    \draw[black] (0,0)--(0.7,0);
\draw[black] (0,0)--(0,4.4);
\draw[black] (0,4.4)--(0.7,4.4);
\draw[black] (0.7,0)--(0.7,4.4);
\filldraw[black] (0.35,-1) circle (0.000001pt) node[anchor=south] {\textcolor{black}{$i$}};
\draw[->] (0.35, -1) .. controls (-1,-1.5) .. (-3.2,-0.3);
\end{scope}

\begin{scope}[shift={(3.5,0)}]
\draw[red] (0,0)--(0.7,0);
\draw[red] (0,0)--(0,4);
\draw[red] (0,4)--(0.7,4);
\draw[red] (0.7,0)--(0.7,4);
\end{scope}

\filldraw[black] (4.6,2) circle (0.000001pt) node[anchor=south] {\textcolor{red}{$\dots$}};

\begin{scope}[shift={(4.9,0)}]
\draw[red] (0,0)--(0.7,0);
\draw[red] (0,0)--(0,3.7);
\draw[red] (0,3.7)--(0.7,3.7);
\draw[red] (0.7,0)--(0.7,3.7);
\end{scope}

\begin{scope}[shift={(5.8,0)}]
\draw[black] (0,0)--(0.7,0);
\draw[black] (0,0)--(0,3.5);
\draw[black] (0,3.5)--(0.7,3.5);
\draw[black] (0.7,0)--(0.7,3.5);
\draw[black] (0,2.8)--(0.7,2.8);
\filldraw[black] (0.35,-1) circle (0.000001pt) node[anchor=south] {\textcolor{black}{$j$}};
\draw[->] (0.35, -1) .. controls (1.5,-1.5) .. (3.5,-0.3);
\end{scope}

\begin{scope}[shift={(6.7,0)}]
\draw[green] (0,0)--(0.7,0);
\draw[green] (0,0)--(0,2.6);
\draw[green] (0,2.6)--(0.7,2.6);
\draw[green] (0.7,0)--(0.7,2.6);
\end{scope}

\filldraw[black] (7.9,1.7) circle (0.000001pt) node[anchor=south] {\textcolor{green}{$\dots$}};

\begin{scope}[shift={(8.2,0)}]
\draw[green] (0,0)--(0.7,0);
\draw[green] (0,0)--(0,2.3);
\draw[green] (0,2.3)--(0.7,2.3);
\draw[green] (0.7,0)--(0.7,2.3);
\end{scope}

\draw[->,thick] (4.7,-2.5)--(4.7,-5.3);
\filldraw[black] (4.7,-3.9) circle (0.000001pt) node[anchor=east] {column exchange rule};

\end{scope}

\begin{scope}[shift={(0,-12.5)}]

\draw[->] (0.35, -0.2) .. controls (5,-1.2) .. (9.3,-0);

\begin{scope}[shift={(0.9,0)}]
\draw[blue] (0,0)--(0.7,0);
\draw[blue] (0,0)--(0,6);
\draw[blue] (0,6)--(0.7,6);
\draw[blue] (0.7,0)--(0.7,6);

\filldraw[black] (1.2,3) circle (0.000001pt) node[anchor=south] {\textcolor{blue}{$\dots$}};

\begin{scope}[shift={(1.7,0)}]
\draw[blue] (0,0)--(0.7,0);
\draw[blue] (0,0)--(0,5.5);
\draw[blue] (0,5.5)--(0.7,5.5);
\draw[blue] (0.7,0)--(0.7,5.5);
\end{scope}
\end{scope}

\begin{scope}[shift={(8.2,0)}]
\draw[black] (0,0)--(0.7,0);
\draw[black] (0,0)--(0,5.1);
\draw[black] (0,5.1)--(0.7,5.1);
\draw[black] (0.7,0)--(0.7,5.1);
\draw[black] (0,4.4)--(0.7,4.4);
\end{scope}

\begin{scope}[shift={(3.5,0)}]
\draw[red] (0,0)--(0.7,0);
\draw[red] (0,0)--(0,4);
\draw[red] (0,4)--(0.7,4);
\draw[red] (0.7,0)--(0.7,4);
\end{scope}

\filldraw[black] (4.6,2) circle (0.000001pt) node[anchor=south] {\textcolor{red}{$\dots$}};

\begin{scope}[shift={(4.9,0)}]
\draw[red] (0,0)--(0.7,0);
\draw[red] (0,0)--(0,3.7);
\draw[red] (0,3.7)--(0.7,3.7);
\draw[red] (0.7,0)--(0.7,3.7);
\end{scope}

\begin{scope}[shift={(0,0)}]
\draw[black] (0,0)--(0.7,0);
\draw[black] (0,0)--(0,2.8);
\draw[black] (0,2.8)--(0.7,2.8);
\draw[black] (0.7,0)--(0.7,2.8);
\end{scope}

\begin{scope}[shift={(5.8,0)}]
\draw[green] (0,0)--(0.7,0);
\draw[green] (0,0)--(0,2.6);
\draw[green] (0,2.6)--(0.7,2.6);
\draw[green] (0.7,0)--(0.7,2.6);
\end{scope}

\filldraw[black] (7.0,1.7) circle (0.000001pt) node[anchor=south] {\textcolor{green}{$\dots$}};

\begin{scope}[shift={(7.3,0)}]
\draw[green] (0,0)--(0.7,0);
\draw[green] (0,0)--(0,2.3);
\draw[green] (0,2.3)--(0.7,2.3);
\draw[green] (0.7,0)--(0.7,2.3);
\end{scope}

\draw[->,thick] (5,-2.1)--(5,-5.3);
\filldraw[black] (5,-3.7) circle (0.000001pt) node[anchor=east] {cycling};
\end{scope}

\begin{scope}[shift={(11.5,-12.5)}]

\draw[->] (0.35, -0.2) .. controls (5,-1.2) .. (9.3,-0);

\begin{scope}[shift={(0.9,0)}]
\draw[blue] (0,0)--(0.7,0);
\draw[blue] (0,0)--(0,6);
\draw[blue] (0,6)--(0.7,6);
\draw[blue] (0.7,0)--(0.7,6);

\filldraw[black] (1.2,3) circle (0.000001pt) node[anchor=south] {\textcolor{blue}{$\dots$}};

\begin{scope}[shift={(1.7,0)}]
    \draw[blue] (0,0)--(0.7,0);
\draw[blue] (0,0)--(0,5.5);
\draw[blue] (0,5.5)--(0.7,5.5);
\draw[blue] (0.7,0)--(0.7,5.5);
\end{scope}
\end{scope}

\begin{scope}[shift={(0,0)}]
    \draw[black] (0,0)--(0.7,0);
\draw[black] (0,0)--(0,4.4);
\draw[black] (0,4.4)--(0.7,4.4);
\draw[black] (0.7,0)--(0.7,4.4);
\end{scope}

\begin{scope}[shift={(3.5,0)}]
\draw[red] (0,0)--(0.7,0);
\draw[red] (0,0)--(0,4);
\draw[red] (0,4)--(0.7,4);
\draw[red] (0.7,0)--(0.7,4);
\end{scope}

\filldraw[black] (4.6,2) circle (0.000001pt) node[anchor=south] {\textcolor{red}{$\dots$}};

\begin{scope}[shift={(4.9,0)}]
\draw[red] (0,0)--(0.7,0);
\draw[red] (0,0)--(0,3.7);
\draw[red] (0,3.7)--(0.7,3.7);
\draw[red] (0.7,0)--(0.7,3.7);
\end{scope}

\begin{scope}[shift={(8.2,0)}]
    \draw[black] (0,0)--(0.7,0);
\draw[black] (0,0)--(0,3.5);
\draw[black] (0,3.5)--(0.7,3.5);
\draw[black] (0.7,0)--(0.7,3.5);
\draw[black] (0,2.8)--(0.7,2.8);
\end{scope}

\begin{scope}[shift={(5.8,0)}]
\draw[green] (0,0)--(0.7,0);
\draw[green] (0,0)--(0,2.6);
\draw[green] (0,2.6)--(0.7,2.6);
\draw[green] (0.7,0)--(0.7,2.6);
\end{scope}

\filldraw[black] (7.0,1.7) circle (0.000001pt) node[anchor=south] {\textcolor{green}{$\dots$}};

\begin{scope}[shift={(7.3,0)}]
\draw[green] (0,0)--(0.7,0);
\draw[green] (0,0)--(0,2.3);
\draw[green] (0,2.3)--(0.7,2.3);
\draw[green] (0.7,0)--(0.7,2.3);
\end{scope}

\draw[->,thick] (5,-2.1)--(5,-5.3);
\filldraw[black] (5,-3.7) circle (0.000001pt) node[anchor=east] {cycling};

\end{scope}

\begin{scope}[shift={(0,-25)}]

\filldraw[black] (5,-1) circle (0.000001pt) node[anchor=east] {$\mathfrak{D}_{\lambda,\mu}(\mu,f_{\mu}^{\st})$};

\draw[blue] (0,0)--(0.7,0);
\draw[blue] (0,0)--(0,6);
\draw[blue] (0,6)--(0.7,6);
\draw[blue] (0.7,0)--(0.7,6);

\filldraw[black] (1.2,3) circle (0.000001pt) node[anchor=south] {\textcolor{blue}{$\dots$}};

\begin{scope}[shift={(1.7,0)}]
    \draw[blue] (0,0)--(0.7,0);
\draw[blue] (0,0)--(0,5.5);
\draw[blue] (0,5.5)--(0.7,5.5);
\draw[blue] (0.7,0)--(0.7,5.5);
\end{scope}

\begin{scope}[shift={(7.3,0)}]
    \draw[black] (0,0)--(0.7,0);
\draw[black] (0,0)--(0,5.1);
\draw[black] (0,5.1)--(0.7,5.1);
\draw[black] (0.7,0)--(0.7,5.1);
\draw[black] (0,4.4)--(0.7,4.4);
\end{scope}

\begin{scope}[shift={(2.6,0)}]
\draw[red] (0,0)--(0.7,0);
\draw[red] (0,0)--(0,4);
\draw[red] (0,4)--(0.7,4);
\draw[red] (0.7,0)--(0.7,4);
\end{scope}

\filldraw[black] (3.7,2) circle (0.000001pt) node[anchor=south] {\textcolor{red}{$\dots$}};

\begin{scope}[shift={(4,0)}]
\draw[red] (0,0)--(0.7,0);
\draw[red] (0,0)--(0,3.7);
\draw[red] (0,3.7)--(0.7,3.7);
\draw[red] (0.7,0)--(0.7,3.7);
\end{scope}

\begin{scope}[shift={(8.2,0.7)}]
    \draw[black] (0,0)--(0.7,0);
\draw[black] (0,0)--(0,2.8);
\draw[black] (0,2.8)--(0.7,2.8);
\draw[black] (0.7,0)--(0.7,2.8);
\end{scope}

\begin{scope}[shift={(4.9,0)}]
\draw[green] (0,0)--(0.7,0);
\draw[green] (0,0)--(0,2.6);
\draw[green] (0,2.6)--(0.7,2.6);
\draw[green] (0.7,0)--(0.7,2.6);
\end{scope}

\filldraw[black] (6,1.7) circle (0.000001pt) node[anchor=south] {\textcolor{green}{$\dots$}};

\begin{scope}[shift={(6.4,0)}]
\draw[green] (0,0)--(0.7,0);
\draw[green] (0,0)--(0,2.3);
\draw[green] (0,2.3)--(0.7,2.3);
\draw[green] (0.7,0)--(0.7,2.3);
\end{scope}

\end{scope}

\begin{scope}[shift={(11.5,-25)}]

\draw[blue] (0,0)--(0.7,0);
\draw[blue] (0,0)--(0,6);
\draw[blue] (0,6)--(0.7,6);
\draw[blue] (0.7,0)--(0.7,6);

\filldraw[black] (1.2,3) circle (0.000001pt) node[anchor=south] {\textcolor{blue}{$\dots$}};

\begin{scope}[shift={(1.7,0)}]
    \draw[blue] (0,0)--(0.7,0);
\draw[blue] (0,0)--(0,5.5);
\draw[blue] (0,5.5)--(0.7,5.5);
\draw[blue] (0.7,0)--(0.7,5.5);
\end{scope}

\begin{scope}[shift={(8.1,0.7)}]
    \draw[black] (0,0)--(0.7,0);
\draw[black] (0,0)--(0,4.4);
\draw[black] (0,4.4)--(0.7,4.4);
\draw[black] (0.7,0)--(0.7,4.4);

\end{scope}

\begin{scope}[shift={(2.6,0)}]
\draw[red] (0,0)--(0.7,0);
\draw[red] (0,0)--(0,4);
\draw[red] (0,4)--(0.7,4);
\draw[red] (0.7,0)--(0.7,4);
\end{scope}

\filldraw[black] (3.7,2) circle (0.000001pt) node[anchor=south] {\textcolor{red}{$\dots$}};

\begin{scope}[shift={(4.0,0)}]
\draw[red] (0,0)--(0.7,0);
\draw[red] (0,0)--(0,3.7);
\draw[red] (0,3.7)--(0.7,3.7);
\draw[red] (0.7,0)--(0.7,3.7);
\end{scope}

\begin{scope}[shift={(7.3,0)}]
\draw[black] (0,0)--(0.7,0);
\draw[black] (0,0)--(0,3.5);
\draw[black] (0,3.5)--(0.7,3.5);
\draw[black] (0.7,0)--(0.7,3.5);
\draw[black] (0,2.8)--(0.7,2.8);
\end{scope}

\begin{scope}[shift={(6.7-1.8,0)}]
\draw[green] (0,0)--(0.7,0);
\draw[green] (0,0)--(0,2.6);
\draw[green] (0,2.6)--(0.7,2.6);
\draw[green] (0.7,0)--(0.7,2.6);
\end{scope}

\filldraw[black] (7.9-1.8,1.7) circle (0.000001pt) node[anchor=south] {\textcolor{green}{$\dots$}};

\begin{scope}[shift={(8.2-1.8,0)}]
\draw[green] (0,0)--(0.7,0);
\draw[green] (0,0)--(0,2.3);
\draw[green] (0,2.3)--(0.7,2.3);
\draw[green] (0.7,0)--(0.7,2.3);
\end{scope}

\filldraw[black] (5,-1) circle (0.000001pt) node[anchor=east] {$\mathfrak{D}_{\lambda,\mu}(\lambda,f_{\lambda}^{\st})$};

\end{scope}

\end{scope}

\end{tikzpicture}
\caption{Illustration of column exchange rule and cycling applied to the filled diagrams $(\mu,f^{\st}_{\mu})$ and$(\lambda,f^{\st}_{\lambda})$.}
\label{fig: visualization}
\end{figure}

\begin{lem}\label{lem: step1 validity}
    While performing (Step 1) in Definition~\ref{def: deformation}, the condition \eqref{eq: column exchange condition} is always satisfied. 
\end{lem}
\begin{proof}
    We will give a proof for the case when we perform (Step 1) to \((\lambda,f^{\st}_{\lambda})\). The remaining case \((\mu,f^{\st}_{\mu})\) can be proved in a similar way.

Assume we moved the \(j\)-th column of \((\lambda,f^{\st}_{\lambda})\) to the right \(n\)-times (\(n<\ell-j\)) by applying the sequence of operators \((S_{j+n-1}\dots S_{j})\), and denote the resulting filled diagram by \((D^{(n)},f_{D^{(n)}})\). Then the \(j\)-th column of \((\lambda,f^{\st}_{\lambda})\) now corresponds to the \((j+n)\)-th column of \((D^{(n)},f_{D^{(n)}})\), and by \eqref{eq: column exchange the second condition}, we have
\[
    f_{D^{(n)}}(x,j+n) = q^{-c_x}f^{\st}_{\lambda}(x,j)
\]
where \(c_x\) is the number of \((x-1)\)s in the list \(a_{j+1},\dots,a_{j+n}\). If $1<x\leq a_{j+n+1}$, it is easy to check that
\begin{align*}
    &c_{a_{j+n+1}+1}+\arm_{\lambda}(a_{j+n+1}+1,j)=n, \qquad c_x=0, \\
    &\arm_{\lambda}(x,j)=\arm_{\lambda}(x,j+n+1)+n+1, \quad\\
    &\leg_{\lambda}(x,j)+1=\leg_{\lambda}(x,j+n+1)+1+a_j-a_{j+n+1}.
\end{align*}
We conclude that for \(1<x\leq a_{j+n+1}\),
\begin{align*}
    f_{D^{(n)}}(x,j+n)
    &= q^{-\arm_{\lambda}(x,j)}t^{a_j-x+1}\\
    &= \left(q^{-n-1}t^{a_{j}-a_{j+n+1}}\right)\left(q^{-\arm_{\lambda}(x,j+n+1)}t^{a_{j+n+1}-x+1}\right)\\
    &= \left(q^{-1}f_{D^{(n)}}(a_{j+n+1}+1,j+n)\right)\left(f_{D^{(n)}}(x,j+n+1)\right).
\end{align*}
Thus, the condition \eqref{eq: column exchange condition} is satisfied, and we can apply the operator \(S_{j+n}\) to \((D^{(n)},f_{D^{(n)}})\).

With the same argument, we can show that the condition \eqref{eq: column exchange condition} is always satisfied while moving the \(i\)-th column of \((\lambda,f^{\st}_{\lambda})\) to the far left.

\end{proof}

\begin{lem}\label{lem: deformed shape is nice}
    Keep the notations in Definition~\ref{def: deformation}. Then the restriction to the last two columns of $\mathfrak{D}_{\lambda,\mu}(\mu,f_{\mu}^{\st})$ is in $\bar{\mathcal{V}}(a_i,a_j;\frac{T_{\mu}}{T_{\lambda}})$ and we have 
        $\mathfrak{D}_{\lambda,\mu}(\lambda,f_{\lambda}^{\st})=\bar{S}_{\ell-1}\left(\mathfrak{D}_{\lambda,\mu}(\mu,f_{\mu}^{\st})\right)$.
\end{lem}
\begin{proof}For \( y \neq i, j \), we have
\begin{equation}\label{eq: lambda mu standard filling relation}
    \begin{cases*}
        q f_{\mu}^{\st}(x,y) = f_{\lambda}^{\st}(x,y), & \text{if  $x = a_i$  and $ y < i$}, \\
        f_{\mu}^{\st}(x,y) = q f_{\lambda}^{\st}(x,y), & \text{if  $x = a_j$  and $ y < j$}, \\
        f_{\mu}^{\st}(x,y) = f_{\lambda}^{\st}(x,y), & \text{otherwise}.
    \end{cases*}
\end{equation}

Denoting \(\mathfrak{D}_{\lambda,\mu}(\mu,f_{\mu}^{\st}) = (D_{\mu},f_{D_{\mu}})\) and \(\mathfrak{D}_{\lambda,\mu}(\lambda,f_{\lambda}^{\st}) = (D_{\lambda},f_{D_{\lambda}})\), we have
\begin{align*}
    D_{\mu} &= [[a_1], \dots, \hat{[a_i]}, \dots, \hat{[a_j]}, \dots, [a_{\ell}], [a_i], [a_j] \setminus \{1\}], \\
    D_{\lambda} &= [[a_1], \dots, \hat{[a_i]}, \dots, \hat{[a_j]}, \dots, [a_{\ell}], [a_j], [a_i] \setminus \{1\}],
\end{align*}
where \(\hat{[a_i]}\) denotes that \([a_i]\) is omitted.

We first claim that the fillings \(f_{D_{\mu}}\) and \(f_{D_{\lambda}}\) agree on the first \((\ell-2)\) columns.

For \(1 \leq y \leq i-1\), the \(y\)-th column of \(\mu\) (respectively \(\lambda\)) is swapped with the \(j\)-th column (respectively \(i\)-th column) of \(\mu\) (respectively \(\lambda\)) during (Step 1). By \eqref{eq: column exchange the second condition}, we obtain
\[
\begin{cases*}
    f_{D_{\mu}}(x,y) = q^{-1} f_{\mu}^{\st}(x,y), & \text{if  $x = a_j$}, \\
    f_{D_{\mu}}(x,y) = f_{\mu}^{\st}(x,y), & \text{otherwise},
\end{cases*}
\qquad
\begin{cases*}
    f_{D_{\lambda}}(x,y) = q^{-1} f_{\lambda}^{\st}(x,y), & \text{if  $x = a_i$}, \\
    f_{D_{\lambda}}(x,y) = f_{\lambda}^{\st}(x,y), & \text{otherwise}.
\end{cases*}
\]

For \(i \leq y \leq j-2\), the \(y\)-th column of \(D_{\mu}\) (respectively \(D_{\lambda}\)) corresponds to the \((y+1)\)-th column of \(\mu\) (respectively \(\lambda\)). Note that the \((y+1)\)-th column of \(\mu\) swaps with the \(i\)-th and \(j\)-th columns of \(\mu\) during (Step 1), but the filling on the \((y+1)\)-th column changes only when it swaps with a column of smaller length, namely the \(j\)-th column of \(\mu\). On the other hand, the \((y+1)\)-th column of \(\lambda\) does not swap with any other columns during (Step 1). Therefore, we have
\[
\begin{cases*}
    f_{D_{\mu}}(x,y) = q^{-1} f_{\mu}^{\st}(x,y+1), & \text{if  $x = a_j$}, \\
    f_{D_{\mu}}(x,y) = f_{\mu}^{\st}(x,y+1), & \text{otherwise},
\end{cases*}
\qquad
f_{D_{\lambda}}(x,y) = f_{\lambda}^{\st}(x,y+1).
\]

For \(j-1 \leq y \leq \ell-2\), the \(y\)-th column of \(D_{\mu}\) (respectively \(D_{\lambda}\)) corresponds to the \((y+2)\)-th column of \(\mu\) (respectively \(\lambda\)). In (Step 1), the \((y+2)\)-th column of \(\mu\) (respectively \(\lambda\)) does not swap with any column whose length is smaller than that of \((y+2)\)-th column. Hence, we simply have
\[
f_{D_{\mu}}(x,y) = f_{\mu}^{\st}(x,y+2),
\qquad
f_{D_{\lambda}}(x,y) = f_{\lambda}^{\st}(x,y+2).
\]

In conclusion, combining the above three cases with \eqref{eq: lambda mu standard filling relation}, we obtain that \(f_{D_{\mu}}(x,y) = f_{D_{\lambda}}(x,y)\) for \(1 \leq y \leq \ell-2\).

Let \((D_1, f_{D_1})\) (respectively \((D_2, f_{D_2})\)) be the filled subdiagram of \(\mathfrak{D}_{\lambda,\mu}(\mu,f_{\mu}^{\st})\) (respectively \(\mathfrak{D}_{\lambda,\mu}(\lambda,f_{\lambda}^{\st})\)) obtained by restricting to the last two columns. It remains to show that \((D_1, f_{D_1}) \in \bar{\mathcal{V}}(a_i,a_j;\frac{T_{\mu}}{T_{\lambda}})\) and that \((D_2, f_{D_2}) = \bar{S}(D_1, f_{D_1})\).

Note that in (Step 1), the \(i\)-th column of \(\mu\) swaps with the \((i+1)\)-th through \(\ell\)-th columns of \(\mu\), while the \(j\)-th column of \(\mu\) only swaps with columns that are longer than the \(j\)-th column. Therefore, we have
\[
f_{D_1}(x,1) = q^{-c_x} f_{\mu}^{\st}(x,i), \qquad f_{D_1}(x+1,2) = f_{\mu}^{\st}(x,j),
\]
where \(c_x\) is the number of \((x-1)\)s in the list \(a_{i+1}, \dots, a_{j}-1, \dots, a_{\ell}\).

Similarly, we have
\[
f_{D_2}(x,1) = q^{-d_x} f_{\lambda}^{\st}(x,j), \qquad f_{D_2}(x+1,2) = f_{\lambda}^{\st}(x,i),
\]
where \(d_x\) is the number of \((x-1)\)s in the list $a_{j+1}, \dots, a_{\ell}$.

For \(1<x\leq a_j\), we have the relation \(c_x + \arm_{\mu}(x,i) = \arm_{\mu}(x-1,i)\). Since \(a_{j-1} \geq a_j\) (as \(\lambda\) is a partition), we obtain
\begin{equation}\label{eq: finding alpha}
    q^{-1} f_{D_1}(a_j+1,1) = q^{-1-\arm_{\mu}(a_j,i)} t^{\leg_{\mu}(a_j+1,1)+1} = q^{i-j} t^{a_i-a_j} = \frac{T_{\mu}}{T_{\lambda}}.
\end{equation}

Now, for \(2<x\leq a_j\), it is easy to check that \(\arm_{\mu}(x-1,i) = j-i+\arm_{\mu}(x-1,j)\). Denoting $\alpha= q^{-1} f_{D_1}(a_j+1,1)$, we conclude
\begin{align}\label{eq: finding conditions sat}
    f_{D_1}(x,1)
    &= q^{-\arm_{\mu}(x-1,i)} t^{a_i-x+1} \\
    &= \left(q^{i-j} t^{a_i-a_j}\right) \left(q^{-\arm_{\mu}(x-1,j)} t^{a_j-x+1}\right) \nonumber \\
    &= \alpha f_{D_1}(x,2). \nonumber
\end{align}

Since \eqref{eq: finding alpha} and \eqref{eq: finding conditions sat} verify the conditions in \eqref{eq: column exchange half condition}, we have shown that \((D_1, f_{D_1}) \in \bar{\mathcal{V}}(a_i,a_j;\frac{T_{\mu}}{T_{\lambda}})\). The claim that \((D_2, f_{D_2}) = \bar{S}(D_1, f_{D_1})\) follows by a similar argument.

\end{proof}

\begin{example}\label{ex: 5.4}
     Let $\mu=[[4],[4],[3],[1],[1]]$ and $\lambda = [[4],[3],[3],[2],[1]]$. First of all, the filled diagram $(\mu,f^{\st}_\mu)$ is 
\[
    \ytableausetup{boxsize=3em}
    \begin{ytableau}
    q^{-1}t & t  \\
    q^{-2}t^2 & q^{-1}t^2 & t \\
    q^{-2}t^3 & q^{-1}t^3 & t^2\\
     & & & & 
    \end{ytableau}.
\]
After applying the sequence of operators $(S_4S_3)(S_1S_2S_3)$ to move the fourth column to the far left and the second column to the far right, we obtain the filled diagram
\[
    \begin{ytableau}
    \none & q^{-1}t & \none & \none & q^{-1}t \\
    \none & q^{-2}t^2 &  t & \none & q^{-1}t^2\\
    \none & q^{-3}t^3 & q^{-1}t^2 & \none & q^{-3}t^3 \\
    & & & &
    \end{ytableau}.
\]
Then by applying $\cycling$, we have the filled diagram $\mathfrak{D}_{\lambda,\mu}(\mu,f^{\st}_\mu)$
\[
    \begin{ytableau}
    q^{-1}t  & \none & \none & q^{-1}t & \none \\
    q^{-2}t^2 & t & \none & q^{-1}t^2\\
    q^{-3}t^3 & q^{-1}t^2 & \none & q^{-3}t^3 & \\
     & & & 
    \end{ytableau}.
\]

    On the other hand, the filled diagram $(\lambda,f^{\st}_\lambda)$ is 
\[
    \ytableausetup{boxsize=3em}
    \begin{ytableau}
    t  \\
    q^{-2}t^2 & q^{-1}t & t \\
    q^{-3}t^3 & q^{-2}t^2 & q^{-1}t^2 & t\\
     & & & &
    \end{ytableau}.
\]
After applying the sequence of operators $(S_1)(S_4)$ to move the second column to the far left and the fourth column to the far right, we obtain the filled diagram
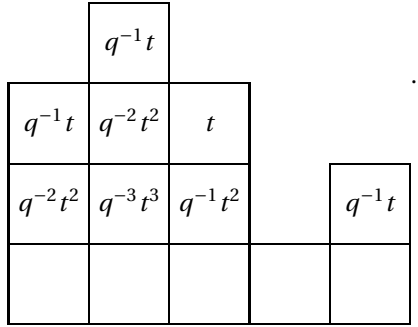
\[
    \begin{ytableau}
    \none & q^{-1}t  \\
    q^{-1}t &  q^{-2}t^2 & t \\
    q^{-2}t^2 & q^{-3}t^3 & q^{-1}t^2 & \none & q^{-1}t\\
     & & & &
    \end{ytableau}.
\]
Then by applying $\cycling$, we have the filled diagram $\mathfrak{D}_{\lambda,\mu}(\lambda,f^{\st}_\lambda)$
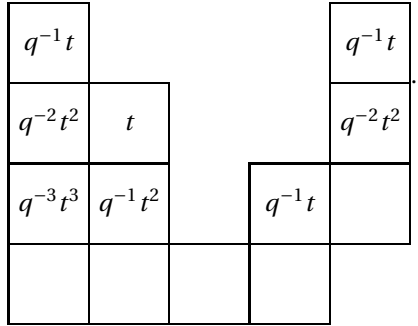
\[
    \begin{ytableau}
    q^{-1}t  & \none & \none & \none & q^{-1}t\\
    q^{-2}t^2 & t & \none & \none & q^{-2}t^2\\
    q^{-3}t^3 & q^{-1}t^2 & \none & q^{-1}t & \\
     & & & 
    \end{ytableau}.
\]

We can check that the first three columns of $\mathfrak{D}_{\lambda,\mu}(\mu,f^{\st}_\mu)$ and $\mathfrak{D}_{\lambda,\mu}(\lambda,f^{\st}_\lambda)$ coincide, and that the last two columns of $\mathfrak{D}_{\lambda,\mu}(\mu,f^{\st}_\mu)$ are in $\bar{\mathcal{V}}(4,2;q^{-2}t^{2})$, where $q^{-2}t^{2}=\frac{T_{\mu}}{T_{\lambda}}$. In addition, the last two columns of $\mathfrak{D}_{\lambda,\mu}(\lambda,f^{\st}_\lambda)$ are obtained by applying $\bar{S}$ to the last two columns of $\mathfrak{D}_{\lambda,\mu}(\mu,f^{\st}_\mu)$.
\end{example}

\begin{definition}\label{def: butler permutation for two partitions}
Keep the notations in Definition~\ref{def: deformation}. We denote the shape of $\mathfrak{D}_{\lambda,\mu}(\mu,f^{\st}_\mu)$ by $D_{\mu}$ and let $D'$ be its subdiagram obtained by restricting to the last two columns. We define
\begin{align*}
    \mathfrak{B}_{\lambda,\mu} &:= \left\{w\in\mathfrak{S}_{|\mu|} : \std\left(\left.(w\downarrow_{D'}^{D_{\mu}})\right|_{[a_i-a_j,a_i+a_j-1]}\right) \in \mathfrak{B}_{2a_j} \right\}, \\
    \stat_{\lambda,\mu} &:= \stat_{\mathfrak{D}_{\lambda,\mu}(\mu,f^{\st}_\mu)}.
\end{align*}
\end{definition}

\begin{proof}[Proof of Theorem~\ref{thm: first main, F-expansion}]
Without loss of generality, we let $\mu$ be the partition obtained from $\lambda$ by moving a single cell to an upper row. By Lemma~\ref{lem: cycling} and Corollary~\ref{cor: column exchange}, we have
\begin{equation*}
    \widetilde{H}_{\mu}[X;q,t] = \widetilde{H}_{\mathfrak{D}_{\lambda,\mu}(\mu,f^{\st}_\mu)}[X;q,t], \qquad \widetilde{H}_{\lambda}[X;q,t] = \widetilde{H}_{\mathfrak{D}_{\lambda,\mu}(\lambda,f^{\st}_\lambda)}[X;q,t].
\end{equation*}
Therefore, we obtain
\begin{align*}
    \I_{\lambda,\mu}[X;q,t] &= \frac{T_{\lambda}\widetilde{H}_{\mu}[X;q,t]-T_{\mu}\widetilde{H}_{\lambda}[X;q,t]}{T_{\lambda}-T_{\mu}} \\
    &= \frac{\widetilde{H}_{\mathfrak{D}_{\lambda,\mu}(\mu,f^{\st}_\mu)}[X;q,t] - \frac{T_{\mu}}{T_{\lambda}}\widetilde{H}_{\mathfrak{D}_{\lambda,\mu}(\lambda,f^{\st}_\lambda)}[X;q,t]}{1-\frac{T_{\mu}}{T_{\lambda}}}.
\end{align*}
Now Lemma~\ref{lem: deformed shape is nice} and \eqref{eq: concatenation2} complete the proof.
\end{proof}

\subsection{Butler Words and monomial symmetric function expansion}
In this subsection, based on Theorem~\ref{thm: first main, F-expansion}, we provide a monomial symmetric function expansion for $\I_{\lambda,\mu}[X;q,t]$. The \emph{weight} of a word $w$ is given by $(m_1, m_2, \dots)$, where $m_i$ denotes the number of occurrences of $i$ in $w$. For a partition $\epsilon$, we say that $w$ is a \emph{Butler word} of type $\lambda, \mu$ with weight $\epsilon$ if the weight of $w$ is $\epsilon$ and $\std(w) \in \mathfrak{B}_{\lambda,\mu}$. We define $\mathfrak{B}_{\lambda,\mu}^{\epsilon}$ to be the set of Butler words of type $\lambda, \mu$ with weight $\epsilon$.

Recall that $F_{\alpha} = \sum_{\alpha \succeq \beta} M_{\beta}$ where $\alpha \succeq \beta$ means that $\alpha$ can be obtained by adding together adjacent parts of $\beta$. Also we have 
 $m_{\lambda} = \sum_{\alpha} M_{\alpha}$
where the sum ranges over all compositions $\alpha$ whose parts can be rearranged to give $\lambda$.
 For a partition $\epsilon$, it is straightforward to see that $F_{\iDes(w)}$ contains $M_{\epsilon}$ if and only if $w = \std(w')$ for some word $w'$ of weight $\epsilon$. Since we know that $\I_{\lambda,\mu}[X;q,t]$ is a symmetric function, to obtain the coefficient of $m_{\epsilon}$, it is enough to compute the coefficient of $M_{\epsilon}$. Therefore, we have the following monomial symmetric function expansion of $\I_{\lambda,\mu}[X;q,t]$.

\begin{corollary}\label{Cor: m-expansion}
Let $\nu$ be a partition and $\lambda, \mu \subseteq \nu$ be two distinct partitions such that $|\nu/\lambda| = |\nu/\mu| = 1$. Then we have
\begin{equation*}
    \I_{\lambda,\mu}[X;q,t] = \sum_{\epsilon \vdash n} \sum_{w \in \mathfrak{B}_{\lambda,\mu}^\epsilon} \stat_{\lambda,\mu}(w) m_\epsilon[X],
\end{equation*}
where $|\lambda| = |\mu| = n$.
\end{corollary}

\subsection{Specialization at $q=t=1$}
Recall that the modified Macdonald polynomial $\widetilde{H}_\mu[X;q,t]$ for any partition $\mu \vdash n$ at $q = t = 1$ gives $h_{(1^n)}[X]$, which is the Frobenius characteristic of the regular representation $\mathbb{C}[\mathfrak{S}_n]$ of $\mathfrak{S}_n$. We now provide an analogous result for $\I_{\lambda,\mu}[X;q,t]$.

\begin{corollary}\label{Cor: h21111 at q=t=1}
Let $\nu$ be a partition, and let $\lambda, \mu \subseteq \nu$ be two distinct partitions such that $|\nu/\lambda| = |\nu/\mu| = 1$. Then we have
\[
    \I_{\lambda,\mu}[X;1,1] = h_{(2, 1^{n-2})}[X],
\]
which depends only on the size $n=|\mu| = |\lambda|$.
\end{corollary}

\begin{proof}
    Note that the set $\mathfrak{B}_{\lambda,\mu}$ can be described as $\{w \in \mathfrak{S}_n : \std(w \vert_{A'}) \in \mathfrak{B}_{|A'|}\}$ for some subset $A' \subseteq [n]$. We conclude
    \begin{align*}
         \I_{\lambda,\mu}[X;1,1] &= \sum_{w \in \mathfrak{B}_{\lambda,\mu}} F_{\iDes(w)} \\
         &= \sum_{\{w \in \mathfrak{S}_n : \std(w \vert_{A'}) \in \mathfrak{B}_{|A'|}\}} F_{\iDes(w)} \\
         &= (h_1)^{n - |A'|} \left( \sum_{v \in \mathfrak{B}_{|A'|}} F_{\iDes(v)} \right).
    \end{align*}
    Here, the last equality follows from Lemma~\ref{lem: pieri rule for F}. The proof is then completed by Proposition~\ref{Prop: h21111}.
\end{proof}

\begin{rmk}
    In the companion paper \cite{KLO22+}, we give an explicit expansion (in terms of complete homogeneous symmetric functions) formula for the Macdonald intersection polynomial $\I_{\mu^{(1)},\dots,\mu^{(k)}}[X;q,t]$ at $q=t=1$. Corollary~\ref{Cor: h21111 at q=t=1} corresponds to the case $k=2$ of \cite[Theorem 1.3]{KLO22+}.
\end{rmk}

\section{Proof of Theorem~\ref{thm: second main, s-positivities}}
\label{Sec: Proof of Schur positivities}
\subsection{LLT polynomials}
This subsection covers the background for LLT polynomials. The original definition of LLT polynomials involves ribbon tableaux and spin statistics \cite{LLT97}. Later, an alternative model was discovered by Haiman and Bylund \cite{HHLRU05}, and the two models are directly related via the Littlewood quotient map (sometimes called the Stanton--White correspondence \cite{SW85}). Subsequently, a fundamental quasisymmetric function expansion of LLT polynomials was given in \cite{HHL05}, which we adopt as the definition.

Let $\bmnu = (\nu^{(1)}, \nu^{(2)}, \dots)$ be a tuple of skew partitions. For a standard tableau $\bmT = (T^{(1)}, T^{(2)}, \dots)$ of shape $\bmnu$, an \emph{inversion} of $\bmT$ is a pair of cells $u \in \nu^{(i)}$ and $v \in \nu^{(j)}$ such that $T^{(i)}(u) > T^{(j)}(v)$ and either
\begin{itemize}
    \item $i < j$ and $c(u) = c(v)$, or
    \item $i > j$ and $c(u) = c(v) + 1$,
\end{itemize}
where $c(u)$ denotes the \emph{content} of the cell $u = (a, b)$, which is $a - b$.
Denote by $\inv(\bmT)$ the number of inversions in $\bmT$. For a tuple of skew shapes, the \emph{reading order} is a total order on cells such that $u < v$ if $c(u) < c(v)$ or $c(u) = c(v)$ and $u$ is to the left of $v$. A \emph{reading word} of a semistandard tableau of a tuple of skew shapes is the word obtained by recording its entries in the reading order (from larger to smaller in that order). We denote the reading word of a tableau $\bmT$ as $\rw(\bmT)$. Finally, the \emph{LLT polynomial} $\llt_\bmnu[X;q]$ is defined by
\[
\llt_\bmnu[X;q] := \sum_{\bmT \in \syt(\bmnu)} q^{\inv(\bmT)} F_{\iDes(\rw(\bmT))}.
\]
For a tuple $\bmnu = (\nu^{(1)}, \dots, \nu^{(\ell)})$ of skew partitions, we have
\begin{equation}\label{eq: cycling11}
    \llt_{\bmnu}[X;q]
    = \llt_{(\nu^{(2)}, \dots, \nu^{(\ell)}, \mu)}[X;q],
\end{equation}
where $\mu$ is the skew partition obtained by replacing each cell $(a,b)\in \nu^{(1)}$ with the cell $(a+1,b)$.
 The proof of \eqref{eq: cycling11} is a routine exercise and parallels that of Lemma~\ref{lem: cycling}.

We summarize some key properties of LLT polynomials. By specializing $q = 1$, LLT polynomials become the product of the skew Schur functions:
\[
    \llt_\bmnu[X; 1] = \prod_{i \geq 1} s_{\nu^{(i)}}[X].
\]
Thus, LLT polynomials are $q$-analogues of the product of the skew Schur functions. Although it is not immediate from the definition, LLT polynomials are indeed symmetric functions \cite{LLT97, HHL05}. A prominent result of Haiman and Grojnowski \cite{GH07} shows that LLT polynomials are Schur-positive. Their proof relies on the Kazhdan-Lusztig theory of the Hecke algebra of affine type $A_n$. However, their proof does not provide a combinatorial formula for the Schur coefficients.

Now we introduce a notion of LLT equivalence, which will be an important tool for the remainder of this section. 
\begin{definition}
    Two $\mathbb{F}$-linear combinations, $\sum_{\bmnu} a_\bmnu(q,t) \llt_{\bmnu}[X;q]$ and $\sum_{\bmmu} b_\bmmu(q,t) \llt_{\bmmu}[X;q]$, of LLT polynomials are called \emph{LLT equivalent}, which we denote
\[
\sum_{\bmnu} a_\bmnu(q,t) \llt_{\bmnu}[X;q] \equiv \sum_{\bmmu} b_\bmmu(q,t) \llt_{\bmmu}[X;q]
\]
if for every tuple of skew partitions $\bmlambda$, we have
\[
  \sum_\bmnu a_\bmnu(q,t) \llt_{(\bmnu,\bmlambda)}[X;q] = \sum_\bmmu b_\bmmu(q,t) \llt_{(\bmmu,\bmlambda)}[X;q].
\]
Here, for a tuple $\bmnu$ of skew partitions, $(\bmnu,\bmlambda)$ indicates a tuple of skew partitions obtained by concatenating $\bmlambda$ after $\bmnu$. We will frequently abuse the notation and denote the LLT equivalence by:
\[
    \sum_\bmnu a_\bmnu(q,t) \bmnu \equiv \sum_\bmmu b_\bmmu(q,t) \bmmu.
\]
Due to \eqref{eq: cycling11} we also have 
\[
  \sum_\bmnu a_\bmnu(q,t) \llt_{(\bmlambda',\bmnu,\bmlambda)}[X;q] = \sum_\bmmu b_\bmmu(q,t) \llt_{(\bmlambda',\bmmu,\bmlambda)}[X;q].
\]
\end{definition}
for any tuples of skew partitions $\bmlambda$ and $\bmlambda'$.

\begin{lem}\label{lem: llt equivalence condition}
Assume there is a bijection $\Phi : \cup_{\bmnu} \syt(\bmnu) \to \cup_{\bmmu} \syt(\bmmu)$ satisfying the following:
\begin{align*}
    (\Phi1) \quad &a_{\bmnu}(q,t) q^{\inv(\bmT)} = b_{\bmmu}(q,t) q^{\inv(\Phi(\bmT))}, \text{ where $\bmT \in \syt(\bmnu)$ \text{ and } $\Phi(\bmT)\in \syt(\bmmu)$,}  \\
    (\Phi2) \quad &\iDes(\rw(\bmT)) = \iDes(\rw(\Phi(\bmT))), \text{ and } \\
    (\Phi3) \quad &\{\bmT(u): c(u) = m\} = \{\Phi(\bmT)(u): c(u) = m\} \text{ for all $m$}.
\end{align*}
Then we have \[
    \sum_\bmnu a_\bmnu(q,t) \bmnu \equiv \sum_\bmmu b_\bmmu(q,t) \bmmu.
\]
\end{lem}
\begin{proof}
For any tuple $\bmlambda$ of skew partitions, we construct a map
$\overline{\Phi}: \bigcup_{\bmnu} \syt((\bmnu,\bmlambda)) \to \bigcup_{\bmmu} \syt((\bmmu,\bmlambda))$
as follows. For $\bmT \in \syt((\bmnu,\bmlambda))$, let $\bmT'$ (respectively $Q$) be the tableau obtained by restricting $\bmT$ to the shape $\bmnu$ (respectively $\bmlambda$). Entries in $\bmT'$ are distinct; denote them by $p_1 < p_2 < \cdots < p_{\ell}$. Next, consider $\bmT'' \in \syt(\bmnu)$ obtained by replacing each $p_i$ with $i$ in $\bmT'$. Then $\Phi(\bmT'') \in \syt(\bmmu)$ for some $\bmmu$. Let $P$ be the tableau obtained from $\Phi(\bmT'')$ by replacing each $i$ with $p_i$. Finally, we define $\overline{\Phi}(\bmT) \in \syt((\bmmu,\bmlambda))$ such that, restricted to the shape $\bmmu$ (respectively $\bmlambda$), it coincides with $P$ (respectively $Q$).

Now it suffices to prove that:
\begin{align*}
    (1) \quad &a_{\bmnu}(q,t) q^{\inv(\bmT)} = b_{\bmmu}(q,t) q^{\inv(\overline{\Phi}(\bmT))}, \text{where $\bmT \in \syt((\bmnu,\bmlambda))$ and $\overline{\Phi}(\bmT)\in \syt((\bmmu,\bmlambda))$,}  \\
    (2) \quad &\iDes(\rw(\bmT)) = \iDes(\rw(\overline{\Phi}(\bmT))).
\end{align*}
We first prove (2). Assume it is not true, and without loss of generality, take $r$ such that $r \notin \iDes(\rw(\bmT))$ and $r \in \iDes(\rw(\overline{\Phi}(\bmT)))$. Let $u,v \in (\bmnu,\bmlambda)$ and $u',v' \in (\bmmu,\bmlambda)$ be such that $\bmT(u) = r$, $\bmT(v) = r+1$, $\overline{\Phi}(\bmT)(u') = r$, and $\overline{\Phi}(\bmT)(v') = r+1$. Then $u < v$ in reading order while $u' > v'$. It is not possible to have $u,v \in \bmlambda$, and it is also not possible to have $u,v \in \bmnu$, since in that case property $(\Phi2)$ would be violated. Now assume $u \in \bmnu$ and $v \in \bmlambda$. By property $(\Phi3)$, we have $c(u) = c(u')$ and  $c(v) = c(v')$, so it is impossible to have both $u < v$ and $u' > v'$. A similar argument applies to the case where $u \in \bmlambda$ and $v \in \bmmu$, leading to a contradiction.

To prove (1), define
\begin{align*}
    &A=\{(u,v)\in \Inv(\bmT):u,v\in \bmnu \},\quad \hspace{26mm} A'=\{(u,v)\in \Inv(\overline{\Phi}(\bmT)):u,v\in \bmmu \}\\
    &B=\{(u,v)\in \Inv(\bmT): \text{only one of $u,v$ lies in $\bmnu$}\},\quad B'=\{(u,v)\in\Inv(\overline{\Phi}(\bmT)): \text{only one of $u,v$ lies in $\bmmu$}\}\\
    &C=\{(u,v)\in \Inv(\bmT):u,v\in \bmlambda \},\quad \hspace{26mm} C'=\{(u,v)\in \Inv(\overline{\Phi}(\bmT)):u,v\in \bmlambda \}
\end{align*}
where the notation $\Inv(\bmT)$ denotes the set of inversions in $\bmT$. Then we have \begin{equation*}
    a_{\bmnu}(q,t) q^{|A|} = b_{\bmmu}(q,t) q^{|A'|}
\end{equation*} by property $(\Phi2)$, and $|C| = |C'|$, which follows directly from the construction of the map $\overline{\Phi}$. Now take $(u,v) \in B$ with $u \in \bmnu$ and $v \in \bmlambda$. There exists $u' \in \bmmu$ such that $\bmT(u) = \overline{\Phi}(\bmT)(u')$. By property $(\Phi3)$, we have $c(u) = c(u')$, so we deduce that $(u',v) \in B'$. Similarly, one can associate an element in $B'$ to each $(u,v) \in B$ with $v \in \bmnu$ and $u \in \bmlambda$, and this correspondence is a bijection. Therefore, we conclude that $|B| = |B'|$, and (1) is proved.
\end{proof}

\subsection{LLT expansion of the (generalized) modified Macdonald polynomials}\label{subsec: LLT expansion of the modified Macdonald}
A \emph{ribbon} is a connected skew partition that contains no $2 \times 2$ block of cells. Note that the contents of the cells in a ribbon are consecutive integers. The \emph{descent set} of a ribbon $R$ is the set of contents $c(u)$ of those cells $u = (i,j) \in R$ such that the cell $v = (i-1,j)$, which lies directly below $u$, also belongs to $R$. For an interval $I = [r, r+s]$, there is a natural one-to-one correspondence between ribbons of content $I$ and subsets $S \subseteq I \setminus \{r\}$, determined by the descent set of each ribbon. We denote the ribbon with content set $I$ and descent set $S$ by $R_I(S)$. As a special case, we denote $C_a := R_{\{a\}}(\emptyset)$ which is a ribbon of size 1 whose content is $a$.

For a diagram $D=[D^{(1)},D^{(2)},\dots,]$ and a subset of cells $S\subseteq D^{+}$, let
\[
    S^{(j)}:=\{i:(i,j)\in S\}.
\]
We define $\bmR_D(S)$ to be a tuple of ribbons defined by
\[
\bmR_{D}(S):=(R_{D^{(1)}}(S^{(1)}),R_{D^{(2)}}(S^{(2)}),\dots).
\]
This allows us to give an expansion of the modified Macdonald polynomials for a filled diagram into LLT polynomials indexed by tuples of ribbons: For a filled diagram $(D,f)$, we have 
\begin{equation}\label{eq: modified macdonald to llt}
\widetilde{H}_{(D,f)}[X;q,t] = \sum_{S\subseteq D^{+}} f(S)\llt_{\bmR_{D}(S)}[X;q],
\end{equation}
where $f(S) := \prod_{u\in S}f(u)$. 
The following example illustrates the ribbon LLT expansion of the modified Macdonald polynomial for a filled diagram.
\begin{example}
    Let $(D,f)$ be the filled diagram from Example~\ref{ex: filled diagram and hdf}. Let $C_2$ be the cell with content $2$, and let $V$ and $H$ denote the vertical and horizontal strips with contents $\{1,2\}$, respectively. Then,
    \[
        \widetilde{H}_{(D,f)} = \alpha \llt_{(C_2,V)} + \llt_{(C_2, H)} = \alpha (F_{\{1\}} + F_{\{2\}} + q F_{\{1,2\}}) + (F_{\emptyset}+ qF_{\{1\}} + qF_{\{2\}}),
    \]
    which agrees with the computation in Example~\ref{ex: filled diagram and hdf}.
\end{example}

    With the notations in Proposition~\ref{lem: column exchange}, the identity 
    \begin{equation*}
        \widetilde{H}_{(\mu,f_{\mu})}[X;q,t]=\widetilde{H}_{(\lambda,f_{\lambda})}[X;q,t]
    \end{equation*}
    can be written in terms of LLT polynomials using \eqref{eq: modified macdonald to llt}. Moreover, the map $\phi_{n,m}$ serves as a bijection satisfying  $(\Phi1)$, $(\Phi2)$, and $(\Phi3)$ (which corresponds to $(\phi1)$, $(\phi2)$, and $(\phi3)$, respectively). Thus we have 
    \begin{equation*}
        \widetilde{H}_{(\mu,f_{\mu})}[X;q,t]\equiv\widetilde{H}_{(\lambda,f_{\lambda})}[X;q,t],
    \end{equation*}
     which implies \eqref{eq: column exchange final}.
\subsection{Proof of Theorem~\ref{thm: second main, s-positivities}}
Let $\mu$ be a partition obtained from $\lambda$ by moving a single cell in the first or the second row to the upper row. Then the last two columns of $\mathfrak{D}_{\lambda,\mu}(\mu,f_{\mu}^{\st})$ are in $\bar{\mathcal{V}}(n,m;\frac{T_{\mu}}{T_{\lambda}})$ for $m=1$ or $2$. It suffices to prove the following.

\begin{lem}\label{lem: s positivity reduction}
    Let $(D,f)\in \bar{\mathcal{V}}(n,m;\alpha)$ and $(D',f')=\bar{S}(D,f)$. Then \[\dfrac{\widetilde{H}_{(D,f)}[X;q,t] - \alpha \widetilde{H}_{(D',f')}[X;q,t]}{1 - \alpha}\] is LLT equivalent to a positive linear combination of LLT polynomials, for $m=1$ or $2$.
\end{lem}

\begin{proof}[Proof of Theorem~\ref{thm: second main, s-positivities})] 
Denote $(D,f) = \mathfrak{D}_{\lambda,\mu}(\mu, f_{\mu}^{\st})$ and $(D',f') = \mathfrak{D}_{\lambda,\mu}(\lambda, f_{\lambda}^{\st})$, and let $(\bar{D}, \bar{f})$ be the filled diagram obtained by restricting to all columns except the last two columns of $(D, f)$ (or $(D', f')$, since they differ only in the last two columns). Let $(E, g)$ (respectively $(E', g')$) be the filled diagram obtained by restricting to the last two columns of $(D, f)$ (respectively $(D', f')$). Recall that for any diagram $D$, we denote by $D^+$ the collection of cells in $D$ that are not bottom cells.

We have
\begin{align*}
    \widetilde{H}_{(D,f)}[X;q,t] &= \sum_{S \subseteq \bar{D}^+} \bar{f}(S) \sum_{S' \subseteq E^+} g(S') \llt_{(\bmR_{\bar{D}}(S), \bmR_{E}(S'))}[X; q], \\
    \widetilde{H}_{(D',f')}[X;q,t] &= \sum_{S \subseteq \bar{D}^+} \bar{f}(S) \sum_{S' \subseteq (E')^+} g'(S') \llt_{(\bmR_{\bar{D}}(S), \bmR_{E'}(S'))}[X; q].
\end{align*}

By Lemma~\ref{lem: s positivity reduction}, we have
\[
    \frac{\widetilde{H}_{(E, g)}[X; q, t] - \alpha\, \widetilde{H}_{(E', g')}[X; q, t]}{1 - \alpha} \equiv \sum_{\bmnu} a_{\bmnu}(q, t) \llt_{\bmnu}[X; q],
\]
for some positive polynomials $a_{\bmnu}(q, t)$ and $\alpha = \frac{T_{\mu}}{T_{\lambda}}$. Therefore, for any $S \subseteq \bar{D}^+$, we conclude
\begin{align*}
    &\frac{\bar{f}(S) \sum_{S' \subseteq E^+} g(S')\llt_{(\bmR_{\bar{D}}(S), \bmR_{E}(S'))}[X; q] - \alpha \bar{f}(S) \sum_{S' \subseteq (E')^+} g'(S') \llt_{(\bmR_{\bar{D}}(S), \bmR_{E'}(S'))}[X; q]}{1 - \alpha} \\
    &= \bar{f}(S) \sum_{\bmnu} a_{\bmnu}(q, t) \llt_{(\bmR_{\bar{D}}(S), \bmnu)}[X; q].
\end{align*}

Finally, summing over $S \subseteq \bar{D}^+$, we obtain
\begin{equation*}
    \I_{\lambda, \mu}[X; q, t] = \sum_{S \subseteq \bar{D}^+} \bar{f}(S) \sum_{\bmnu} a_{\bmnu}(q, t) \llt_{(\bmR_{\bar{D}}(S), \bmnu)}[X; q],
\end{equation*}
which proves Schur positivity.

\end{proof}

To prove Lemma~\ref{lem: s positivity reduction}, we need various LLT equivalences. We use Proposition~\ref{prop: llt equivalence bandwidth 2} to show Lemma~\ref{lem: s positivity reduction} for $m=1$. Then we use Proposition~\ref{Prop: LLT equivalences bandwidth 3} to show Lemma~\ref{lem: s positivity reduction} for $m=2$. See Figure~\ref{Fig: figure for Prop 3.4} for the pictorial description of Proposition~\ref{prop: llt equivalence bandwidth 2} and Proposition~\ref{Prop: LLT equivalences bandwidth 3}.

\begin{figure}[ht]
    \centering
    \ytableausetup{nosmalltableaux}
    \ytableausetup{boxsize=2em}
    \ytableausetup{nobaseline}
    \begin{equation*}
                \begin{ytableau}
                \none & \none[\diag] & & \none[\diag]\\
                \none[\diag] & \none[\diag] &\none[\diag] & \none[\diag]\\
                \none[\diag] & &\none[\diag] & \none 
                \end{ytableau}  
                \qquad = q \begin{ytableau}
                \none[\diag] & & \none[\diag]&\none\\
                \none[\diag] & & \none[\diag] 
                \end{ytableau}  + \begin{ytableau}
                \none[\diag] &\none[\diag] & \none[\diag]\\ &  & \none[\diag]
                \end{ytableau}  
            \end{equation*} 

    \begin{equation*}
                \begin{ytableau}
                \none & \none & \none[\diag] & \none[\diag]\\
                \none &  \none[\diag] &  & \\
                \none[\diag] & \none[\diag] &\none[\diag] & \none[\diag]\\
                \none[\diag] & & & \none 
                \end{ytableau}  
                \qquad =
                q^2
                \begin{ytableau}
                \none & \none  & \none[\diag]  & \none[\diag]  \\
                \none & \none[\diag]  & \none[\diag]  & \none[\diag]  \\
                \none[\diag] &   &   &  \none[\diag] \\
                \none[\diag] &   &   & \none
                \end{ytableau} 
                \qquad +
                \begin{ytableau}
                \none & \none  & \none[\diag]  & \none[\diag]  \\
                \none & \none[\diag]  & \none[\diag]  & \\
                \none[\diag] & \none[\diag]  & \none[\diag]  &  \none[\diag] \\
                 &   &   & \none
                \end{ytableau} 
            \end{equation*} 
     \begin{equation*}
                \begin{ytableau}
                \none & \none & \none & \\
                \none &  \none & \none[\diag] & \\
                \none & \none[\diag]&\none[\diag]&\none[\diag]\\
                \none[\diag] & & \none[\diag]&\none\\
                \none[\diag] & & \none & \none
                \end{ytableau}  
                \qquad =
                q
                \begin{ytableau}
                \none & \none  & \none[\diag]  & \none[\diag]  \\
                \none & \none[\diag]  & \none[\diag]  & \none[\diag]  \\
                \none[\diag] &   &   &  \none[\diag] \\
                \none[\diag] &   &   & \none
                \end{ytableau} 
                \qquad + q
                \begin{ytableau}
                \none & \none[\diag]  &  & \none[\diag]  \\
                 & \none[\diag]  & \none[\diag]  & \none  \\
                 & \none[\diag]  & \none  &  \none \\
                 & \none  & \none  & \none
                \end{ytableau} 
        \end{equation*}
     \begin{equation*}
                        \begin{ytableau}
                \none & \none & \none[\diag] & \none[\diag]\\
                \none &  \none[\diag] &  &\none[\diag] \\
                 & \none[\diag] &\none[\diag] & \none[\diag]\\
                 & & \none[\diag]& \none 
                \end{ytableau}  
                \qquad =
                q
                \begin{ytableau}
                \none & \none  & \none[\diag]  & \none[\diag]  \\
                \none & \none[\diag]  & \none[\diag]  & \none[\diag]  \\
                \none[\diag] &   &   &  \none[\diag] \\
                \none[\diag] &   &   & \none
                \end{ytableau} 
                \qquad + q^{-1}
                \begin{ytableau}
                \none & \none  &  & \none[\diag]  \\
                \none & \none[\diag]  & & \none[\diag]  \\
                \none[\diag] & \none[\diag]  & \none[\diag]  & \none \\
                 &   & \none  & \none
                \end{ytableau} 
        \end{equation*}
     \begin{equation*}
                \begin{ytableau}
                \none & \none & \none[\diag] & \\
                \none &  \none[\diag] & \none[\diag] & \none[\diag]\\
                 &  &\none[\diag] & \none \\
                \none[\diag] & & \none & \none 
                \end{ytableau}  
                \qquad =
                q
                \begin{ytableau}
                \none & \none  & \none[\diag]  & \none[\diag]  \\
                \none & \none[\diag]  & \none[\diag]  & \none[\diag]  \\
                \none[\diag] &   &   &  \none[\diag] \\
                \none[\diag] &   &   & \none
                \end{ytableau} 
                \qquad +                 \begin{ytableau}
                \none &  & & \none[\diag]  \\
                \none[\diag] & \none[\diag]  & \none[\diag]  & \none\\
                 & \none[\diag]  & \none  &  \none \\
                &  \none & \none  & \none
                \end{ytableau} 
        \end{equation*}
    \caption{Demonstration of LLT equivalences in Proposition~\ref{prop: llt equivalence bandwidth 2} and \ref{Prop: LLT equivalences bandwidth 3}.}
    \label{Fig: figure for Prop 3.4}
\end{figure}

\begin{proposition}\label{prop: llt equivalence bandwidth 2}
Let $V$ be a vertical strip, and $H$ be a horizontal strip of content $1,2$. Then the following LLT equivalence holds:
\[
    (C_1,C_2) \equiv q (V) + (H).
\]
\end{proposition}

\begin{proof}
There are two standard Young tableaux of shape $(C_1, C_2)$:
$T^{(1)}$, obtained by filling $C_1$ with $1$ and $C_2$ with $2$, and
$T^{(2)}$, obtained by filling $C_1$ with $2$ and $C_2$ with $1$.
Then we have $\inv(T^{(1)}) = 1$, $\rw(T^{(1)}) = 21$, 
$\inv(T^{(2)}) = 0$, and $\rw(T^{(2)}) = 12$.
The unique standard Young tableau of shape $V$ corresponds to $T^{(1)}$, 
and the unique standard Young tableau of shape $H$ corresponds to $T^{(2)}$.
It is straightforward to verify that the conditions in 
Lemma~\ref{lem: llt equivalence condition} are satisfied.
\end{proof}

\begin{proposition}\cite{Mil19}\label{Prop: LLT equivalences bandwidth 3}
Let $V_1 = R_{[2]}(\{2\})$ and $V_2 = R_{\{2,3\}}(\{3\})$ denote the vertical strips with contents $\{1,2\}$ and $\{2,3\}$, respectively. Let $H_1 = R_{[2]}(\emptyset)$ and $H_2 = R_{\{2,3\}}(\emptyset)$ be the horizontal strips with contents $\{1,2\}$ and $\{2,3\}$. Finally, let $S$ be the $2\times 2$ square with cell contents $\{1, 2, 2, 3\}$. Then, the following LLT equivalences hold:
\begin{enumerate}
    \item $(H_1,H_2) \equiv q^2(S) + (R_{[3]}(\emptyset), C_2)$.
    \item $(V_1,V_2) \equiv q(S) + q(R_{[3]}(\{2,3\}), C_2)$.
    \item $(R_{[3]}(\{3\}),C_2) \equiv q(S) + q^{-1}(H_1,V_2)$.
    \item $(R_{[3]}(\{2\}),C_2) \equiv q(S) + (V_1,H_2)$.
\end{enumerate}
\end{proposition}

\begin{proof}[Proof of Lemma~\ref{lem: s positivity reduction} for $m=1$]
We aim to construct a bijection between $\mathrm{SYT}(D)$ and $\mathrm{SYT}(D')$ that preserves $\mathrm{inv}$, $\mathrm{iDes}$, and content, as required by Lemma~\ref{lem: llt equivalence condition}. Observe that the proposed equivalence involves a local transformation on cells containing the entries $\{1, 2\}$. For any $n > 2$, we extend the bijection by fixing the position of all entries $k > 2$. Thus, it suffices to verify the claim for the case $n=2$.

Consider the diagrams $(D,f)$ and $(D',f')$ defined by:
\[
    (D,f) = \begin{ytableau}
    q \alpha \\
    & \none
    \end{ytableau},
    \qquad \qquad \text{and} \qquad \qquad
    (D',f') = \begin{ytableau}
    \none & \\
    & \none
    \end{ytableau}.
\]
Let $V$ denote the vertical strip and $H$ denote the horizontal strip. Using the expansion in \eqref{eq: modified macdonald to llt}, we have:
\[
    \widetilde{H}_{(D,f)}[X;q,t] = q\alpha \mathrm{LLT}_{(V)}[X;q] + \mathrm{LLT}_{(H)}[X;q].
\]
Similarly, applying \eqref{eq: modified macdonald to llt} followed by the equivalence from Proposition~\ref{prop: llt equivalence bandwidth 2}, we obtain:
\[
    \widetilde{H}_{(D',f')}[X;q,t] = \mathrm{LLT}_{(C_1,C_2)}[X;q] \equiv q \mathrm{LLT}_{(V)}[X;q] + \mathrm{LLT}_{(H)}[X;q].
\]
Combining these results, we compute:
\begin{align*}
    \widetilde{H}_{(D,f)} - \alpha \widetilde{H}_{(D',f')} &\equiv (q\alpha \mathrm{LLT}_{(V)} + \mathrm{LLT}_{(H)}) - \alpha(q \mathrm{LLT}_{(V)} + \mathrm{LLT}_{(H)}) \\
    &= \mathrm{LLT}_{(H)} - \alpha \mathrm{LLT}_{(H)} \\
    &= (1 - \alpha) \mathrm{LLT}_{(H)}.
\end{align*}
Dividing by $(1-\alpha)$, we conclude:
\[
    \frac{\widetilde{H}_{(D,f)}[X;q,t] - \alpha \widetilde{H}_{(D',f')}[X;q,t]}{1 - \alpha} \equiv \mathrm{LLT}_{(H)}[X;q]. \qedhere
\]
\end{proof}

\begin{proof}[Proof of Lemma~\ref{lem: s positivity reduction} for $m=2$]
As discussed in the proof for the $m=1$ case, it suffices to consider the restriction where $n=m+1=3$. Consider the filled diagrams $(D,f)$ and $(D',f')$ defined by:
\[
    (D,f) = \begin{ytableau}
    q\alpha \\
    \alpha\beta & \\
    & \none
    \end{ytableau},
    \qquad \qquad \text{and} \qquad \qquad
    (D',f') = \begin{ytableau}
    \none & \alpha\\
    \beta & \\
    &\none
    \end{ytableau}.
\]
Let $R_1 = R_{[3]}(\emptyset)$, $R_2 = R_{[3]}(\{3\})$, $R_3 = R_{[3]}(\{2\})$, and $R_4 = R_{[3]}(\{2,3\})$ denote the ribbon shapes. Using the expansion in \eqref{eq: modified macdonald to llt}, we obtain:
\begin{align*}
    \widetilde{H}_{(D,f)}[X;q,t] &= \mathrm{LLT}_{(R_1,C_2)} + q\alpha\mathrm{LLT}_{(R_2,C_2)} + \alpha\beta\mathrm{LLT}_{(R_3,C_2)} + q\alpha^2\beta\mathrm{LLT}_{(R_4,C_2)}, \\
    \widetilde{H}_{(D',f')}[X;q,t] &= \mathrm{LLT}_{(H_1,H_2)} + \alpha\mathrm{LLT}_{(H_1,V_2)} + \beta \mathrm{LLT}_{(V_1,H_2)} + \alpha\beta\mathrm{LLT}_{(V_1,V_2)}.
\end{align*}
We apply Proposition~\ref{Prop: LLT equivalences bandwidth 3} (1) and (2) to the terms $\mathrm{LLT}_{(R_2,C_2)}$ and $\mathrm{LLT}_{(R_3,C_2)}$ in the expansion of $\widetilde{H}_{(D,f)}$:
\begin{align*}
    \widetilde{H}_{(D,f)} &=\mathrm{LLT}_{(R_1,C_2)}
    + q\alpha \left( q\mathrm{LLT}_{(S)} + q^{-1}\mathrm{LLT}_{(H_1,V_2)} \right)
     + \alpha\beta \left( q\mathrm{LLT}_{(S)} + \mathrm{LLT}_{(V_1,H_2)} \right)
    + q\alpha^2\beta\mathrm{LLT}_{(R_4,C_2)} \\
    &\equiv \mathrm{LLT}_{(R_1,C_2)} + \alpha \mathrm{LLT}_{(H_1,V_2)} + \alpha\beta \mathrm{LLT}_{(V_1,H_2)}
    + (q^2\alpha + q\alpha\beta)\mathrm{LLT}_{(S)} + q\alpha^2\beta\mathrm{LLT}_{(R_4,C_2)}.
\end{align*}
Next, we apply Proposition~\ref{Prop: LLT equivalences bandwidth 3} (3) and (4) to the terms $\mathrm{LLT}_{(H_1,H_2)}$ and $\mathrm{LLT}_{(V_1,V_2)}$ in the expansion of $\widetilde{H}_{(D',f')}$:
\begin{align*}
    \widetilde{H}_{(D',f')} &\equiv
    \left( q^2\mathrm{LLT}_{(S)} + \mathrm{LLT}_{(R_1,C_2)} \right)
    + \alpha\mathrm{LLT}_{(H_1,V_2)} + \beta\mathrm{LLT}_{(V_1,H_2)}
    + \alpha\beta \left( q\mathrm{LLT}_{(S)} + q\mathrm{LLT}_{(R_4,C_2)} \right).
\end{align*}
Finally, we compute the difference quotient. Observing the cancellations, we have:
\[
    \widetilde{H}_{(D,f)} - \alpha \widetilde{H}_{(D',f')} \equiv (1-\alpha)\mathrm{LLT}_{(R_1,C_2)} + \alpha(1-\alpha)\mathrm{LLT}_{(H_1,V_2)} + q\alpha\beta(1-\alpha)\mathrm{LLT}_{(S)}.
\]
Dividing by $(1 - \alpha)$, we conclude that
\[
    \frac{\widetilde{H}_{(D,f)}[X;q,t] - \alpha \widetilde{H}_{(D',f')}[X;q,t]}{1 - \alpha} \equiv \mathrm{LLT}_{(R_1,C_2)} + \alpha \mathrm{LLT}_{(H_1,V_2)} + q\alpha\beta\mathrm{LLT}_{(S)}. \qedhere
\]
\end{proof}

\begin{rmk} One might hope to show Theorem~\ref{thm: second main, s-positivities} for the general case, i.e. $\nu/\lambda$ is in the $m$-th row, where $m\geq 3$ by showing Lemma~\ref{lem: s positivity reduction} for $m\geq 3$. However,  Lemma~\ref{lem: s positivity reduction} is no more true when $m\geq 3$. For example, let $\mu=(2,2,2)$ and $\lambda=(2,2,1,1)$ then $\lambda$ is obtained from $\mu$ moving a cell on the third row to the upper row. Then the coefficient of $t^5$ of $\I_{\lambda,\mu}[X;q,t]$ is $s_{(3,2,1)}[X] + qs_{(2,2,1,1)}[X]$ which cannot be written as a positive linear combination of LLT polynomials of a tuple of skew partitions where there is one, two, two, and one cells of content 1, 2, 3, and 4.
\end{rmk}

We conclude this section with a proof of Butler's conjecture for $t = 1$ (or $q = 1$).

\begin{corollary}\label{Cor: Schur positivity at t=1}
Let $\nu$ be a partition. Let $\lambda, \mu \subseteq \nu$ be two partitions such that $|\nu/\lambda| = |\nu/\mu| = 1$. Then $\I_{\lambda, \mu}[X; q, 1]$ (or $\I_{\lambda, \mu}[X; 1, t]$) is Schur-positive.
\end{corollary}

\begin{proof}
It is well known that the modified Macdonald polynomials factor at \( t = 1 \): for a partition \( \rho = (\rho_1, \rho_2, \dots) \), we have
\[
    \widetilde{H}_\rho[X; q, 1] = \prod_{i \ge 1} \widetilde{H}_{(\rho_i)}[X; q, 1]
\]
(see \cite{DM08}, for example). Suppose $\mu$ and $\lambda$ differ only at the $i$-th and $j$-th rows. Then we have
\[
    \I_{\lambda, \mu}[X; q, 1] = \prod_{k \neq i, j} \widetilde{H}_{(\mu_k)} \left(\dfrac{T_\lambda \widetilde{H}_{(\mu_i, \mu_j)}[X; q, 1] - T_\mu \widetilde{H}_{(\lambda_i, \lambda_j)}[X; q, 1]}{T_\lambda - T_\mu}\right),
\]
where we abuse the notation for $T_\lambda$ to denote its value at $t = 1$. 

By the above identity, to show Schur positivity of $\I_{\lambda, \mu}[X; q, 1]$, we only need to prove the claim for two-row partitions $\mu$ and $\lambda$, which follows from Theorem~\ref{thm: second main, s-positivities} (even for general $t$). 

By the $(q, t)$-symmetry of modified Macdonald polynomials:
\[
    \widetilde{H}_\mu[X; q, t] = \widetilde{H}_{\mu'}[X; t, q],
\]
the case for $q = 1$ follows.
\end{proof}

\section{Toward Combinatorial formulas for $(q,t)$-Kostka polynomials}\label{Sec: Schur expansions of Macdonald polynomials}

Butler's conjecture (or its partial result, Theorem~\ref{thm: second main, s-positivities}) suggests a method for predicting the Schur expansion of a modified Macdonald polynomial by relating it to another modified Macdonald polynomial whose Schur expansion is better understood. To explain this, let $\nu$ be a partition, and let $\mu, \lambda \subseteq \nu$ be two distinct partitions such that $\nu/\mu = (i-k,j+\ell)$ and $\nu/\lambda = (i,j)$. Suppose we are given a combinatorial formula for the $(q,t)$-Kostka polynomial $\widetilde{K}_{\mu,\rho}(q,t)$ for some partition $\rho$. Then, Butler's conjecture suggests that we may obtain $\widetilde{K}_{\lambda,\rho}(q,t)$ by modifying certain monomials of the form $t^a q^b$ appearing in $\widetilde{K}_{\mu,\rho}(q,t)$, specifically by replacing them with $t^{a+k} q^{b-\ell}$. %This section deals with a combinatorial formula for the Schur expansion of modified Macdonald polynomials that are consistent with this prediction suggested by Butler's conjecture.

% Any two-column partition can be obtained by successively moving the top cell in the first column to the second column, starting from the single-column shape, as illustrated at the top of Figure~\ref{fig:two column and hook}. This observation suggests the possibility of a combinatorial formula for the Schur expansion of the modified Macdonald polynomials indexed by two-column shapes that aligns with Butler's conjecture. In subsection~\ref{subsec: Two column}, we present such a formula for the 2-Schur expansion of modified Macdonald polynomials indexed by two-column partitions (Theorem~\ref{thm: combinatorial formula for two column}), motivated by the \emph{vertical dual Pieri rule} and \emph{shift invariance} from \cite{BMPS19}. While working on this paper, we discovered that our formula is equivalent to Zabrocki’s formula \cite{Zab98}.

Any hook shape partitions can be obtained from a single-column partition by successively moving the top cell in the first column to the first row, as illustrated in Figure~\ref{fig:two column and hook}. We review a Schur expansion formula for the modified Macdonald polynomials indexed by hooks, discovered by Assaf~\cite{Ass18}, and describe a Schur expansion of $\I_{\mu,\nu}$ for $\nu=(n-k+1,1^{k-1})$ and $\mu=(n-k,1^k)$ (Corollary~\ref{cor: hook}).

\begin{figure}[ht]
    \begin{displaymath}
    \tableau{\\ \ \\ \ \\ \ \\ \ \\ \ }
    \hspace{\cellsize} \raisebox{-4\cellsize}{$\rightarrow$} \hspace{\cellsize}
    \tableau{\\ \\ \ \\ \ \\ \ \\ \ & \ }
    \hspace{\cellsize} \raisebox{-4\cellsize}{$\rightarrow$} \hspace{\cellsize}
    \tableau{\\ \\ \\ \ \\ \ \\ \ & \ & \ }
    \hspace{\cellsize} \raisebox{-4\cellsize}{$\rightarrow$} \hspace{\cellsize}
    \tableau{\\ \\ \\ \\ \ \\ \ & \ & \ & \ }
    \hspace{\cellsize} \raisebox{-4\cellsize}{$\rightarrow$} \hspace{\cellsize}
    \tableau{\\ \\ \\ \\ \\ \ & \ & \ & \ & \ }
  \end{displaymath}
  \caption{Obtaining hook shape partitions from $(1^5)$.}\label{fig:two column and hook}
\end{figure}

For a permutation $w$ with $\iDes(w)=\{i_1,i_2,\dots\}$, the \emph{de-standardization} of $w$ is defined to be the word obtained by changing $1,2,\dots,i_1$ to $1$, $i_1+1,i_1+2,\dots,i_2$ to $2$, and so on. We denote $\dst(w)$ for the de-standardization of $w$. A permutation $w$
is called \emph{super-standard} if its de-standardization $\dst(w)$ is Yamanouchi, i.e. for every $k\in [n]$, the $\dst(w)_k \cdots \dst(w)_n$ has at least as many $i-1$'s as $i$'s for every $i$. For a super-standard word $w$, we say that $w$ has weight $\lambda$ if $\dst(w)$ is of weight $\lambda$. For example, $41523$ is super-standard of weight $(3,2)$ since its de-standardization
\[
\dst(41523) = 21211
\]
is Yamanouchi word of weight $(3,2)$. The set of super-standard words of weight $\lambda$ is denoted by $\SSS(\lambda)$.

For a word $w$ and a letter $x$ not in $w$, define a partitioning $\Gamma_x(w)$ as follows. Break before each index $i$ such that
\[
\begin{cases}
w_i < x &\text{ if } w_1 < x;\\
w_i > x &\text{ if } w_1 > x.
\end{cases}
\]
The map $\gamma_x$ is defined by sending the first letter in each block of $\Gamma_x$ to the end of the block \cite{DF68}. For example, \[
\Gamma_4(5613782) = |5|613|7|82, \text{ and } \gamma_4(5613782)=5136728.
\]
For $k\in[n]$, Assaf defined the bijections $\mathcal{A}_k$ on $\mathfrak{S}_n$ by 
\[
\mathcal{A}_k(w_1 \cdots w_n) = w_1 \cdots w_k \gamma_{w_k}(w_{k+1}\cdots w_n).
\]
Using this map, Assaf gave a combinatorial formula for the Schur expansion of the modified Macdonald polynomials indexed by hook partitions: 

\begin{thm}\cite[Corollary 4.5]{Ass18}\label{thm: Assaf}
For $\mu=(n-k,1^k)$ a hook partition, set $\mathcal{A}_\mu:=\mathcal{A}_{k+1}\cdots\mathcal{A}_{n-1}$. Then we have
\[
\widetilde{H}_\mu[X;q,t] = \sum_{\lambda\vdash n}\left(\sum_{w\in \SSS(\lambda)}q^{\inv_{(\mu,{f^{\st}_{\mu}})}(\mathcal{A}_\mu(w))}t^{\maj_{(\mu,{f^{\st}_{\mu}})}(\mathcal{A}_\mu(w))}\right)s_\lambda[X].
\]
\end{thm}
We reformulate Theorem~\ref{thm: Assaf} in terms of the usual major statistic for words, 
i.e.\ $\maj(w):=\sum_{i\in \Des(w)} i$ (Lemma~\ref{lem: reformulate assaf}). 
Recall also the usual inversion statistic for words, 
i.e.\ $\inv(w):=|\{(i,j): w_i > w_j\}|$. 
Note that both $\maj$ and $\inv$ should not be confused with 
the statistics $\maj_{(D,f)}$ and $\inv_{(D,f)}$ in Definition~\ref{def: stat definition}. 
For a hook partition $\mu = (n-k, 1^k)$, it is straightforward to see that
\begin{equation}\label{eq:xx}
    \maj_{(\mu,f_{\mu}^{\st})}(w) = t^{\maj(w_{[k+1]})},
    \qquad
    \inv_{(\mu,f_{\mu}^{\st})}(w) = q^{\inv(w_{[k+1,n]})}.
\end{equation}

\begin{lem}\label{lem: reformulate assaf}
For $\mu=(n-k,1^k)$, we have
\begin{equation*}
    q^{\inv_{(\mu,{f^{\st}_{\mu}})}(\mathcal{A}_\mu(w))}t^{\maj_{(\mu,{f^{\st}_{\mu}})}(\mathcal{A}_\mu(w))}=t^{\maj(w)}\prod_{i\in \Des(w)\cap[k+1,n-1]}\frac{q^{n-i}}{t^i}
\end{equation*}
for any word $w$.
\end{lem}
\begin{proof}

We proceed by descending induction on $k$. The base case $k = n-1$ follows immediately from \eqref{eq:xx}. Now assume the claim holds for $\mu = (n-k,1^{k})$, and consider the case $\nu = (n-k+1,1^{\,k-1})$. We will establish the following property of $\mathcal{A}_k$: For any word $w$ and denoting $u=\mathcal{A}_k(w)$,
\begin{enumerate}
    \item if $w_k > w_{k+1}$ then we have $\begin{cases}
    \maj(w_{[k+1]}) = \maj(u_{[k]}) + k \\ 
    \inv(w_{[k+1,n]}) = \inv(u_{[k,n]}) - (n-k),
    \end{cases}$ \quad and
    \item if $w_k < w_{k+1}$ then we have $\begin{cases}
    \maj(w_{[k+1]}) = \maj(u_{[k]}) \\ 
    \inv(w_{[k+1,n]}) = \inv(u_{[k,n]}).
    \end{cases}$
\end{enumerate}

Suppose $w_k>w_{k+1}$ then $\maj(w_{[k+1]}) = \maj(u_{[k]}) + k$ follows as $w_{[k]}=u_{[k]}$.  Let $I$ be the set of indices of the first letter of each block of $\Gamma_{w_k}(w_{k+1}\cdots w_n)$, i.e. $I=\{k+1\le i\le n: w_k>w_i\}$. Say $I=\{k+1=i_1<\cdots<i_\ell\}$. Then \[
\Gamma_{w_k}(w_{k+1}\cdots w_n) = |w_{k+1}\cdots w_{i_2-1}|w_{i_2}\cdots w_{i_3 -1}| \cdots |w_{i_\ell}\cdots w_n,
\]
and 
\begin{align*}
    u=\mathcal{A}_k(w)&= w_1\cdots w_k \gamma_{w_k}(w_{k+1}\cdots w_n)\\
    &=w_1\cdots w_k w_{k+2} \cdots w_{i_2-1}w_{k+1} w_{i_2+1} \cdots w_{i_3-1}w_{i_2} \cdots w_{i_\ell+1} \cdots w_n w_{i_\ell}.
\end{align*}
For each $i\in I$, $(w_k,w_i)$ forms an inversion pair in $u_{[k,n]}$ These give $\ell$ inversion pairs. In addition, there are $n-k-\ell$ inversion pairs of the form $(x,w_i)$ in $u_{[k,n]}$, where $i\in I$ and $x$ and $w_i$ are in the same block of $\Gamma_{w_k}(w_{k+1}\cdots w_n)$. Except for these $n-k$ inversion pairs, the others are the same in $w_{[k+1,n]}$ and $u_{[k,n]}$. This proves $\inv(w_{[k+1,n]}) = \inv(u_{[k,n]}) - (n-k)$.

Now suppose $w_k < w_{k+1}$, again $\maj(w_{[k+1]}) = \maj(u_{[k]})$ follows. Let $I=\{i_1<\cdots<i_\ell\}$ be the set of indices of the first letter of each block of $\Gamma_{w_k}(w_{k+1}\cdots w_n)$, i.e. $I=\{k+1\le i\le n: w_k < w_i\}$.  Then again we have 
\begin{align*}
    u=\mathcal{A}_k(w)&= w_1\cdots w_k \gamma_{w_k}(w_{k+1}\cdots w_n)\\
    &=w_1\cdots w_k w_{k+2} \cdots w_{i_2-1}w_{k+1} w_{i_2+1} \cdots w_{i_3-1}w_{i_2} \cdots w_{i_\ell+1} \cdots w_n w_{i_\ell}.
\end{align*}
For each $j\notin I$, $(w_k,w_j)$ forms an inversion pair in $u_{[k,n]}$. These contribute $n-k-\ell$ inversion pairs. On the other hand, there are also $n-k-\ell$ inversion pairs in $w_{[k+1,n]}$ of the form $(x,w_i)$, where $i\in I$ and $x$ and $w_i$ are in the same block of $\Gamma_{w_k}(w_{k+1}\cdots w_n)$. These pairs are no longer inversions in $u_{[k,n]}$. This proves $\inv(w_{[k+1,n]}) = \inv(u_{[k,n]})$.

For any word $w$ and denoting $u=\mathcal{A}_{\mu}(w)$, we have $w_k<w_{k+1}$ if and only if $u_k<u_{k+1}$ as $w_{[k+1]}=u_{[k+1]}$. This shows the desired claim for $\nu=(n-k+1,1^{k-1})$.
\end{proof}
\begin{corollary}\label{cor: hook}
    For hook partitions $\mu=(n-k,1^k)$ and $\nu=(n-k+1,1^{k-1})$  we have
\[
\I_{\mu,\nu}[X;q,t] = \sum_{\lambda\vdash n}\sum_{\substack{w\in \SSS(\lambda)\\ w_k<w_{k+1}}}\left(t^{\maj(w)}\prod_{i\in \Des(w)\cap[k+1,n-1]}\frac{q^{n-i}}{t^i}\right)s_\lambda[X].
\]
\end{corollary}
\begin{proof}
By Lemma~\ref{lem: reformulate assaf} we have
\begin{align*}
    \widetilde{H}_\mu[X;q,t] &= \sum_{\lambda\vdash n}\sum_{w\in \SSS(\lambda)}\left(t^{\maj(w)}\prod_{i\in \Des(w)\cap[k+1,n-1]}\frac{q^{n-i}}{t^i}\right)s_\lambda[X]\\
  \widetilde{H}_\nu[X;q,t] &= \sum_{\lambda\vdash n}\sum_{w\in \SSS(\lambda)}\left(t^{\maj(w)}\prod_{i\in \Des(w)\cap[k,n-1]}\frac{q^{n-i}}{t^i}\right)s_\lambda[X].
\end{align*}
Now the proof follows as $\frac{T_{\nu}}{T_{\mu}}=\frac{q^{n-k}}{t^k}$.
\end{proof}

As an illustration, Table~\ref{table: hooks of length 4} summarizes how the $(q,t)$-statistics change step by step as we move cells from $(1^4)$ to obtain all hook shapes of size $4$.
Dots on the letters of each word $w \in \SSS(\lambda)$ indicate the descent positions. 
\begin{table}[ht]
    \begin{tabular}{|cc|cc|cc|cc|c|}\hline
    $\lambda$& $w\in \SSS(\lambda)$ & $(1,1,1,1)$ && $(2,1,1)$ && $(3,1)$ && $(4)$ \\\hline
    
    (4)&$1234$ & 1 && 1 && 1 && 1 \\ \hline

    $(3,1)$&$12\dot{4}3$ & $t^3$ & $\rightarrow$ & $q$ && $q$ && $q$ \\
    &$1\dot{4}23$ & $t^2$ && $t^2$ & $\rightarrow$ & $q^2$ && $q^2$ \\
    &$\dot{4}123$ & $t$ && $t$ && $t$ & $\rightarrow$ & $q^3$\\ \hline

    $(2,2)$ &$\dot{3}1\dot{4}2$ & $t^4$ & $\rightarrow$ & $qt$ && $qt$ & $\rightarrow$ & $q^4$ \\ 
    &$3\dot{4}12$ & $t^2$ && $t^2$ & $\rightarrow$ & $q^2$ &&   $q^2$ \\ \hline
    
    $(2,1,1)$ &$1\dot{4}\dot{3}2$ & $t^5$ &$\rightarrow$& $qt^2$ &$\rightarrow$& $q^3$ && $q^3$ \\ 
    &$\dot{4}1\dot{3}2$ & $t^4$ &$\rightarrow$& $qt$ &&$qt$   &$\rightarrow$& $q^4$ \\ 
    &$\dot{4}\dot{3}12$ & $t^3$ && $t^3$ &$\rightarrow$& $q^2t$ &$\rightarrow$& $q^5$ \\ \hline
    
    $(1,1,1,1)$&$\dot{4}\dot{3}\dot{2}1$ & $t^6$ &$\rightarrow$& $qt^3$ &$\rightarrow$& $q^3t$ &$\rightarrow$& $q^6$ \\ \hline
    \end{tabular}
    \caption{Table for $\widetilde{K}_{\lambda,\mu}(q,t)$ for hook partitions $\mu$.}
    \label{table: hooks of length 4}
\end{table} 

Substituting $q = t = 1$ into Corollary~\ref{cor: hook}, and using 
Corollary~\ref{Cor: h21111 at q=t=1}, we obtain
\begin{equation}\label{eq: dfdfss}
   h_{(2,1^{\,n-2})}
   = \sum_{\lambda \vdash n}
     \bigl|\{\, w \in \SSS(\lambda) : w_k < w_{k+1} \,\}\bigr|
     \, s_\lambda[X].
\end{equation}
We now give a direct explanation of \eqref{eq: dfdfss}, independent of 
Corollary~\ref{Cor: h21111 at q=t=1}.  
For each $w \in \SSS(\lambda)$, associate a standard Young tableau of shape 
$\lambda$ as follows.  
Consider $\dst(w)$, and fill the shape $\lambda$ so that the entries in the 
$i$-th row are precisely those $j$ for which $\dst(w)_{\,n+1-j} = i$.

Under this correspondence, the condition $w_k < w_{k+1}$ holds if and only if 
$n-k+1$ appears to the right of $n-k$ in the associated standard Young tableau.  
Therefore, the right-hand side of \eqref{eq: dfdfss} becomes
\begin{equation}\label{10}
\sum_{\lambda \vdash n}
    \bigl|\{\, T \in \syt(\lambda) : \text{$n-k+1$ is to the right of $n-k$} \,\}\bigr|
    \, s_{\lambda}[X].
\end{equation}
For each $T \in \syt(\lambda)$ such that $n-k+1$ is to the right of $n-k$, we may identify it with a sequence
\[
\emptyset=\nu^{(0)}\subset \nu^{(1)}\subset \cdots \subset \nu^{(n-1)}=\lambda
\]
of partitions, where the skew shape $\nu^{(i+1)}/\nu^{(i)}$ is a horizontal strip of length $2$ if $i=n-k-1$, and a single cell otherwise.
The correspondence is defined as follows: for $i< n-k$, let $\nu^{(i)}$ be the subshape of $T$ occupied by the entries $1,\dots,i$, while for $i\ge n-k$, let $\nu^{(i)}$ be the subshape occupied by the entries $1,\dots,i+1$.
Thus, \eqref{10} equals $h_{(2,1^{n-2})}$ by Pieri’s rule \cite[Theorem 7.15.7]{StanleyEC2}.

\section{Future questions}

Various questions naturally arise from this work.

\begin{question}
    Construct a dual equivalence graph on Butler permutations to prove Butler's conjecture.
\end{question}

Assaf introduced \emph{dual equivalence graphs} as a universal method for establishing Schur positivity of a quasisymmetric function \cite{Ass15}. It is our hope that by examining $\frac{n!}{2}$ permutations, rather than the whole $n!$ permutations, we can find a path to construct the desired dual equivalence graph. Constructing dual equivalence graphs on Butler permutations would then provide a proof of Schur positivity for $\I_{\lambda, \mu}[X; q, t]$, as stated in Theorem~\ref{thm: first main, F-expansion}.

\begin{question}
    Let $\nu$ be a partition and let $\mu^{(1)}, \dots, \mu^{(k)} \subseteq \nu$ be $k$ distinct partitions such that $|\nu/\mu^{(1)}| = \cdots = |\nu/\mu^{(k)}| = 1$. Find a positive $F$-expansion formula for the Macdonald intersection polynomial $\I_{\mu^{(1)}, \dots, \mu^{(k)}}[X; q, t]$, for $k \geq 3$.
\end{question}

In the companion paper \cite{KLO22+}, we show
\begin{equation}\label{eq: sf to sf}
    \dfrac{1}{T_{\bigcap_{i=1}^k \mu^{(i)}}}e_{n-k}^\perp\left(\I_{\mu^{(1)}, \dots, \mu^{(k)}}[X; q, t]\right) = D_{k-1}[X; q, t],
\end{equation}
where
\[
   D_n[X; q, t] := \sum_{\lambda \in \delta_n} \sum_{T \in \syt(\lambda + (1^n)/\lambda)} t^{|\delta_n/\lambda|} q^{\dinv(T)} F_{\dd(T)}[X].
\]
Here $\delta_n=(n-1,\dots,1)$ is a staircase partition and $D_n[X;q,t]$
is the $F$-expansion formula for $\nabla e_n$ by the Shuffle theorem \cite{CM18}. Equation \eqref{eq: sf to sf} shows that the Shuffle formula $D_{k-1}[X; q, t]$ arises in certain $F$-coefficients of the Macdonald intersection polynomial $\I_{\mu^{(1)}, \dots, \mu^{(k)}}[X; q, t]$, suggesting that finding a full $F$-expansion is a challenging problem.

\section*{acknowledgement}
We are grateful to François Bergeron for bringing the Science Fiction conjecture to our attention. We also thank James Haglund and Brendon Rhoades for helpful conversations, and Foster Tom for sharing his SAGE code for computing LLT polynomials. We thank the anonymous reviewer for their careful reading and valuable suggestions, which significantly improved the paper.  D. Kim was supported by the National Research Foundation of Korea (NRF) grant funded by the Korean government (MEST) (No. 2019R1A6A1A10073437). S. J. Lee was supported by the National Research Foundation of Korea (NRF) grant funded by the Korean government (MSIT) (No.0450-20240021). J. Oh was supported by NRF grant RS-2025-16067413 and a KIAS Individual Grant (CG083401, HP083401) at the Korea Institute for Advanced Study.

\printbibliography

\end{document}